\documentclass[a4paper,10pt]{article}

\usepackage[utf8]{inputenc}

\usepackage{amssymb}
\usepackage{amsthm}
\usepackage{amsmath}
\usepackage{anysize}
\usepackage{url}
\usepackage{color}
\usepackage{subfigure,graphicx}
\usepackage{overpic}

\newcommand{\bitem}{\begin{itemize}}
\newcommand{\eitem}{\end{itemize}}
\newcommand{\beq}{\begin{equation}}
\newcommand{\eeq}{\end{equation}}
\newcommand{\beqn}{\begin{eqnarray*}}
\newcommand{\eeqn}{\end{eqnarray*}}

\newcommand{\cE}{{\cal E}}

\newcommand{\cF}{\mathfrak{F}}

\newcommand{\argmin}{\mbox{argmin}}

\newcommand{\Hol}{\text{\sl Höl}}

\def\supp{{\text{\rm supp }}}

\newcommand{\spa}{{\text{\rm span}}}   

\def\cH{\mathcal{H}}

\def\N{\mathbb{N}}

\def\Z{\mathbb{Z}}

\def\R{\mathbb{R}}
\def\RR{\mathbb{R}}

\def\epsilon{\varepsilon}

\newtheorem{theorem}{Theorem}[section]
\newtheorem{lemma}[theorem]{Lemma}
\newtheorem{cor}[theorem]{Corollary}
\newtheorem{definition}[theorem]{Definition}

\newtheorem{prop}[theorem]{Proposition}

\newcommand{\ms}[1]{{\color{black}{#1}}}

\author{Philipp Grohs, Sandra Keiper, Gitta Kutyniok, Martin Sch\"{a}fer
\footnote{
PG was supported in part by Swiss National Fund (SNF) Grant 146356.
SK acknowledges support from the Berlin Mathematical School and the DFG
Collaborative Research Center TRR 109 "Discretization in Geometry and
Dynamics. GK was supported in part by the Einstein Foundation Berlin, by the
Einstein Center for Mathematics Berlin (ECMath), by Deutsche
Forschungsgemeinschaft (DFG) Grant KU 1446/14, by the DFG Collaborative
Research Center TRR 109 "Discretization in Geometry and Dynamics", and by
the DFG Research Center {\sc Matheon} "Mathematics for key technologies"
in Berlin. }}
\title{Cartoon Approximation with $\alpha$-Curvelets\\
}
\begin{document}
\maketitle

\begin{abstract}
It is well-known that curvelets provide optimal approximations for so-called cartoon images
which are defined as piecewise $C^2$-functions, separated by a $C^2$ singularity curve. In this
paper, we consider the more general case of piecewise $C^\beta$-functions, separated by a $C^\beta$
singularity curve for $\beta \in (1,2]$. We first prove a benchmark result for the possibly achievable
best $N$-term approximation rate for this more general signal model. Then we introduce what we call
$\alpha$-curvelets, which are systems that interpolate between wavelet systems on the one hand
($\alpha = 1$) and curvelet systems on the other hand ($\alpha = \frac12$). Our main result
states that those frames achieve this optimal rate for $\alpha = \frac{1}{\beta}$, up to $\log$-factors.
\end{abstract}

\section{Introduction}
In various applications in signal processing it has proven useful to decompose a given signal
in a multiscale dictionary, for instance to achieve compression by coefficient thresholding or to
solve inverse problems. The most popular family of such dictionaries are undoubtedly wavelets, which
have had a tremendous impact in applied mathematics since Daubechies' construction of orthonormal
wavelet bases with compact support in the 1980s. While wavelets are now a well-established tool in
numerical signal processing (for instance the JPEG2000 coding standard is based on a wavelet transform),
it has been recognized in the past decades that they also possess several shortcomings, in particular
with respect to the treatment of multidimensional data where anisotropic structures such as edges in
images are typically present. This deficiency of wavelets has given birth to the research area of
geometric multiscale analysis where frame constructions which are optimally adapted to anisotropic
structures are sought.

\subsection{Geometric Multiscale Analysis}

A milestone in this area has been the construction of curvelet and shearlet frames which are indeed
capable of optimally resolving curved singularities in multidimensional data. To be more precise,
the landmark paper \cite{CD04} has shown that simple coefficient thresholding in a curvelet frame
yields an, up to (negligible) $\log$-factors, optimal $N$-term approximation rate for the class of
so-called \emph{cartoon images} which are roughly defined as compactly supported, bivariate, piecewise
$C^2$ functions, separated by a $C^2$ discontinuity curve. After this breakthrough result new constructions
of anisotropic frame systems have been introduced which achieve the same optimal approximation rates for
cartoon images. Among those we would like to single out \emph{shearlets} which are better suited for
digital implementation than the original curvelets. They were first introduced in \cite{GKL05} and a
comprehensive summary of their properties can be found in the survey \cite{KL12i}. The recent work
\cite{Grohs2011} introduced the framework of \emph{parabolic molecules} which subsumes all the earlier
constructions mentioned above and which established a transference principle of approximation results
between any two systems of parabolic molecules, provided that one of them constitutes a tight frame for
$L^2(\mathbb{R}^2)$.

\subsection{General Image Models}

By now the model of such cartoon-images is widely recognized as a useful mathematical model for image data.
However, for certain applications it might appear still too restrictive. For instance one could imagine
signals which consist of piecewise smooth parts separated by a singularity curve which might not necessarily
be exactly $C^2$ but of higher or lower smoothness.

The approximation of such a generalized signal class by hybrid shearlets has been studied in \cite{Kutyniok2012}
for three-dimensional data. In that paper the best $N$-term approximation results which could be established
are suboptimal by a small margin. In addition, hybrid shearlets are compactly supported, but do not form a tight
frame.

\subsection{Our Contribution}

In the present paper we study optimal approximation schemes for a more general model,  namely bivariate
piecewise $C^\beta$ functions, separated by a $C^\beta$ discontinuity curve, where $\beta \in (1,2]$.
We establish a benchmark result which states that one cannot in general
achieve an $N$-term approximation rate higher than $\frac{\beta}{2}$ for such functions. We then introduce
the notion of $\alpha$-curvelet frames which generalize the construction of \cite{CD04}, where
$\alpha\in [0,1]$ describes the degree of anisotropy in the scaling operation used for the frame construction.
The parameter $\alpha=0$ corresponds to ridgelets \cite{Can98}, the case $\alpha = \frac12$ to second-generation
curvelets as in \cite{CD04}, and the case $\alpha=1$ to wavelets. $\alpha$-curvelets form a tight frame for
any parameter $\alpha$ in contrast to the systems considered in \cite{Kutyniok2012}. Our main result then
states that all intermediate cases $\alpha\in [\frac12,1]$ provide optimal (up to $\log$-factors) $N$-term
approximation rates for the above mentioned signal class with $\beta = \frac{1}{\alpha}$. This is particularly
surprising, since for the result in the 3D situation (\cite{Kutyniok2012}) for intermediate parameters
only a suboptimal rate could be proven. Our proof techniques rely somehow on \cite{CD04}, however
we wish to mention that the technical details are actually quite different from those used there. In particular,
we have to deal with smoothness spaces of fractional orders which forces us to replace estimates for derivatives
in \cite{CD04} by estimates on certain moduli of smoothness.

Having such an approximation result for a tight frame system at hand will allow us to transform our approximation
result to arbitrary systems of frames of \emph{$\alpha$-molecules} (including compactly supported hybrid shearlets
as special case), as first introduced in \cite{alphaSPIE}. This will be the subject of the forthcoming work
\cite{GKKS13}, generalizing the framework of parabolic molecules \cite{Grohs2011}.

\subsection{Outline}

After introducing the general class of image models in Section \ref{sec:cartoon} and deriving the benchmark
result in Theorem \ref{thm:upperbound}, Section \ref{ssec:HyCurve} is devoted to the construction of
$\alpha$-curvelets. Our main result on the optimal approximation properties of $\alpha$-curvelets for
any parameter $\alpha\in[\frac{1}{2},1]$ is then stated in the next section as Theorem \ref{thm:mainappr1}. This section
as well as the appendix do contain the lengthy, quite technical proof of this result.

\ms{\subsection{Preliminaries}

Let us fix some notational conventions used throughout this work.
For $x\in\R^d$ we denote the Euclidean distance by $|x|_2$ and the $\ell_1$-norm by $|x|$. Further,
we put $x^m=x_1^{m_1}\cdots x_d^{m_d}$ for $m\in\N_0^d$.

For a function $f:\R^d\rightarrow\R$ the forward difference operator $\Delta_{(h_1,\ldots,h_d)}$, where $h_1,\ldots,h_d\in\R$, is given by
\[
\Delta_{(h_1,\ldots,h_d)} f(x_1,\ldots,x_d):=f(x_1+h_1,\ldots,x_d+h_d) - f(x_1,\ldots,x_d).
\]
In the one-dimensional case it takes the simple form $\Delta_h f(t):=f(t+h) - f(t)$.
We will often
apply the forward difference operator to a bivariate
function $f:\R^2\to \R$.
To simplify notation, the operator $\Delta_h$ shall exlusively act on variables
denoted $t$ or $\tau \in \R$, e.g., the symbol $\Delta_h f(t,u)$ denotes the function
$(t,u)\mapsto f(t+h,u) - f(t,u)$.

We denote by
$\partial_i$ the derivative in the $i$-th coordinate
direction, $i\in \{1,\dots, d\}$, and we let $\partial^{m}=\partial_1^{m_1}\cdots \partial_d^{m_d}$
for $m\in\N_0^d$.

We write $\|\cdot\|_p$ for the respective norms in the spaces $L^p(\R^d)$ and $\ell^p(\Lambda)$, where $\Lambda$ is
a discrete set.

For $f\in L^1(\R^d)$, $d\in\N$, the Fourier transform $\mathcal{F}f$ is defined by
\[
\mathcal{F}f(\xi)= \int_{\R^d} f(x) \exp(-2\pi i x\cdot\xi) \,dx.
\]
The short-hand notation for $\mathcal{F}f$ is $\widehat{f}$.

We shall also make extensive use of
the Radon transform $\mathcal{R}f$ given by
\[
\mathcal{R}F(t,\eta):=\int_{\mathfrak{L}_{t,\eta}} F(x)\,dl(x),
\]
with $(t,\eta)\in \R\times (-\pi/2,\pi/2]$ and $\mathfrak{L}_{t,\eta}:=\left\{(x_1,x_2)\in \R^2:\, \sin(\eta)x_1 + \cos(\eta)x_2 = t\right\}$.
%

The Hölder spaces $C^\gamma$ for $\gamma \in (1,2]$ are defined as follows:
\[
	C^{\gamma}(\R^d):=
	\left\{f\in C(\R^d):\, \max_{i=1}^d
	\text{\sl Höl}(\partial_i f,\gamma -1 ) <\infty
	\right\}
\]
where for a function $f:\R^d\rightarrow\R$ we we use the notation
\[
\text{\sl Höl}(f,\alpha)=\sup_{x,y\in \R^d} \frac{|f(x)-f(y)|}{|x-y|^{\alpha}},\quad\alpha\in[0,1],
\]
for the H\"older constant. 
The notation $C_0^\gamma$ refers to compactly supported
functions in $C^\gamma$.

}

In the sequel, for two quantities $A,B\in \R$, which may depend on several parameters we
shall write $A\lesssim B$, if there exists a constant $C>0$ such that $A\le C B$, uniformly
in the parameters. If the converse inequality holds true, we write $A\gtrsim B$ and if both
inequalities hold we shall write $A \asymp B$.

\section{Cartoon Approximation}
\label{sec:cartoon}

Many applications require efficient encoding of multivariate data in the sense of optimal
(sparse) approximation rates by a suitable representation system. This is typically phrased
as a problem of best $N$-term approximation as explained in Subsection \ref{ssec:sparseapprox}. The
performance of an approximation scheme is then analyzed with respect to certain subclasses
of the Hilbert space $L^2(\R^2)$, which is the standard continuum domain model for $2$-dimensional
data, in particular in imaging science. The key feature of most
multivariate data is the appearance of anisotropic phenomena. Hence such a subclass of
$L^2(\R^2)$ is required to provide a suitable model for this fact, which is fulfilled by
the subclass of so-called cartoon-like images as defined in Subsection \ref{ssec:imagedata}.
The main result of this section, Theorem~\ref{thm:benchmark}, will provide \ms{an upper bound for the maximal achievable approximation rate for this class of cartoon images.
It serves as a benchmark for the approximation performance of concrete representation systems}.


\subsection{Sparse Approximation}
\label{ssec:sparseapprox}

Let us start with a short introduction to some aspects of approximation theory.
The standard continuum model for $2$-dimensional data, in particular in imaging science, is
the Hilbert space $L^2(\R^2)$. From a practical standpoint however,
a function $f\in L^2(\R^2)$ is a rather intractable object. Therefore, in order to analyze $f$, the most
common approach is to represent it with respect to some representation system
$\Phi=(\varphi_\lambda)_{\lambda\in \Lambda}\subset L^2(\R^2)$, i.e., to expand $f$ as
\beq \label{eq:decomp}
     f = \sum_{\lambda\in \Lambda} c_\lambda \varphi_\lambda,
\eeq
and then consider the coefficients $c_\lambda\in\R$.

Moreover, since in real world applications infinitely many coefficients
are infeasible, the function $f$ has to be approximated by a finite subset of this system. Letting
$N$ be the number of elements allowed in this approximation, we obtain what is called an
\emph{$N$-term approximation} for $f$ with respect to $\Phi$. The {\em best $N$-term
approximation}, usually denoted by $f_N$, is optimal among those in terms of a minimal
approximation error and is defined by
\[
    f_N = \mathop{\argmin}_{g=\sum_{\lambda\in \Lambda_N}c_\lambda\varphi_\lambda}\, \|f -  g\|_2^2
    \quad \mbox{s.t.}\quad \#\Lambda_N \le N.
\]
The approximation performance of a system $\Phi$ for a subclass
$\cF\subset L^2(\R^2)$ is typically measured by the {\em asymptotic approximation rate}, i.e., the
decay of the $L^2$-error of the best $N$-term approximation $\|f - f_N\|_2$ as $N \to \infty$.

In a very general form the representation system $\Phi$ can just be an arbitrary \emph{dictionary}, i.e. we require only $L^2(\R^2)=\overline{\spa\,\Phi}$.
General dictionaries $\Phi$ can be rather pathological, and for $f\in L^2(\R^2)$ there might not even exist a best $N$-term approximation.
For example, since $L^2(\R^2)$ is separable, we can choose $\Phi$ to be a countable dense subset of $L^2(\R^2)$.
This choice would yield arbitrarily good $1$-term approximations since the elements of $\Phi$ can come arbitrarily close to $f$.

Therefore, without reasonable restrictions on the search depth, the investigation of best $N$-term approximation with respect to a given
dictionary may not even make sense.
A commonly used and suitable restriction is to impose \emph{polynomial depth search}, which requires that the terms of the $N$-term approximation have to come from the
first $\pi(N)$ elements of the dictionary, where $\pi$ is some fixed polynomial \cite{Don01}.

\subsubsection{Polynomial Depth Search in a Dictionary}

Let us get more specific and assume that we have a countable dictionary $\Phi=(\varphi_n)_{n\in\N}$ indexed by the natural numbers, and a fixed polynomial
$\pi$ specifying the search depth.
A non-linear $N$-term approximation scheme can then be described by a set-valued selection function $\mathcal{S}$, which determines for given $f\in L^2(\R^2)$ and $N\in\N$ the
selected dictionary elements, i.e. $\mathcal{S}(f,N)\subset\Phi$ with $\# \mathcal{S}(f,N)=N$.
The function $\mathcal{S}$ shall obey the polynomial depth search constraint, i.e. $\mathcal{S}(f,N)\subset \{ \varphi_1, \ldots, \varphi_{\pi(N)} \}$, and the dependence on $f$
allows for adaptive approximation schemes.

Let us recall a benchmark concerning the optimality rate of approximation in a dictionary with polynomial depth search,
as derived in \cite{Don01}. Before, we have to define what it means for a class of functions to contain a copy of $\ell_0^p$.

\begin{definition}
\begin{enumerate}
\item

A class of functions $\cF\subset L^2(\R^2)$ is said to \emph{contain an embedded orthogonal hypercube of dimension $m$ and sidelength $\delta$} if there exist
$f_0 \in \cF$ and orthogonal functions $\psi_{i}\in L^2(\R^2)$ for $i=1,...,m$ with $\|\psi_{i}\|_{2}=\delta$ such that the collection of hypercube vertices
\begin{align*}
\cH(m;f_0,(\psi_{i})_i)= \Big{\{ }h=f_0 + \sum_{i=1}^{m} \epsilon_i\psi_{i} ~:~ \epsilon_i \in \{0,1\} \Big{\}}
\end{align*}
is contained in $\cF$. It should be noted that $\cH$ just consists of its vertices.
\item
A class of functions $\cF$ is said to \emph{contain a copy of $\ell_{0}^{p}$, $p>0$,} if there exists a sequence of orthogonal hypercubes $(\cH_k)_{k\in\N}$,
embedded in $\cF$, which have dimensions $m_k$ and side-lengths $\delta_k$, such that $\delta_k \rightarrow 0$ and for some constant $C>0$
\begin{equation}\label{eqn:locopy}
m_k\ge C\delta_k^{-p} \quad\text{for all }k \in\N.
\end{equation}
\end{enumerate}
\end{definition}

Note, that if $\mathfrak{F}$ contains a copy of $\ell_0^p$, then it also contains a copy of $\ell_0^{q}$ for all $0<q< p$.
It was shown in \cite{Don01} that if a function class $\mathfrak{F}$ contains a copy of $\ell_0^p$ there exists an upper bound on the maximal achievable approximation rate
via reconstruction in a fixed dictionary.

We state a reformulation of this landmark result, which in its original form \cite{Don01} was stated in terms of the coefficient decay.
The original proof however can be adapted to lead to the following formulation in terms of the best $N$-term approximation, which is more appropriate for our needs.
\ms{We remark, that this theorem is also valid in a general Hilbert space setting.}

\begin{theorem}\label{thm:upperbound}
Suppose, that a class of functions $\cF\subset L^2(\R^2)$ is uniformly $L^2$-bounded and contains a copy of $\ell^p_0$.
Then, allowing only polynomial depth search in a given dictionary, there is a constant $C>0$ such that for
every $N_0\in\N$ there is a function $f\in\cF$ and an $N\in\N$, $N\ge N_0$ such that
\begin{equation*}
\|f - f_{N} \|^2_2 \ge C \big(N \log_2 (N)\big)^{-(2-p)/p},
\end{equation*} 
where $f_N$ denotes the best $N$-term approximation under the polynomial depth search constraint.
\end{theorem}
\begin{proof}
Let $\Phi=(\varphi_n)_{n\in\N}$ be a given dictionary and $\pi$ the polynomial, specifying the search depth.
The best $N$-term approximation of $f\in L^2(\R^2)$ obtained in this setting shall be denoted by $f_N$,
the corresponding selection rule, as described above, by $\mathcal{S}$.

Each system $\mathcal{S}(f,N)$ can be orthonormalized by the Gram-Schmidt procedure (starting from lower indices to higher indices), giving rise to an orthonormal basis
of $\spa\, \mathcal{S}(f,N)$ (with the exception of some possible zero vectors).
Therefore we can represent each $f_N$ by the unique set of coefficients obtained from an expansion in this basis.
(If a basis element is zero, the corresponding coefficient is chosen to be zero.)

In order to apply information theoretic arguments, we consider the following coding procedure.
For $f\in \cF$ we select the dictionary elements $\mathcal{S}(f,N)$ and quantize the coefficients
of $f_N$ obtained as above by rounding to multiples of the quantity $\mathfrak{q}=N^{-2/p}$.

We need $N\log_2(\pi(N))$ bits of information to encode the locations of the selected elements $\mathcal{S}(f,N)$ and
$N \log_2 (2T/\mathfrak{q})$ bits for the coefficients themselves, where $T$ is the $L^2$-bound for the elements of $\cF$.
Hence, in this procedure we are encoding with at most
\[
R(N)=N\big(C_1 + C_2 \log_2(N)\big),\qquad C_1,C_2>0,
\]
bits, and for $N\ge 2$ we have $R(N)\le  C_3 N \log_2(N)$ for some constant $C_3>0$.
To decode, we simply reconstruct the rounded values of the coefficients and then synthesize using the selected dictionary elements.

Let $\mathcal{H}$ be a hypercube in $\cF$ of dimension $m$ and sidelength $\delta$.
Starting with a vertex $h\in\mathcal{H}$ the coding-decoding procedure (for some fixed $N\in\N$) yields some $\tilde{h}\in L^2(\R^2)$.
By passing to the $L^2$-closest vertex $\hat{h}$, we again obtain an element of the hypercube $\mathcal{H}$.

Every vertex $h\in\mathcal{H}$ can be represented
as a word of $m$ bits, each bit corresponding to one side of the cube.
Thus the above coding procedure gives a map of the $m$ bits, which specify the
vertex $h\in\mathcal{H}$, to $R=R(N)$ bits. The decoding then reconstructs the $m$ bits specifying the
vertex $\hat{h}\in\mathcal{H}$. Since at the intermediate step we just have $R$ bits of information we
unavoidably loose information if $R<m$.

Now we can apply an information theoretic argument.
By rate-distortion theory \cite{Don01,berger1971rate} there must be at least one vertex $h\in\mathcal{H}$, where the
number of false reconstructed bits is larger than $D_m(R)$.
Here $D_m(R)$ is the so-called $m$-letter distortion rate function.
Since each bit determines a side of the cube, the error we make for this vertex $h$ obeys
\begin{align*}
\|h-\hat{h}\|_2^2 \ge \delta^2 \cdot D_m(R).
\end{align*}
Since by construction
$
\| \tilde{h}-h\|_2 \ge \|\tilde{h}-\hat{h} \|_2
$
we have
$
\|\tilde{h}-h\|_2 \ge \frac{1}{2} \| \hat{h}-h \|_2.
$
It follows
\[
\|\tilde{h}-h\|^2_2\ge \frac{1}{4} \delta^2 \cdot D_m(R).
\]

By assumption, $\cF$ contains a copy of $\ell^p_0$. Therefore we can find a sequence of hypercubes $\mathcal{H}_k$ with sidelengths $\delta_k\rightarrow 0$ as $k\rightarrow\infty$
and dimensions $m_k=m_k(\delta_k)\ge C\delta_k^{-p}$.
We pick $N_k\in\N$ such that $C_3 N_k \log_2(N_k)\le \frac{1}{3} m_k$ and $N_k\rightarrow \infty$ as $k\rightarrow\infty$.
For large $k$ we then obey the inequality $R(N_k)\le \frac{1}{3} m_k$, in fact $N_k\ge 2$ is sufficient.

Here we can apply another result from rate-distortion theory.
If $\frac{R}{m}\le\rho$ for some $\rho<\frac{1}{2}$ it holds $D_m(R)/m \ge D_1(\rho)$, where $D_1$ is the so-called single-letter distortion-rate function.
Hence, if $\frac{R}{m}\le\frac{1}{3}$, we have
\[
\|\tilde{h}-h\|^2_2\ge \frac{1}{4}D_1(\frac{1}{3}) \delta^2 m.
\]

Let $h_k$ denote the vertices with maximal reconstruction error $\|h_k-\tilde{h}_k\|_{2}$ at each hypercube $\mathcal{H}_k$.
Then we can conclude for large $k$
\begin{align*}
\|\tilde{h}_k-h_k\|_2^2  \ge \frac{1}{4}D_1({\scriptsize\frac{1}{3}}) \delta_k^2   m_k \gtrsim \delta_k^2 m_k \gtrsim  (N_k \log_2 (N_k))^{-(2-p)/p}.
\end{align*}
Finally we have to take care of the rounding errors.
The Euclidean distance between the best $N_k$-term approximation $h^\prime_{k}$ differs from $\tilde{h}_{k}$ by at most
$\mathfrak{q}\sqrt{N_k}$, i.e.
\begin{align*}
\|\tilde{h}_k-h^\prime_k\|_2 \le \mathfrak{q}\sqrt{N_k}\,,
\end{align*}
since the coefficients belong to an orthonormal basis.
It follows, with some constant $C>0$,
\begin{align*}
\|h_k-h^\prime_k\|_2 \ge  \|\tilde{h}_k-h_k\|_2 - \|\tilde{h}_k-h^\prime_k\|_2 \ge  C (N_k \log_2 (N_k))^{\frac{1}{2}-\frac{1}{p}} - N_k^{1/2-2/p} \gtrsim (N_k \log_2 (N_k))^{-(2-p)/(2p)}.
\end{align*}
This finishes the proof.
\end{proof}

\subsubsection{Frame Approximation}

In practice one needs representations, which are robust and stable with respect to noise.
This leads to the notion of a \emph{frame} \cite{Christensen2003a}, which is a special kind of dictionary.

A system $\Phi=(\varphi_\lambda)_{\lambda\in \Lambda}\subset L^2(\R^2)$
is called a frame for $L^2(\R^2)$, if there exist constants $0 < C_1 \le C_2 < \infty$ such that
\[
C_1 \|f\|_2^2 \le \sum_{\lambda\in \Lambda} |\langle f,\varphi_\lambda \rangle|^2 \le C_2 \|f\|_2^2 \quad \mbox{for all } f \in L^2(\R^2).
\]
A frame is called {\em tight}, if $C_1=C_2$ is possible, and {\em Parseval}, if $C_1=C_2=1$. Since the {\em
frame operator} $S : L^2(\R^2) \to L^2(\R^2)$ defined by $Sf=\sum_{\lambda\in \Lambda} \langle f,\varphi_\lambda \rangle \varphi_\lambda$
is always invertible, it follows that one particular sequence of coefficients in the expansion \eqref{eq:decomp} can be computed as
\begin{align}\label{eq:framecoeff}
c_\lambda=\langle f, S^{-1} \varphi_\lambda \rangle, \quad \lambda\in \Lambda,
\end{align}
where $(S^{-1}\varphi_\lambda)_\lambda$ is usually referred to as the {\em canonical dual frame}.
Note, that in general for a redundant system the expansion coefficients in \eqref{eq:decomp} are not unique.
The particular coefficients \eqref{eq:framecoeff}, however, are usually referred to as {\em the} frame coefficients.
They have the distinct property that they minimize the $\ell^2$-norm.

Even in the frame setting, the computation of the
best $N$-term approximation is not yet well-understood. The delicacy of this problem can for instance
be seen in \cite{GN04}. A typical approach to circumvent this problem is to consider instead the
$N$-term approximation obtained by choosing the $N$ largest frame coefficients.
It is evident that the associated error then also provides a bound for the error of best $N$-term approximation.

There exists a close relation between the $N$-term approximation rate achieved by a frame, and
the decay rate of the corresponding frame coefficients. A typical measure for the sparsity of a sequence $(c_\lambda)_\lambda$
are the $\ell^p$-(quasi-)norms, for $p>0$ defined by
\[
\|(c_\lambda)_\lambda\|_{\omega\ell^p}:=  \Big( \sup_{\varepsilon>0}  \varepsilon^p \cdot\#\big\{ \lambda : |c_\lambda|>\varepsilon \big\} \Big)^{1/p}.
\]
Equivalently, for a zero sequence $(c_n)_{n\in\N}$ indexed by $\N$, these (quasi-)norms can be characterized by
$\|(c_n)_n\|_{\omega\ell^p}= \sup_{n>0} n^{1/p} |c^\ast_n|$, where $(c_n^\ast)_n$ denotes the non-increasing rearrangement of $(c_n)_n$.

The following lemma shows that membership of the coefficient
sequence to an $\ell^p$-space for small $p$ implies good $N$-term approximation rates. For a
proof, we refer to \cite{Devore1998,Kutyniok2010}.

 \begin{lemma}\label{lem:decayapprox}
 Let $f=\sum c_\lambda \varphi_\lambda$ be an expansion of $f\in L^2(\R^2)$ with respect to a frame $(\varphi_\lambda)_{\lambda\in \Lambda}$.
 Further, assume that the coefficients satisfy $(c_\lambda)_\lambda\in \omega\ell^{2/(2k+1)}$ for some $k>0$.
 Then the best $N$-term approximation rate is at least of order $N^{-k}$, i.e.
 \begin{equation*}
  {\| f-f_N \|}_2 \lesssim N^{-k}.
 \end{equation*}
 \end{lemma}
Next, we apply Theorem~\ref{thm:upperbound} to obtain an upper bound on the approximation rate of cartoon-like images.

\subsection{Cartoon-like Images}\label{ssec:imagedata}

To model the fact that multivariate data appearing in applications is typically governed by anisotropic
features -- in the 2-dimensional case curvilinear structures --, the so-called {\em cartoon-like functions}
were introduced in \cite{Don01}. This class is by now a widely used
standard model in particular for natural images. It mimics the fact that natural images often consist
of nearly smooth parts separated by discontinuities.


The first rigorous mathematical definition was given in \cite{Don01} and extensively employed starting from
the work in \cite{CD04}. It postulates that images consist of $C^2$-regions separated by \ms{piecewise} smooth $C^2$-curves.
Since its introduction several extensions of this model have been introduced and studied, starting with
the extended model in \cite{Kutyniok2012}.
We consider images that contain two smooth $C^{\beta}$-regions separated by a \ms{piecewise smooth} $C^{\gamma}$-curve, where $\beta,\gamma\in(1,2]$.

\ms{Without loss of generality we will subsequently assume the smoothness of the edge curve.}
\ms{In addition, to avoid technicalities, we} restrict our considerations to star-shaped discontinuity curves, which
allow an easy parametrization. \ms{This is a classical simplification also used in \cite{Don01,CD04,Kutyniok2010,Kutyniok2012} }.
We remark however that this restriction \ms{to star-shaped curves} is artificial. In fact, we could work with the class of regular
$C^\gamma$-Jordan domains contained in $[0,1]^2$, where $\gamma\in(1,2]$ and the Hölder constants in the canonical parametrization are bounded.

\ms{
We begin by introducing the class of star-shaped sets $STAR^{\gamma}(\nu)$ for $\gamma\in(1,2]$ and $\nu>0$.
For that we take a $C^\gamma$-function $\rho:\mathbb{T}\rightarrow[0,\infty)$ defined on the torus $\mathbb{T}=[0,2\pi]$, where the boundary points are identified.
Additionally, we assume that there exists $0<\rho_0<1$ such that $\rho(\eta)\le\rho_0$ for all $\eta\in\mathbb{T}$.
Then we define the subset $\mathcal{B}\subset\R^2$ by
\begin{align}\label{eqn:cartoonim}
\mathcal{B}= \Big\{x \in \R^2 ~:~  x=(|x|_2,\eta) \text{ in polar coordinates with } \eta \in \mathbb{T},\, |x|_2 \le \rho(\eta) \Big\},
\end{align}
such that the boundary $\Gamma=\partial \mathcal{B}$ of $\mathcal{B}$ is a closed regular $C^\gamma$-Jordan curve in $[0,1]^2$ parameterized by
\begin{align}\label{eqn:radius}
b(\eta)=\begin{pmatrix} \rho(\eta)\cos(\eta)\\ \rho(\eta)\sin(\eta) \end{pmatrix}, \quad \eta\in\mathbb{T} .
\end{align}
Furthermore, we require the Hölder constant of the radius function $\rho$ to be bounded by $\nu$, i.e.,
\begin{align} \label{eqn:curvature}
\text{\sl Höl}(\rho^\prime,\gamma-1)= \sup_{\eta, \tilde{\eta}\in\mathbb{T}}\frac{|\rho^\prime(\eta)-\rho^\prime(\tilde{\eta})|}{|\eta-\tilde{\eta}|^{\gamma-1}} \le \nu,
\end{align}
where the distance $|\eta-\tilde{\eta}|$ is measured on the torus $\mathbb{T}$.
}

\begin{definition}
For $\nu>0$ the set $STAR^{\gamma}(\nu)$ is defined as the collection of all subsets $\mathcal{B} \subset [0,1]^2$, which are translates of sets \ms{of the form} \eqref{eqn:cartoonim}
\ms{with a boundary} obeying \eqref{eqn:radius} and \eqref{eqn:curvature}.
\end{definition}

The class of cartoon-like functions is then defined as follows.
\begin{definition}
\begin{enumerate}
\item Let $\nu>0$, $\beta,\gamma \in \left(1,2 \right]$. The \emph{cartoon-like functions} $\mathcal{E}^{\beta}_{\gamma}(\R^2;\nu)\subset L^2(\R^2)$ are defined as
the set of functions $f: \R^2 \to \mathbb{C}$ of the form
\begin{align*}
f=f_0+f_1\chi_{\mathcal{B}},
\end{align*}
where $\mathcal{B}\in STAR^\gamma(\nu)$ and $f_i \in C^{\beta}(\R^2)$ with $\supp f_0 \subset [0,1]^2$ and $\|f_i\|_{C^{\beta}} \le \nu$ for each $i=0,1$.
To simplify notation we let $\mathcal{E}^{\beta}(\R^2)=\mathcal{E}^{\beta}_{\beta}(\R^2)$. 
\item $\mathcal{E}^{bin}_{\gamma}(\R^2;\nu)$ denotes the class of \emph{binary cartoon-like images}, i.e., functions $f=f_0+f_1\chi_{\mathcal{B}} \in \mathcal{E}^{\beta}_{\gamma}(\R^2;\nu)$,
where $f_0=0$ and $f_1=1$.
\end{enumerate}
\end{definition}

We aim for a benchmark for sparse approximation of functions belonging to $\mathcal{E}^{\beta}_{\gamma}(\R^2)$.
For this we can use the results from the previous section.
In order to apply these results to our model of cartoon-like functions we investigate for which $p>0$ the class of cartoon-like images contains a copy of $\ell_0^p$
following \cite{Kutyniok2012,Kei13}.

\begin{theorem}
Let $\nu>0$ be fixed, and $\beta,\gamma \in \left(1,2 \right]$.
\begin{enumerate}
\item The class of binary cartoon-like images $\mathcal{E}^{bin}_{\gamma}(\R^2;\nu)$ contains a copy of $\ell_0^p$ for $p=2/(\gamma + 1)$.
\item The class $C_0^\beta([0,1]^2;\nu)$ of H\"older smooth functions $f\in C_0^\beta([0,1]^2)$ with $\|f\|_{C^\beta}\le \nu$
contains a copy of $\ell_0^p$ for $p=2/(\beta+1)$.
\end{enumerate}
\end{theorem}
\begin{proof}
ad 1) This was proved by Donoho in \cite{Don01}.

\noindent
ad 2)
To show that $C_0^\beta([0,1]^2;\nu)$ contains a copy of $\ell_0^p$ we have to find a sequence of embedded orthogonal hypercubes $(\mathcal{H}_k)_{k\in\N}$ of dimension $m_k$
and size $\delta_k$ such that $\delta_k\rightarrow 0$ and \eqref{eqn:locopy} holds.

Let $\phi \in C_0^{\infty}(\R)$ with $\supp \phi \subset [0,1]$ and $\phi\ge0$ and put $\psi(t)=\phi(t_1)\phi(t_2)$ for $t=(t_1,t_2)\in\R^2$.
Then $\psi\in C^{\beta}(\R^2)$ with $\supp \psi\subset[0,1]^2$. By possibly rescaling we can ensure $\|\psi\|_{ C^{\beta}(\R^2)}\le \nu$.
Further, we define for
$k \in \N$ and $i=(i_1,i_2) \in \{0,...,k-1\}^2$ the functions
\begin{center}
$\psi_{i,k}(t)= k^{-\beta}\phi(kt_1-i_1)\phi(kt_2-i_2)$.
\end{center}
These functions $\psi_{i,k}\in C^{\beta}(\R^2)$ are dilated and translated versions
of $\psi$ satisfying $\Hol(\psi_{i,k},\beta)=\Hol(\psi,\beta)$ and $\|\psi_{i,k}\|_{ C^{\beta}(\R^2)}\le\|\psi\|_{ C^{\beta}(\R^2)}\le\nu$.
Moreover, it holds $\|\psi_{i,k}\|_{2}^2=k^{-2\beta-2}\|\psi\|^2_{2}$, and
since $\supp\psi_{i,k} \subset [\frac{i_1}{k},\frac{i_1+1}{k}]\times[\frac{i_2}{k},\frac{i_2+1}{k}]$ the functions
$\psi_{i,k}$ and $\psi_{j,k}$ are orthogonal in $L^2(\R^2)$ for $i\neq j$.

The hypercubes $\cH_k$ of dimensions $m_k=k^2$ given by
\begin{align*}
 \cH_k:=\cH(k^2;0,(\psi_{i,k}))&= \big{\{ }h= \sum\limits_{i_1=1}^{k} \sum\limits_{i_2=1}^{k}\epsilon_{i_1,i_2}\phi(k\cdot -i_1)\phi(k\cdot -i_2) \hspace{0.1cm}|\hspace{0.1cm}\epsilon_{i_1,i_2} \in \{0,1\}  \big{\}}\\
& =\big{\{ }h=\sum\limits_{i=1}^{m^2}\epsilon_i \psi_{i,k} \hspace{0.1cm}|\hspace{0.1cm}\epsilon_i \in \{0,1\} \big{\}}
\end{align*}
with side-lengths $\delta_k=\|\psi_{i,k}\|_{2} = k^{-\beta-1}\|\psi\|_{2}$ are clearly contained in $C_0^\beta([0,1]^2;\nu)$. We have to check \eqref{eqn:locopy}. It holds
\begin{align*}
m_k=k^2 &=  \left(\frac{\delta_k}{\|\psi\|_{2}}\right)^{-\frac{2}{\beta+1}} = \|\psi\|_{2}^{\frac{2}{\beta+1}} \cdot \left(\delta_k\right)^{-\frac{2}{\beta+1}},
\end{align*}
and the sequence $(\delta_k)_{k\in\N}$ obeys $\delta_k \rightarrow 0$. This finishes the proof.
\end{proof}

An immediate corollary is the following result.

\begin{cor}
The function class $\mathcal{E}^{\beta}_{\gamma}(\R^2;\nu)$ contains a copy of $\ell_0^p$ for $p=\max\{2/(\beta + 1),2/(\gamma + 1)\}$.
\end{cor}

As a direct consequence of Theorem~\ref{thm:upperbound} we now state the main result of this subsection, which has also been proved in \cite{Kutyniok2012}.

\begin{theorem}\label{thm:benchmark}
Given an arbitrary dictionary and allowing only polynomial depth search, for the class of cartoon-like images $f \in \mathcal{E}^{\beta}_{\gamma}(\mathbb{R}^2)$
the decay rate of the error $\|f-f_N\|_2^2$ of the best $N$-term approximation cannot exceed $N^{-\min\{\beta,\gamma\}}$.
\end{theorem}

We observe that the upper bound on the approximation rate for $\mathcal{E}^{\beta}_{\gamma}(\R^2)$ is the same as for the larger class $\mathcal{E}^{\min\{\beta,\gamma\}}(\R^2)$.
Therefore we restrict our investigation to classes of the form $\mathcal{E}^{\beta}(\R^2)$. Since we are able to
establish a lower bound for $\mathcal{E}^{\beta}(\R^2)$ differing only by a log-factor from this upper bound, we do not loose much.

%
 %
\section{Hybrid Curvelets}
\label{ssec:HyCurve}
%

Second generation curvelets, which are nowadays simply referred to as curvelets, were introduced in 2004 by Cand\`{e}s and Donoho~\cite{CD04}.
The construction involves a parabolic scaling law, which can be viewed as a natural compromise between isotropic scaling, as utilized for wavelets, and
scaling in only one coordinate direction, as utilized for ridgelets \cite{GrohsRidLT}.

A more general anisotropic scaling law is realized by so-called $\alpha$-scaling, where the parameter $\alpha\in[0,1]$ is used to
control the degree of anisotropy in the scaling. The associated $\alpha$-scaling matrix is defined by
\[
D_s=\begin{pmatrix} s & 0 \\ 0 & s^\alpha  \end{pmatrix}, \quad s\in\mathbb{R}_+.
\]
Isotropic scaling corresponds to $\alpha=1$, and scaling in one coordinate direction to $\alpha=0$.
Parabolic scaling is the special case for $\alpha=\frac{1}{2}$. In this sense, curvelets can be viewed as lying in between ridgelets and wavelets.

Adapting the construction of curvelets to $\alpha$-scaling yields, what we will call hybrid curvelets, or more specifically $\alpha$-curvelets.
In the following we will construct bandlimited tight frames of $\alpha$-curvelets for every $\alpha\in[0,1]$.
Thus, we obtain a whole scale of representation systems, which interpolate between wavelets for $\alpha=1$ on the one end and
ridgelets for $\alpha=0$ on the other end.

The construction follows the same recipe used for the tight ridgelet frames in \cite{GrohsRidLT}.
As mentioned above it can also be seen as a variation of the classical second generation curvelet frame from \cite{CD04}.

The construction is simplified by treating the radial and angular
components separately.
For the construction of the radial functions $W^{(j)}$ at each scale $j\in\N_0$ we
start with $C^\infty$-functions
$\widetilde{W}^{(0)}:\R_+\rightarrow[0,1]$ and $\widetilde{W}:\R_+\rightarrow[0,1]$ with the following properties
\begin{align*}
\supp\, \widetilde{W}^{(0)}\subset[0,2], \quad \widetilde{W}^{(0)}(r)=1 \text{ for all }r\in[0,\textstyle{\frac{3}{2}}], \\
\supp\, \widetilde{W}\subset[\textstyle{\frac{1}{2}},2], \quad \widetilde{W}(r)=1 \text{ for all }r\in[\textstyle{\frac{3}{4}},\textstyle{\frac{3}{2}}].
\end{align*}
Then we define for $j\in\N$ the functions
$
\widetilde{W}^{(j)}(r):=\widetilde{W} (2^{-j} r)
$ on $\R_+$.
Finally, we rescale for every $j\in\N_0$ (to obtain an integer grid later)
\[
W^{(j)}(r):=\widetilde{W}^{(j)} (8\pi r), \quad r\in\R_+.
\]
Notice, that it holds $2\ge\sum W^{(j)}\ge 1$.

Next, we define the angular functions $V^{(j,\ell)}:\mathbb{S}^{1}\rightarrow[0,1]$
on the unit circle $\mathbb{S}^{1}\subset\R^2$, where $j\in\N$ specifies the scale and $\ell$ the orientation, running through $\big\{0,\ldots,L_j-1\big\}$ with
\begin{align*}
L_j=2^{\lfloor j(1-\alpha) \rfloor}, \quad j\in\N.
\end{align*}
This time we start with a $C^\infty$-function $V:\R\rightarrow[0,1]$, living on the whole of $\R$, satisfying
\begin{align*}
\supp\, V\subset [-\textstyle{\frac{3}{4}}\pi,\textstyle{\frac{3}{4}}\pi] && \text{ and } &&
V(t)= 1 \text{ for all }t\in[-\textstyle{\frac{\pi}{2}},\textstyle{\frac{\pi}{2}}].
\end{align*}
Since the interval $[-\pi,\pi]$ can be identified via $t\mapsto e^{it}$ with $\mathbb{S}^{1}$, we obtain
$C^\infty$-functions $\widetilde{V}^{(j,0)}:\mathbb{S}^{1}\rightarrow[0,1]$ for every $j\in\N$ by restricting the scaled functions
$V(2^{\lfloor j(1-\alpha) \rfloor} \cdot)$ to $[-\pi,\pi]$.
In order to end up with real-valued curvelets, we then symmetrize
\[
V^{(j,0)}(\xi):=\widetilde{V}^{(j,0)}(\xi) + \widetilde{V}^{(j,0)}(-\xi) \quad,\xi\in\mathbb{S}^1.
\]
At each scale $j\in\N$ we define the characteristic angle $\omega_j=\pi 2^{-\lfloor j(1-\alpha) \rfloor}$,
and for $\ell\in\big\{0,1,\ldots,L_j-1\big\}$ we let
\begin{align}\label{eq:matrixrot}
R_{j,\ell}
=\begin{pmatrix} \cos(\omega_{j,\ell}) & -\sin(\omega_{j,\ell}) \\
 \sin(\omega_{j,\ell}) & \,\cos(\omega_{j,\ell})
\end{pmatrix}
\end{align}
be the rotation matrix by the angle $\omega_{j,\ell}=\ell\omega_j$. By rotating $V^{(j,0)}$ we finally get $V^{(j,\ell)}:\mathbb{S}^{1}\rightarrow[0,1]$,
\[
V^{(j,\ell)}(\xi):=V^{(j,0)}(R_{j,\ell} \xi) \quad\text{ for }\xi\in\mathbb{S}^1.
\]
In order to secure the tightness of the frame we need to renormalize with the function
\[
\Psi(\xi):= W^{(0)}(|\xi|)^2 + \sum_{j,\ell} W^{(j)}(|\xi|)^2V^{(j,\ell)}\Big(\frac{\xi}{|\xi|}\Big)^2,
\]
which satisfies $1\le\Psi(\xi)\le8$ for all $\xi\in\R^2$. 
Bringing the radial and angular components together, we obtain the functions
\begin{align}\label{eq:suppfunctions}
\chi_{0}(\xi):=\frac{W^{(0)}(|\xi|))}{\sqrt{\Psi(\xi)}}, \qquad \chi_{j,\ell}(\xi)=\frac{W^{(j)}(|\xi|)V^{(j,\ell)}\Big(\frac{\xi}{|\xi|}\Big)}{\sqrt{\Psi(\xi)}} \qquad\text{for }\xi\in\R^2.
\end{align}
It is obvious that $\chi_{0}$, $\chi_{j,\ell}\in C^\infty(\R^2)$.
Moreover, these functions are real-valued, non-negative, compactly supported, and $L^\infty$-bounded by $1$. Let $\mathfrak{e}_1=(1,0)\in\R^2$ be the first unit vector and put
\[
S_j:=\Big\{ x\in\mathbb{S}^1 ~:~ |\langle x,\mathfrak{e}_1 \rangle| \ge \cos(\omega_j/2) \Big\}.
\]
Indeed, we have $\supp\, \chi_{0}\subset \mathcal{W}_{0}:=\{ \xi\in\R^2 : 8\pi|\xi|\le2 \}$ and $\supp\, \chi_{j,\ell}\subset \mathcal{W}_{j,\ell}$, where
for $j\in\N$, $\ell\in\big\{0,\ldots,L_j-1\big\}$ the sets
\begin{align}\label{eq:wedgePJ}
\mathcal{W}_{j,\ell}:=\Big\{ \xi\in\R^2 ~:~ 2^{j-1}\le 8\pi|\xi|\le2^{j+1},\quad \frac{\xi}{|\xi|}\in R^{-1}_{j,\ell}S_j \Big\}
\end{align}
are antipodal pairs of symmetric wedges.

After this preparation, we are ready to define the functions $\psi_{0}$ and $\psi_{j,\ell}$ on the Fourier side by
\[
\widehat{\psi}_{0}(\xi):=\chi_{0}(\xi) \quad\text{ and }\quad \widehat{\psi}_{j,\ell}(\xi):= \chi_{j,\ell}(\xi) \text{ for }\xi\in\R^2.
\]
They have the following important property.

\begin{lemma}\label{lem:CurveletFramePW}
For every $f\in L^2(\R^2)$ it holds
\[
\|f\|_2^2= \|f\ast\psi_{0}\|_2^2 + \sum_{j,\ell} \|f\ast\psi_{j,\ell}\|_2^2.
\]
\end{lemma}
\begin{proof}
By construction the system satisfies the discrete Caldéron condition, i.e., 
\[
|\widehat{\psi}_{0}(\xi)|^2 + \sum_{j,\ell} |\widehat{\psi}_{j,\ell}(\xi)|^2=1 \quad\text{ for all }\xi\in\R^2.
\]
This yields for all $f\in L^2(\R^2)$
\begin{align*}
\|f\|_2^2&= \int_{\R^2} |\widehat{f}(\xi)|^2 \,d\xi \\
&= \int_{\R^2} |\widehat{\psi}_{0}(\xi)|^2 |\widehat{f}(\xi)|^2 \,d\xi + \int_{\R^2} \sum_{j,\ell} |\widehat{\psi}_{j,\ell}(\xi)|^2 |\widehat{f}(\xi)|^2 \,d\xi\\
&= \|f\ast\psi_{0}\|_2^2 + \sum_{j,\ell} \|f\ast\psi_{j,\ell}\|_2^2.
\end{align*}
\end{proof}
The full frame of $\alpha$-curvelets is obtained by taking translates of $\psi_{0}$ and $\psi_{j,\ell}$ in the spatial domain
and $L^2$-normalizing afterwards.
Accordingly, for $j\in\N$, $\ell\in\{0,1,\ldots,L_j-1\}$, and $k\in\Z^2$ we define $\psi_{0,k}:=\psi_0(\cdot-k)$ and
\begin{align*}
 \psi_{j,\ell,k}:= 2^{-j(1+\alpha)/2} \cdot \psi_{j,\ell}(\cdot-x_{j,\ell,k})\quad\text{with}\quad x_{j,\ell,k}=R^{-1}_{j,\ell}D^{-1}_{2^{j}}k.
\end{align*}
The corresponding set of curvelet indices $\mu=(j,\ell,k)$ will henceforth be denoted by $M$. Further, we will use the following notation for this system
\[
\mathcal{C}_\alpha(W^{(0)},W,V)=(\psi_\mu)_{\mu\in M} =\Big\{ \psi_{0,k} ~:~ k\in\Z^2 \Big\} \cup \Big\{ \psi_{j,\ell,k} ~:~ j\in\N,\,k\in\Z^2,\,\ell\in\{0,1,\ldots,L_j-1\} \Big\}.
\]

\begin{theorem}
Let $\alpha\in[0,1]$ and $W^{(0)}$, $W$, $V$ be defined as above. The $\alpha$-curvelet system $\mathcal{C}_\alpha(W^{(0)},W,V)$
constitutes a tight frame for $L^2(\R^2)$.
\end{theorem}
\begin{proof}
For $j\in\N$ the wedges $\mathcal{W}_{j,0}$ are contained in the rectangles $\Xi_{j,0}$,
\begin{align}\label{eq:supprect}
\mathcal{W}_{j,0}\subset \frac{1}{8\pi} \Big([-2^{j+1},2^{j+1}]\times[-4\pi 2^{j\alpha},4\pi 2^{j\alpha}]\Big) \subset [-2^{j-1},2^{j-1}]\times[-2^{j\alpha-1},2^{j\alpha-1}]=:\Xi_{j,0},
\end{align}
and $\mathcal{W}_0$ is contained in the cube $[-2,2]^2/(8\pi)\subset[-\frac{1}{2},\frac{1}{2}]^2=:\Xi_0$. Since $\mathcal{W}_{j,\ell}=R^{-1}_{j,\ell}\mathcal{W}_{j,0}$
we can put $\Xi_{j,\ell}:=R^{-1}_{j,\ell}\Xi_{j,0}$ such that $\mathcal{W}_{j,\ell}\subset\Xi_{j,\ell}$.

For each fixed $j\in\N$ the Fourier system $(u_{j,k})_{k\in\Z^2}$ given by
\begin{align}\label{eq:fouriersys}
u_{j,k}(\xi)=2^{-j(1+\alpha)/2}\exp\big(2\pi ix_{j,0,k}\cdot\xi\big) \quad \text{with} \quad x_{j,0,k}=(2^{-j}k_1, 2^{-j\alpha}k_2)=D^{-1}_{2^j}k
\end{align}
constitutes an orthonormal basis for $L^2(\Xi_{j,0})$. Observe that $x_{j,\ell,k}=R^{-1}_{j,\ell}x_{j,0,k}$. Therefore,
\[
u_{j,k}(R_{j,\ell}\xi)=2^{-j(1+\alpha)/2}\exp\big(2\pi ix_{j,\ell,k}\cdot\xi\big),
\]
and the system $\big(u_{j,k}(R_{j,\ell}\cdot)\big)_{k\in\Z^2} $ is an orthonormal basis for $L^2(\Xi_{j,\ell})$. By Lemma~\ref{lem:CurveletFramePW} it holds
\[
\|f\|_2^2= \|f\ast\psi_{0}\|_2^2 + \sum_{j,\ell} \|f\ast\psi_{j,\ell}\|_2^2.
\]
By the above observation we have for $j\in\N$
\begin{align*}
\|f\ast\psi_{j,\ell}\|_2^2 &=  \|\widehat{f}\cdot\widehat{\psi}_{j,\ell}\|_2^2 =
\int_{\Xi_{j,\ell}} |\widehat{f}(\xi)\widehat{\psi}_{j,\ell}(\xi)|^2 \,d\xi
=\sum_{k\in\Z^2} \Big|\int \widehat{f}(\xi)\widehat{\psi}_{j,\ell}(\xi) u_{j,k}(R_{j,\ell}\xi) \,d\xi\Big|^2 \\
&=\sum_{k\in\Z^2}  |\langle\widehat{f}, 2^{-j(1+\alpha)/2}\widehat{\psi}_{j,\ell}e^{2\pi ix_{j,\ell,k}\cdot} \rangle|^2
=\sum_{k\in\Z^2}  |\langle f,\psi_{j,\ell,k} \rangle|^2.
\end{align*}
A similar argument for $\|f\ast\psi_{0}\|_2^2$ finishes the proof.
\end{proof}

In order to simplify the notation we henceforth use the capital letter $J$ to denote a scale-angle pair $(j,\ell)$.
In this context $|J|$ shall denote the corresponding scale variable $j$,
i.e. for $J=(j,\ell)$ it is $|J|=j$. Therefore, from now on we may write e.g.\ $\omega_J$, $R_J$ to abbreviate $\omega_{j,\ell}$, $R_{j,\ell}$, etc.

\section{Cartoon Approximation with $\alpha$-Curvelets}

In this main section we will use the tight frame of $\alpha$-curvelets $\mathcal{C}_\alpha(W^{(0)},W,V)=(\psi_\mu)_{\mu\in M}$, constructed in the previous section,
to approximate the cartoon image class $\cE^{\beta}(\RR^2)$. In this investigation we keep $\alpha\in [\frac{1}{2},1)$ fixed in the range between wavelets and curvelets,
and put $\beta=\alpha^{-1}$.
Our main goal is to prove the following approximation result, which generalizes an analogous result in \cite{CD04} for second generation curvelets.

\begin{theorem}\label{thm:mainappr1}
   Let $\alpha\in[\frac{1}{2},1)$ and $\beta=\alpha^{-1}$. The tight frame of $\alpha$-curvelets $\mathcal{C}_\alpha(W^{(0)},W,V)$ constructed in Section \ref{ssec:HyCurve}
   provides almost optimally sparse approximations for cartoon-like functions $\cE^{\beta}(\RR^2)$. More precisely,
   there exists some constant $C$ such that for every $f\in \cE^{\beta}(\RR^2)$
\begin{align*}
 \|f-f_N\|_2^2\le CN^{-\beta} \cdot \left(\log_2 N\right)^{\beta+1} \text{\hspace{0.5cm}as }  N \rightarrow \infty,
\end{align*}
where $f_N$ is the $N$-term approximation of $f$ obtained by choosing the $N$ largest curvelet coefficients.
\end{theorem}

When we compare this theorem with the benchmark Theorem~\ref{thm:benchmark}, we see that $\mathcal{C}_\alpha(W^{(0)},W,V)=(\psi_\mu)_{\mu\in M}$
attains the maximal achievable approximation rate up to a log-factor. This is remarkable since this rate is achieved by simply thresholding the coefficients,
leading to an intrinsically non-adaptive approximation scheme.

Theorem~\ref{thm:mainappr1} is proved in several steps, extending the techniques used in \cite{CD04}.
The basic idea is to study the decay of the sequence $(\theta_\mu)_{\mu\in M}$ of curvelet coefficients $\theta_\mu=\langle f,\psi_\mu \rangle$ for $f\in \cE^{\beta}(\RR^2)$.
According to Lemma~\ref{lem:decayapprox} this decay rate provides us with the desired information on the approximation rate.


\subsection{Sparsity of Curvelet Coefficients for Cartoon-like Images}


We first state a simple a-priori estimate for the size of the curvelet coefficients $\theta_\mu=\langle f,\psi_\mu \rangle$ at scale $j$.
The curvelets $\psi_\mu$ are $L^2$-normalized and essentially supported in a box of length $2^{-\alpha j} $ and width $2^{-j}$.
Therefore we can estimate
\[
\|\psi_\mu\|_{1} \le B 2^{-(1+\alpha)j/2} \|\psi_\mu\|_{2} \le B 2^{-(1+\alpha)j/2},
\]
where the constant $B$ is uniform over all scales. This argument can easily be made rigorous, and we obtain the following estimate
for the coefficients at scale $j$:
\begin{align}\label{eq:apriori}
|\theta_\mu|=|\langle f,\psi_\mu \rangle| \le \|f\|_{\infty} \|\psi_\mu\|_{1} \le B \|f\|_{\infty} 2^{-(1+\alpha)j/2}.
\end{align}
Not surprisingly, this a-priori estimate is not yet sufficient to prove Theorem~\ref{thm:mainappr1}.
A more sophisticated result is the following theorem.
\begin{theorem}\label{thm:mainappr2}
   Let $\theta^*_{N}$ denote the (in modulus) $N$-th largest curvelet coefficient. Then there exists some universal constant $C$ such that
\begin{align*}
 \sup_{f\in \mathcal{E}^{\beta}(\R^2)}  |\theta^*_{N}| \le C\cdot N^{-(1+\beta)/2} \cdot \left(\log_2 N\right)^{(1+\beta)/2}.
\end{align*}
 \end{theorem}
  \begin{proof}
  Let $M_j\subset M$ denote the indices corresponding to curvelets at scale $j$, and
  for $\varepsilon>0$ put
  \[
  M_{j,\varepsilon} = \Big\{ \mu\in M_j, |\theta_\mu|>\varepsilon \Big\}.
  \]
  By Theorem~\ref{thm:main_sequence}, which is stated and proved below, we have for $\varepsilon>0$
  \begin{align}\label{eq:proof1}
  \# M_{j,\varepsilon} = \# \Big\{ \mu\in M_j, |\theta_\mu|>\varepsilon \Big\} \lesssim  \varepsilon^{-2/(1+\beta)}.
  \end{align}
  On the other hand, \eqref{eq:apriori} shows that there is a constant $B$, independent of scale, such that
  \[
  |\theta_\mu| \le B \|f\|_{\infty} 2^{-(1+\alpha)j/2}.
  \]
  It follows that for each $\varepsilon>0$ there is $j_\varepsilon$ such that at scales $j\ge j_\varepsilon$ the coefficients satisfy $\theta_\mu<\varepsilon$.
  Hence, for $j\ge j_\varepsilon$
  \[
  \#M_{j,\varepsilon}= \# \Big\{ \mu\in M_j, |\theta_\mu|>\varepsilon \Big\} = 0.
  \]
  The number of scales at which $M_{j,\varepsilon}$ is nonempty is therefore bounded by
  \begin{align}\label{eq:proof2}
  \frac{2}{1+\alpha} \Big( \log_2(B) + \log_2(\|f\|_{\infty})  +  \log_2 (\varepsilon^{-1}) \Big) \lesssim \log_2 (\varepsilon^{-1}).
  \end{align}
  It follows from \eqref{eq:proof1} and \eqref{eq:proof2} that there is a constant $\widetilde{C}\ge1$ such that
  \[
  \# \Big\{ \mu\in M, |\theta_\mu|>\varepsilon \Big\} = \sum_{j} \# \Big\{ \mu\in M_j, |\theta_\mu|>\varepsilon \Big\}  \le \widetilde{C}  \varepsilon^{-2/(1+\beta)} \log_2(\varepsilon^{-1}).
  \]
  Let $\theta^\ast_N$ be the $N$-th largest coefficient. Then for $\varepsilon_N>\delta_N$, where $\delta_N$ satisfies $N=\widetilde{C}\delta_N^{-2/(1+\beta)} \log_2(\delta_N^{-1}) $,
  we have $|\theta^*_{N}|\le \varepsilon_N$.
  If $N\ge2$ it holds $\widetilde{C} N^{4/(1+\beta)} \log_2(N^{2}) \ge N$, because $1\le\beta\le 2$ and $\widetilde{C}\ge1$.  For $N\ge2$ therefore
  \[
  N^{4/(1+\beta)} \log_2(N^{2}) \ge  \delta_N^{-2/(1+\beta)} \log_2(\delta_N^{-1}) .
  \]
  This implies $\delta_N\ge N^{-2}$, and we can conclude that $\varepsilon_N>\delta_N$ if we choose $\varepsilon_N$ as the solution of
  \[
   N= \widetilde{C}\varepsilon_N^{-2/(1+\beta)} \log_2(N^{2}) .
  \]
  This choice leads to
  \[
  \varepsilon_N =(2\widetilde{C})^{(1+\beta)/2} \cdot N^{-(1+\beta)/2} (\log_2 N)^{(1+\beta)/2},
  \]
  which proves our claim with constant $C=(2\widetilde{C})^{(1+\beta)/2}$.
  \end{proof}
In fact, Theorem~\ref{thm:mainappr2} is strong enough to deduce Theorem~\ref{thm:mainappr1} via Lemma~\ref{lem:decayapprox}.
\begin{proof}[Proof of Theorem~\ref{thm:mainappr1}]
Applying Lemma~\ref{lem:decayapprox} and Theorem~\ref{thm:mainappr2} we can estimate
\[
\| f-f_N \|^2 \lesssim \sum_{m>N} |\theta^\ast_m|^2 \lesssim  \sum_{m>N} m^{-(1+\beta)} \cdot \left(\log_2 m\right)^{(1+\beta)}
\lesssim \int_N^\infty t^{-(1+\beta)} \cdot \left(\log_2 t\right)^{(1+\beta)} \,dt.
\]
Using partial integration we obtain
\begin{align*}
 \int_N^\infty t^{-(1+\beta)} \cdot \left(\log_2 t\right)^{(1+\beta)} \,dt &\lesssim [ -t^{-\beta}  \left(\log_2 t\right)^{(1+\beta)}  ]_N^\infty +
 \int_N^\infty t^{-(1+\beta)} \cdot \left(\log_2 t\right)^{\beta} \,dt \\
 &\lesssim N^{-\beta}  \left(\log_2 N\right)^{(1+\beta)}
 + \int_N^\infty t^{-(1+\beta)} \cdot \left(\log_2 t\right)^{\lceil\beta\rceil} \,dt  \\
 &\lesssim  \ldots  \lesssim  N^{-\beta}  \left(\log_2 N\right)^{(1+\beta)} + \int_N^\infty t^{-(1+\beta)}  \,dt  \\
 &\lesssim  N^{-\beta}  \left(\log_2 N\right)^{(1+\beta)}.
\end{align*}
\end{proof}
It remains to prove Theorem~\ref{thm:main_sequence}, which is the main building block in the proof of Theorem~\ref{thm:mainappr2}.
By abuse of notation, we write $\theta_j$ for the subsequence $(\theta_\mu)_{\mu\in M_j}$ of curvelet coefficients, where $M_j$ denotes the
curvelet indices at scale $j$. Theorem~\ref{thm:main_sequence} then reads as follows.

\begin{theorem}\label{thm:main_sequence}
The sequence $\theta_j$ obeys
\[
\| \theta_j \|_{\omega\ell_{2/(1+\beta)}}\le C,
\]
for some constant $C$ independent of scale.
\end{theorem}

For the analysis of the coefficient sequence $\theta_j$ we follow the technique in \cite{CD04} and smoothly decompose $f$ into so-called \emph{fragments},
which can then be analyzed separately. For that we cover $\R^2$ at each scale $j\in\N_0$ with cubes
\[
Q=[(k_1-1) 2^{-j\alpha}, (k_1+1) 2^{-j\alpha}]\times[(k_2-1) 2^{-j\alpha}, (k_2+1) 2^{-j\alpha}], \quad j\in\N,\,(k_1,k_2)\in\Z^2,
\]
which we collect in the sets $\mathcal{Q}_j$. Further, we put $\mathcal{Q}:= \bigcup_{j\in\N} \mathcal{Q}_j$.
Note how the size of the squares depends upon the scale $2^{-j}$:
The `width' of the curvelets at scale $j$ obeys $\sim 2^{-j}$ and the `length' of the curvelets is approximately $\sim 2^{-\alpha j}$.
Thus, the size of the squares is about the length of the curvelets.

Next, we take a smooth partition of unity $(\omega_Q)_{Q\in\mathcal{Q}_j}$, where these squares are used as the index set and
the functions $\omega_Q$ are supported in the corresponding squares $Q:=(2^{-j\alpha}k_1,2^{-j\alpha}k_2) + [-2^{-j\alpha},2^{-j\alpha}]^2$.
More precisely,
for some fixed nonnegative $C^\infty$-function $\omega$ vanishing outside the square $[-1,1]^2$, we put $\omega_Q=\omega(2^{j\alpha}x_1-k_1, 2^{j\alpha}x_2-k_2)$
and assume that
$
\sum_{Q\in\mathcal{Q}_j} \omega_Q(x)\equiv1.
$
The function $f$ can then at each scale $j\in\N_0$ be smoothly localized into the fragments
\[
f_Q:=f \omega_Q, \quad Q\in\mathcal{Q}_j.
\]
For $Q\in\mathcal{Q}_j$ let $\theta_Q$ denote the curvelet coefficient sequence of $f_Q$ at scale $j$, i.e.,
\begin{align}\label{eq:Qsequence}
\theta_{Q}=\Big(\langle f_Q, \psi_\mu \rangle \Big)_{\mu\in M_j}.
\end{align}

The strategy laid out in \cite{CD04} is to analyze the sparsity of the sequences $\theta_Q$ and combine these results
to obtain Theorem~\ref{thm:main_sequence}. In this investigation we have to distinguish between two cases: Either the square
$Q\in\mathcal{Q}_j$ meets the edge curve $\Gamma=\partial\mathcal{B}$ or not.
Accordingly, we let $\mathcal{Q}_j^0$ be the subset of $\mathcal{Q}_j$ containing those cubes, which intersect the edge curve $\Gamma$.
Among the remaining cubes of $\mathcal{Q}_j$ we collect those, which intersect $\supp f$, in $\mathcal{Q}_j^1$. The others can be neglected, because they
lead to trivial sequences $\theta_Q$.

The following two propositions directly lead to Theorem~\ref{thm:main_sequence}.

\begin{prop}\label{prop:sequence1}
Let $Q$ be a square such that $Q\in\mathcal{Q}_j^0$. The curvelet coefficient sequence $\theta_Q$ obeys
\[
\| \theta_Q \|_{\omega\ell_{2/(1+\beta)}}\le C\cdot 2^{-(1+\alpha)j/2},
\]
for some constant $C$ independent of $Q$ and $j$.
\end{prop}

\begin{prop}\label{prop:sequence2}
Let $Q$ be a square such that $Q\in\mathcal{Q}_j^1$. The curvelet coefficient sequence $\theta_Q$ obeys
\[
\| \theta_Q \|_{\omega\ell_{2/(1+\beta)}}\le C\cdot 2^{-(1+\alpha)j},
\]
for some constant $C$ independent of $Q$ and $j$.
\end{prop}
Theorem~\ref{thm:main_sequence} is an easy consequence of these two results and the observation, that
there are constants $A_0$ and $A_1$, independent of scale, such that
\begin{align}\label{eq:numcubes}
\#\mathcal{Q}_j^0\le A_0 2^{\alpha j} \quad\text{and}\quad \#\mathcal{Q}_j^1\le A_1 \big(2^{2\alpha j} + 4\cdot2^{\alpha j}\big).
\end{align}
The estimates \eqref{eq:numcubes} hold true since $f$ is supported in $[0,1]^2\subset [-1,1]^2$.
\begin{proof}[Proof of Theorem~\ref{thm:main_sequence}]
For $0<p\le1$ we have the $p$-triangle inequality
\[
\|a+b\|^p_{\omega\ell_p} \le \|a\|^p_{\omega\ell_p} + \|b\|^p_{\omega\ell_p}, \qquad a,b\in\omega\ell_p.
\]
Since $\theta_j=\sum_{Q\in\mathcal{Q}_j} \theta_Q$, we can conclude
\[
\|\theta_j\|^{2/(1+\beta)}_{\omega\ell_{2/(1+\beta)}} \le \sum_{Q\in\mathcal{Q}_j} \|\theta_Q\|^{2/(1+\beta)}_{\omega\ell_{2/(1+\beta)}}
\le \big(\#\mathcal{Q}_j^0\big) \cdot \sup_{\mathcal{Q}_j^0} \|\theta_Q\|^{2/(1+\beta)}_{\omega\ell_{2/(1+\beta)}} +
\big(\#\mathcal{Q}_j^1\big) \cdot \sup_{\mathcal{Q}_j^1} \|\theta_Q\|^{2/(1+\beta)}_{\omega\ell_{2/(1+\beta)}}.
\]
The claim follows now from the above two propositions together with observation \eqref{eq:numcubes}.
\end{proof}

It remains to prove Propositions~\ref{prop:sequence1} and ~\ref{prop:sequence2}.
%
%
%
For that let $Q\in\mathcal{Q}_j$ be a fixed cube at a fixed scale $j\in\N_0$, which nontrivially intersects $\supp f$.
We need to analyze the decay of the sequence $\theta_Q=(\langle f_Q, \psi_\mu\rangle)_{\mu\in M_j}$. Since the frame elements $\psi_\mu$ are bandlimited, it is
advantageous to turn to the Fourier side. The Plancherel identity yields
\[
\langle f_Q, \psi_\mu\rangle=\langle \widehat{f}_Q, \widehat{\psi}_\mu\rangle.
\]
These scalar products can be estimated, if we have knowledge about the localization of the functions $\widehat{f}_Q$.
This investigation is carried out separately in Subsections~\ref{ss:SmoothFragment} and \ref{ss:EdgeFragment} for the cases $Q\in\mathcal{Q}^0_j$ and $Q\in\mathcal{Q}^1_j$, respectively.


\subsection{Analysis of a Smooth Fragment}\label{ss:SmoothFragment}

Let us first treat the case $Q\in\mathcal{Q}^1_j$, where the cube $Q$ does not intersect the edge curve $\Gamma=\partial \mathcal{B}$.
In this case we call $f_Q=f\omega_Q$ a \emph{smooth fragment}.

Before we begin, we briefly recall our setting.
The parameters $\alpha\in[\frac{1}{2},1)$ and $\beta=\alpha^{-1}\in(1,2]$ are fixed, as is $f\in\mathcal{E}^\beta(\R^2)$.
Since $Q$ does not intersect the edge curve, there is a function $g\in C^\beta(\R^2)$ such that
$f_Q=g\omega_Q$. By smoothly cutting $g$ off outside the square $[-1,1]^2$, we can even assume $g\in C_0^\beta(\R^2)$.

We want to analyze $\widehat{f}_Q$ and for simplicity we look at the following model situation.
Without loss of generality we assume that the cube $Q$ is centered at the origin, by possibly translating the coordinate system.
In this case the smooth fragment takes the simple form
\[
f_Q=g \omega(2^{\alpha j}\cdot),
\]
where $g\in C_0^\beta(\R^2)$ and $\omega\in C_0^\infty(\R^2)$, $\supp \omega\subset [-1,1]^2$, is the fixed window generating the partition of unity $(\omega_Q)_Q$.
By rescaling we obtain for each scale $j$ the functions
\begin{align}
F(x)=g(2^{-\alpha j}x)\omega(x),\quad x\in\R^2,
\end{align}
with $\supp F\subset[-1,1]^2$. We put $g_j(x):=g(2^{-\alpha j}x)$, so that we can write $F(x)=g_j(x)\omega(x)$. It is important to note, that $g_j$ and $F$ depend on the scale, whereas
$\omega$ remains fixed.

\ms{Since the following inverstigation takes place in the model situation, the explicit reference to the cube $Q$ is not necessary in this exposition.
Therefore, by abuse of notation, we will write henceforth $f$ for the standard segment $f_Q$ and accordingly $\theta$ for the sequence $\theta_Q$.}


\subsubsection{Fourier Analysis of a Smooth Fragment}

We first analyze the localization of $\widehat{F}$. The key result in this direction will be Theorem~\ref{thm:mainest0}.
Its proof relies on Lemma~\ref{lem:central0} below.

\begin{lemma}\label{lem:central0}
Assume that $h=C2^{-(1-\alpha)j}$ for some fixed constant $C>0$. We then have
\[
\| \Delta^3_{(h,0)}\partial_1F \|^2_2 \lesssim h^{2\beta} 2^{-2j\alpha},
\]
where the implicit constant is independent of the scale $j$. Notice that $h$ is not independent and depends on the scale $j$.
\end{lemma}
\begin{proof}
Since $\supp F\subset [-1,1]^2$ it suffices to prove
\begin{align}\label{eq:aux_lemg0}
\| \Delta^3_{(h,0)}\partial_1F \|_{\infty} \lesssim h^{\beta} 2^{-j\alpha}.
\end{align}
By the product rule we have
$
\partial_1F = \partial_1g_j\cdot \omega + g_j\cdot\partial_1\omega
$
and it holds
    \begin{align*}
    \Delta^3_{(h,0)} \big(\partial_1g_j\cdot\omega\big) &= \sum_{k=0}^3 \Delta_{(h,0)}^k \partial_1g_j \cdot \Delta^{3-k}_{(h,0)} \omega(\cdot+kh,\cdot), \\
    \Delta^3_{(h,0)} \big(g_j\cdot\partial_1\omega\big) &= \sum_{k=0}^3 \Delta_{(h,0)}^k g_j \cdot \Delta^{3-k}_{(h,0)} \partial_1\omega(\cdot+kh,\cdot).
    \end{align*}
Clearly, we have for every $k\in\N$ the estimates $\|\Delta_{(h,0)}^k \omega\|_\infty \lesssim h^k$
    and
    $\|\Delta_{(h,0)}^k \partial_1\omega\|_\infty \lesssim h^k$.
    Further, according to Lemma~\ref{lemapp:basicg}, it holds $\|\partial_1g_j\|_\infty \lesssim  2^{-\alpha j}$,
    $\|\Delta_{(h,0)} g_j\|_\infty \lesssim  h 2^{-\alpha j}$, and
    \begin{align*}
    \|\Delta_{(h,0)}^k g_j\|_\infty &\lesssim h^{\beta} 2^{-j} \quad\text{for $k\ge 2$}, \\
    \|\Delta_{(h,0)}^k \partial_1g_j\|_\infty &\lesssim  h^{\beta} 2^{-\alpha j} \quad\text{for $k\ge 1$}.
    \end{align*}
    Since $h^3 \lesssim h^{\beta} 2^{-j\alpha}$ for $h=C 2^{-(1-\alpha)j}$ and $\alpha\in [\frac{1}{2},1)$, the assertion \eqref{eq:aux_lemg0} follows.
\end{proof}

The previous lemma is key to the proof of the following theorem.
Here we use the notation $|\xi|\sim 2^{(1-\alpha) j}$ to indicate $|\xi|\in [C_12^{(1-\alpha) j},C_22^{(1-\alpha) j}]$ for some arbitrary but fixed constants $0<C_1\le C_2<\infty$. A typical choice would be $C_1=1$ and $C_2=2^{1-\alpha}$.

\begin{theorem}\label{thm:mainest0}
It holds independent of the scale $j$
\begin{align*}
\int_{|\xi|\sim 2^{(1-\alpha)j}} |\widehat{F}(\xi)|^2 \,d\xi \lesssim 2^{-2\beta j}.
\end{align*}
\end{theorem}
%

\begin{proof}
Let $0<C_1\le C_2<\infty$ be fixed and choose $C>0$ such that $C_2C<2\pi$. Putting $h:=C 2^{-(1-\alpha)j}$,
there then exists $c>0$ such that $|e^{i\xi_1 h}-1|^2\ge c $ for every $\xi_1$ with $|\xi_1|\in[C_12^{(1-\alpha) j},C_22^{(1-\alpha) j}]$.
Using Lemma~\ref{lem:central0} we then estimate the integrals on the vertical strips:
\begin{align*}
2^{2(1-\alpha)j} \int_{|\xi_1|\sim 2^{(1-\alpha)j}} \int_{\xi_2}  |\widehat{F}(\xi_1,\xi_2)|^2 \,d\xi_2\,d\xi_1
&\asymp \int_{|\xi_1|\sim 2^{(1-\alpha)j}} \int_{\xi_2} |\xi_1|^2 |\widehat{F}(\xi_1,\xi_2)|^2 \,d\xi_2\,d\xi_1 \\
&\lesssim \int_{|\xi_1|\sim 2^{(1-\alpha)j}} \int_{\xi_2} |e^{i\xi_1 h}-1|^6 |\xi_1|^2 |\widehat{F}(\xi_1,\xi_2)|^2 \,d\xi_2\,d\xi_1 \\
&= \int_{|\xi_1|\sim 2^{(1-\alpha)j}} \int_{\xi_2} |\widehat{\Delta^3_{(h,0)}\partial_1 F}(\xi_1,\xi_2)|^2  \,d\xi_2\,d\xi_1  \\
&\le \int_{\R^2}  |\widehat{\Delta^3_{(h,0)}\partial_1F}(\xi)|^2   \,d\xi
= \| \Delta^3_{(h,0)}\partial_1F \|_2^2
\lesssim  h^{2\beta} 2^{-2j\alpha}.
\end{align*}
Interchanging $\xi_1$ and $\xi_2$ yields analogous estimates for the horizontal strips. Altogether, we obtain
\begin{align*}
\int_{|\xi|\sim 2^{(1-\alpha)j}} |\widehat{F}(\xi)|^2 \,d\xi
\lesssim  2^{-2j(1-\alpha)}  h^{2\beta}  2^{-2j\alpha}
\asymp  2^{-2\beta j}  .
\end{align*}

\end{proof}

As an immediate corollary, we deduce a corresponding estimate for the original smooth fragment $f$.

\begin{theorem}\label{thm:localf0}
We have independent of scale $j$
\begin{align*}
\int_{|\xi|\sim 2^j} |\widehat{f}(\xi)|^2 \,d\xi \lesssim 2^{-2(\beta+\alpha)j}.
\end{align*}
\end{theorem}
\begin{proof}
The statement follows from the relation $\widehat{f}(\xi)=2^{-2\alpha j}\widehat{F}(2^{-\alpha j}\xi)$.
\end{proof}

Finally, we state a refinement of Theorem~\ref{thm:localf0}.

\begin{cor}\label{cor:localf0}
Let $m=(m_1,m_2)\in N_0^2$ and $\partial^m=\partial_1^{m_1}\partial_2^{m_2}$. We have
\begin{align*}
\int_{|\xi|\sim 2^j} |\partial^m\widehat{f}(\xi)|^2 \,d\xi \lesssim 2^{-2j\alpha |m|}  2^{-2(\beta+\alpha)j}.
\end{align*}
\end{cor}
\begin{proof}
Recall that $f= g\omega(2^{\alpha j}\cdot)$.
Let us define the window $\omega_m(x):=x^m \omega(x)$ and the function
$f_m(x):= g(x)\omega_m(2^{\alpha j}x)$ for $x\in\R^2$.
Then for every $x\in\R^2$
\begin{align*}
x^m f(x)= g(x) x^m \omega(2^{\alpha j}x) = 2^{-j\alpha|m|} g(x) \omega_m(2^{\alpha j}x) = 2^{-j\alpha|m|} f_m(x) .
\end{align*}
We conclude with Theorem~\ref{thm:localf0}
\begin{align*}
\int_{|\xi|\sim 2^j} |\partial^m\widehat{f}(\xi)|^2 \,d\xi
= \int_{|\xi|\sim 2^j} |\widehat{x^mf(x)}(\xi)|^2 \,d\xi
= 2^{-2j\alpha|m|} \int_{|\xi|\sim 2^j} |\widehat{ f_m}(\xi)|^2 \,d\xi   \lesssim 2^{-2j\alpha |m|}  2^{-2(\beta+\alpha)j}.
\end{align*}

\end{proof}

\subsubsection{Curvelet Analysis of a Smooth Fragment}

Let $J$ be a scale-angle pair and $\chi_0$ and $\chi_J$ the functions from \eqref{eq:suppfunctions}, used
in the construction of the $\alpha$-curvelet frame $\mathcal{C}_\alpha(W^{(0)},W,V)$. We remark, that $\chi_J$ is generally not a characteristic function.
However, $\chi_J$ is a non-negative real-valued function, supported in the wedges $\mathcal{W}_J$ given in \eqref{eq:wedgePJ} and satisfying $\|\chi_J\|_\infty\le1$.

Theorem~\ref{thm:localf0} directly leads to a central result, namely that it holds
\begin{align}\label{eq:cent0}
\int_{\R^2} \sum_{|J|=j} |(\widehat{f}\chi_J)(\xi)|^2 \,d\xi \lesssim 2^{-2j(\beta + \alpha)}.
\end{align}
Our next goal, is to refine this result.
Let us first record a basic fact.
\begin{lemma}\label{lem:basicfact}
Let $m\in\N_0^2$. It holds for all $\xi\in\R^2$
\[
\sum_{|J|=j} |\partial^m \chi_J(\xi)|^2 \lesssim 2^{-2j\alpha|m|}.
\]
\end{lemma}
\begin{proof}
From the definition it follows that $\chi_J$ scales with $2^{-\alpha j}$ in one direction
and with  $2^{-j}$ in the orthogonal direction. No matter what direction, we always do better
than $|\partial^m \chi_J(\xi)|^2 \lesssim 2^{-2j\alpha|m|} $. For fixed $\xi$ only a fixed number of summands are not zero, uniformly for all $\xi$.
The claim follows.
\end{proof}

Next we prove an auxiliary lemma. Here $\Delta=\partial_1^2+\partial_2^2$ denotes the standard Laplacian.

\begin{lemma}\label{lem:laplace}
It holds for $m\in\N_0$
\[
\int \sum_{|J|=j} |\Delta^m (\widehat{f}\chi_J)(\xi)|^2 \,d\xi \lesssim 2^{-2j(\beta + \alpha)}\cdot 2^{-4m\alpha j}.
\]
\end{lemma}
\begin{proof}
For $m=0$ this is just \eqref{eq:cent0}, a direct consequence of Theorem~\ref{thm:localf0}.
Now let $m>0$. It holds with $a,b\in\N_0^2$
and certain coefficients $c_{a,b}\in\N_0$
\begin{align*}
\Delta^m (\widehat{f}\chi_J)(\xi) = \sum_{|a|+|b|=2m} c_{a,b} \partial^a\widehat{f}(\xi) \partial^b\chi_J(\xi).
\end{align*}
Let $a,b\in\N_0^2$ such that $|a|+|b|=2m$.
Then with Lemma~\ref{lem:basicfact} and Corollary~\ref{cor:localf0}
\begin{align*}
\int \sum_{|J|=j} |\partial^a\widehat{f}(\xi)|^2 |\partial^b\chi_J(\xi)|^2 \,d\xi &\lesssim 2^{-2j \alpha |b|}  \int_{|\xi|\sim 2^j} |\partial^a\widehat{f}(\xi)|^2 \,d\xi \\
&\lesssim 2^{-2j \alpha |b|} 2^{-2j\alpha |a|}  2^{-2(\beta+\alpha)j} \\
& = 2^{-2j\alpha (|a|+|b|)}   2^{-2(\beta+\alpha)j} = 2^{-4j\alpha m}   2^{-2(\beta+\alpha)j}.
\end{align*}
\end{proof}

\noindent
Now we come to the refinement of \eqref{eq:cent0}.
For that, we need the differential operator
\[
\mathcal{L}=\mathcal{I}- 2^{2\alpha j}\Delta,
\]
where $\mathcal{I}$ is the identity and $\Delta$ the standard Laplacian.
The theorem below shows that $\mathcal{L}^2(\widehat{f}\chi_J)$ obeys the same estimate \eqref{eq:cent0} as $\widehat{f}\chi_J$.

\begin{theorem}\label{thm:central0}
It holds
\[
\int_{\R^2} \sum_{|J|=j} |\mathcal{L}^2 (\widehat{f}\chi_J)(\xi)|^2 \,d\xi \lesssim 2^{-2j(\beta + \alpha)}
\]
\end{theorem}
\begin{proof}
It holds
\[
\mathcal{L}^2= \mathcal{I}- 2\cdot 2^{2\alpha j}\Delta +  2^{4\alpha j}\Delta^2.
\]
Applying \eqref{eq:cent0} and Lemma~\ref{lem:laplace} yields the desired result.
\end{proof}

Finally we can give the proof of Proposition~\ref{prop:sequence2}.

\begin{proof}[Proof of Proposition~\ref{prop:sequence2}]
Recall the curvelet frame $\mathcal{C}_\alpha(W^{(0)},W,V)=(\psi_\mu)_{\mu\in M}$. On the Fourier side
\[
\widehat{\psi}_{j,\ell,k}=\chi_J u_{j,k}(R_J\cdot),
\]
with rotation matrix $R_J$ given as in \eqref{eq:matrixrot} and functions
\[
u_{j,k}(\xi)=2^{-j(1+\alpha)/2}e^{2\pi i (2^{-j}k_1,2^{-\alpha j}k_2) \cdot \xi}, \quad\xi\in\R^2.
\]
The curvelet coefficients $(\theta_\mu)_{\mu\in M}$ of $f=f_Q$ are therefore given by the formula
\begin{align*}
\theta_\mu = \langle f, \psi_{j,\ell,k} \rangle = \int_{\R^2} \widehat{f}\chi_J(\xi) \overline{u_{j,k}(R_J\xi)} \,d\xi.
\end{align*}
We have to study the decay of the subsequence $\theta=\theta_{Q}$ defined in \eqref{eq:Qsequence}. We observe
\begin{align*}
\mathcal{L} u_{j,k} = (1 + 2^{-2j(1-\alpha)}k_1^2 + k_2^2 ) u_{j,k},
\end{align*}
which also holds for the rotated versions $u_{j,k}(R_J\cdot)$. Partial integration thus yields
\begin{align*}
\theta_{\mu} = \int_{\R^2} \widehat{f}(\xi) \chi_J(\xi) \overline{u_{j,k}(R_J \xi)} \,d\xi
=  (1 + 2^{-2j(1-\alpha)}k_1^2 + k_2^2 )^{-2} \int_{\R^2} \mathcal{L}^2( \widehat{f}\chi_J ) \overline{u_{j,k}(R_{J} \xi)} \,d\xi.
\end{align*}
For $K=(K_1,K_2)\in\Z^2$ we define the set
\[
\mathfrak{Z}_{K}:=\Big\{ (k_1,k_2) \in \Z^2 ~:~ k_1 2^{-j(1-\alpha)} \in [K_1,K_1+1),\, k_2 = K_2 \Big\}.
\]
Further, we put
\[
M_{j,K}:=\Big\{ \mu=(j,\ell,k) \in M_j ~:~ k\in \mathfrak{Z}_K \Big\},
\]
where $M_j$ denotes the curvelet indices at scale $j$.
It follows from the orthogonality properties of the Fourier system $(u_{j,k})_{k\in\Z^2}$ that
\[
\sum_{k\in \mathfrak{Z}_K} |\theta_\mu|^2 \le (1+|K|^2)^{-4} \int_{\R^2}  |\mathcal{L}^2(\widehat{f}\chi_J)(\xi)|^2 \,d\xi,
\]
where $|K|^2=K_1^2+K_2^2$. We further conclude
\begin{align*}
\sum_{M_{j,K}} |\theta_\mu|^2 = \sum_{|J|=j} \sum_{k\in \mathfrak{Z}_K} |\theta_\mu|^2
\le (1+|K|^2)^{-4} \int_{\R^2}  \sum_{|J|=j} |\mathcal{L}^2(\widehat{f}\chi_J)(\xi)|^2 \,d\xi.
\end{align*}
Now we apply Theorem~\ref{thm:central0} and obtain
\[
\sum_{M_{j,K}} |\theta_\mu|^2 \lesssim 2^{-2j(\beta + \alpha)}(1+|K|^2)^{-4}  ,
\]
which directly implies
\begin{align}\label{eq:haha}
\big\|(\theta_\mu)_{\mu\in M_{j,K}} \big\|_{\ell^2}  \lesssim 2^{-j(\beta+\alpha)} (1+|K|^2)^{-2}.
\end{align}
It holds $\# \mathfrak{Z}_{K} \le 1+ 2^{j(1-\alpha)}$ and therefore, since $L_j=2^{\lfloor j(1-\alpha)\rfloor}$, the estimate $\# M_{j,K} \le 2\cdot 2^{2j(1-\alpha)}$.
Now we recall the interpolation inequality $\|(c_\lambda)\|_{\ell_p} \le n^{1/p-1/2} \|(c_\lambda)\|_{\ell_2} $ for a finite sequence $(c_\lambda)_\lambda$
with $n$ nonzero entries. Applying this inequality with $p=2/(1+\beta)$ we get from \eqref{eq:haha}
\begin{align*}
\big\|(\theta_\mu)_{\mu\in M_{j,K}} \big\|_{\ell^{2/(1+\beta)}}
\lesssim 2^{j(\beta-1)} \big\|(\theta_\mu)_{\mu\in M_{j,K}} \big\|_{\ell^{2}}
\le 2^{-j(1+\alpha)} (1+|K|^2)^{-2}.
\end{align*}
It follows
\[
\sum_{\mu\in M_{j,K}} |\theta_\mu|^{2/(1+\beta)}
\lesssim  2^{-j(1+\alpha)2/(1+\beta)}\cdot (1+|K|^2)^{-4/(1+\beta)}
=2^{-2\alpha j} (1+|K|^2)^{-4/(1+\beta)}.
\]
Finally, we have
\begin{align*}
\sum_{\mu\in M_j} |\theta_\mu|^{2/(1+\beta)} = \sum_{K\in \Z^2} \sum_{\mu\in M_{j,K}} |\theta_\mu|^{2/(1+\beta)}
\le 2^{-2\alpha j} \sum_{K\in \Z^2} (1+|K|^2)^{-4/(1+\beta)} \lesssim  2^{-2\alpha j}.
\end{align*}
The desired estimate for $\theta=\theta_Q$ follows, i.e.
\begin{align*}
\|\theta_Q\|_{\ell_{2/(1+\beta)}}  \lesssim 2^{-j(\alpha+1)}.
\end{align*}

\end{proof}

\noindent The following subsection is devoted to the proof of Proposition~\ref{prop:sequence1}.



\subsection{Analysis of an Edge Fragment}\label{ss:EdgeFragment}

Let us turn to the more complicated case $Q\in\mathcal{Q}^0_j$ and the proof of Proposition~\ref{prop:sequence1}. In this case
the cube $Q$ intersects the edge curve $\Gamma$ and $f_Q=f\omega_Q$ is accordingly called an \emph{edge fragment}.

In order to prove Proposition~\ref{prop:sequence1} we need to analyze the decay of
the sequence $\theta_Q=(\langle f_Q, \psi_\mu\rangle)_{\mu\in M_j}=(\langle \widehat{f}_Q, \widehat{\psi}_\mu\rangle)_{\mu\in M_j}$.
To estimate these scalar products, we again study the localization of the function $\widehat{f}_Q$.
As in the treatment of the smooth fragments, our investigation starts with some simplifying reductions.

Firstly, we note that it suffices to prove Proposition~\ref{prop:sequence1}
for a function $f\in \mathcal{E}^{\beta}(\R^2)$ of the form $f=g \chi_{\mathcal{B}}$ with $g\in C_0^{\beta}(\R^2)$.
In fact, the curvelet coefficient sequence of a general cartoon $f=f^{(0)}+f^{(1)}=g_0  \cdot \chi_{\mathcal{B}} + g_1$
can be decomposed into $\theta^{(0)}_Q=(\langle f^{(0)}_Q, \psi_\mu\rangle)_{\mu\in M_j}$ and $\theta^{(1)}_Q=(\langle f^{(1)}_Q, \psi_\mu\rangle)_{\mu\in M_j}$.
From Proposition~\ref{prop:sequence2} we already know
\[
\| \theta^{(1)}_Q \|_{\omega\ell_{2/(1+\beta)}}\lesssim 2^{-(1+\alpha)j} \lesssim 2^{-(1+\alpha)j/2}\,.
\]
Therefore it only remains to show $\| \theta^{(0)}_Q \|_{\omega\ell_{2/(1+\beta)}}\lesssim 2^{-(1+\alpha)j/2}$.
Since $\mathcal{B}\subset[-1,1]^2$, we can further smoothly cut off $g_0$ outside of $[-1,1]^2$ to obtain a function $g\in C_0^{\beta}(\R^2)$ such
that $f^{(0)}=g \chi_{\mathcal{B}}$.

Secondly, without loss of generality we restrict to the following model situation.
The cube $Q$ is centered at the origin
and the edge curve $\Gamma$ is the graph of a function $E:[-2^{-j\alpha},2^{-j\alpha}]\rightarrow[-2^{-j\alpha},2^{-j\alpha}]$ belonging to $C^\beta(\R)$, with $x_1=E(x_2)$.
Further, it shall hold $E(0)=E^\prime(0)=0$, so that $\Gamma$ approximates a vertical line through the origin.
If the scale $j$ is big enough, say bigger than some fixed base scale $j_0\in\N_0$, it is always possible to arrive at this setting
by possibly translating or rotating the coordinate axes.
Henceforth, we assume $j\in\N$ and $j\ge j_0$ which clearly
poses no loss of generality.

In this simplified model situation the edge fragment $f=f_Q$ can be written in the form
\begin{align*}
f(x)=\omega(2^{\alpha j}x)g(x)\chi_{\{x_1\ge E(x_2)\}}, \qquad x=(x_1,x_2)\in\R^2,
\end{align*}
where $g\in C_0^{\beta}(\R^2)$, and $\omega\in C^\infty_0(\R^2)$ is the nonnegative window with $\supp \omega\subset [-1,1]^2$, generating the partition of unity $(\omega_Q)_Q$.
For our investigation it is again more convenient to work with a rescaled version $F$ of the edge fragment. Therefore we put $g_j=g(2^{-\alpha j}\cdot)$ and define
\begin{align}\label{eq:functF}
F(x):=\omega(x)g_j(x)\chi_{\{x_1\ge E_j(x_2)\}},\quad x=(x_1,x_2)\in\R^2,
\end{align}
with the rescaled edge function
\[
E_j:[-1,1]\rightarrow[-1,1]~,~ E_j(x_2)=2^{\alpha j} E(2^{-\alpha j}x_2).
\]
It holds $E_j\in C^{\beta}([-1,1])$ with
$E_j^\prime=E^\prime(2^{-\alpha j}\cdot)$ and
$\Hol(E_j^\prime,\beta-1)\le\delta_j$,
where
\[
\delta_j=2^{-j(1-\alpha)}\cdot \Hol(E^\prime,\beta-1).
\]
Observe that $\Hol(E^\prime,\beta-1)$ is a constant independent of the scale $j$.
Together with $E_j(0)=E^\prime_j(0)=0$ this implies for all $u\in[-1,1]$ that
\begin{align}\label{eq:vertstrip}
|E_j(u)|\le\delta_j \quad\text{and}\quad |E^\prime_j(u)|\le\delta_j.
\end{align}

For convenience, we continuously extend the function $E_j$ to the whole of $\R$ by attaching straight lines
on the left and on the right, with constant slopes $E_j^\prime(1)$ and $E_j^\prime(-1)$ respectively. Since this extension occurs outside
of the square $[-1,1]^2$, it does not change the representation \eqref{eq:functF} of the edge fragment. Furthermore, it also does not alter
the regularity and the Hölder constant.

\subsubsection{Fourier Analysis of an Edge Fragment}

Our first goal is to analyze the Fourier transform $\widehat{F}$ (and thus also $\widehat{f}$) along radial lines, whose orientations are specified by
angles $\eta\in[-\pi/2,\pi/2]$ with respect to the $x_1$-axis.
If the angle $\eta$ satisfies $|\sin\eta|>\delta_j$, it is possible because of \eqref{eq:vertstrip} to define a function $u=u_j(\cdot,\eta):\R\rightarrow\R$ implicitly by
\begin{align}\label{eq:u_def}
E_j(u(t))\cos\eta + u(t)\sin\eta=t.
\end{align}
The value $u(t)$ is the $x_2$-coordinate of the intersection point of the (extended) edge curve $\Gamma$ and the line $\mathfrak{L}_{t,\eta}$ defined by
\begin{align}\label{eq:intline}
\mathfrak{L}_{t,\eta}:=\Big\{ x=(x_1,x_2)\in\R^2 ~:~ x_1\cos\eta + x_2\sin\eta=t \Big\}.
\end{align}
Further, we can define the function $a=a_j(\cdot,\eta):\R\rightarrow\R$ by
\begin{align}\label{eq:a_def}
a(t):=-E_j(u(t))\sin\eta + u(t)\cos\eta.
\end{align}
The value $a(t)$ is the $x_2$-coordinate of the point $(E_j(u),u)^T\in\Gamma$ in the coordinate system rotated by the angle $\eta$.
For an illustration we refer to Figure~\ref{fig:edgefragment}.

\begin{figure}[t]
 \begin{center}
  \includegraphics[width=0.35\textwidth]{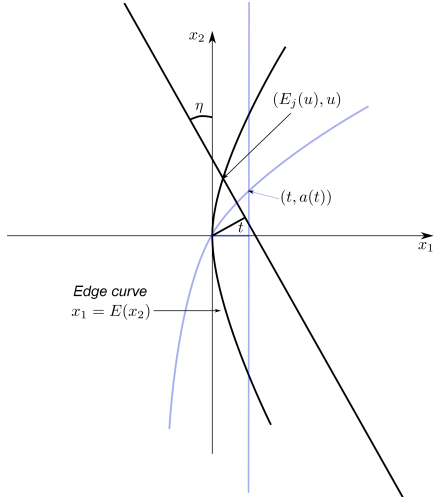}
  \label{fig:edgefragment}
  \caption{Illustration of standard edge fragment.}
  \end{center}
\end{figure}

The functions $u$ and $a$ are strictly monotone, increasing if $\eta>0$ and decreasing if $\eta<0$.
Note, that we suppressed the dependence of $u$ and $a$ on $j$ in the notation.
The following lemma studies the regularity of $u$ under the assumption $|\sin\eta|\ge 2\delta_j$.

\begin{lemma}
Assume $|\sin\eta|\ge 2\delta_j$. Then the function $u:\R\rightarrow\R$ defined implicitly by \eqref{eq:u_def}
belongs to $C^{\beta}(\R)$. Moreover, we have $\|u^\prime\|_\infty\lesssim |\sin\eta|^{-1}$
and
\[
\| \Delta_h u^\prime \|_\infty \lesssim \delta_j  h^{\beta-1} |\sin\eta|^{-1-\beta},
\]
where the implicit constants are independent of the scale $j$, the angle $\eta$, and $h\ge0$.
\end{lemma}
\begin{proof}
First of all it is not difficult to show that $u=u_j(\cdot,\eta)\in C^1(\R)$ with
\[
u^\prime(t)=\big(\sin\eta + E_j^\prime(u(t))\cos\eta\big)^{-1}.
\]
Under the assumption $|\sin\eta|\ge 2\delta_j$ it follows $\|u^\prime\|_\infty\lesssim |\sin\eta|^{-1}$
because of $|E^\prime_j(u)|\le\delta_j\le \frac{1}{2} |\sin\eta|$ for all $u\in[-1,1]$.
Finally, we examine $\Delta_h u^\prime$. For $t\in\R$
\begin{align*}
\Delta_h u^\prime(t)&=u^\prime(t+h)-u^\prime(t)=u^\prime(t+h)u^\prime(t)\big( u^\prime(t)^{-1} - u^\prime(t+h)^{-1} \big) \\
&= u^\prime(t+h)u^\prime(t) \cos\eta \big( E_j^\prime(u(t)) - E_j^\prime(u(t+h)) \big).
\end{align*}
Using $\Hol(E_j^\prime,\beta-1)\le\delta_j$ and the mean value theorem leads to
\begin{align*}
\|\Delta_h u^\prime\|_\infty\le \|u^\prime\|_\infty^2 \delta_j \|\Delta_h u\|_\infty^{\beta-1}
\le   \|u^\prime\|_\infty^{\beta +1} \delta_j   h^{\beta-1} \lesssim \delta_j h^{\beta-1} |\sin\eta|^{-1-\beta}  .
\end{align*}

\end{proof}

The following lemma collects some properties of the function $a:\R\rightarrow\R$ defined in \eqref{eq:a_def}.

\begin{lemma}\label{lem:propa}
Assume $|\sin\eta|\ge 2\delta_j$. It holds $a\in C^{\beta}(\R)$ with
\begin{align*}
\|a^\prime\|_\infty\lesssim |\sin\eta|^{-1}, \quad \|\Delta_h a\|_\infty\lesssim h|\sin\eta|^{-1}, \quad \|\Delta_h a^\prime\|_\infty\lesssim \delta_j h^{\beta-1} |\sin\eta|^{-1-\beta},
\end{align*}
with implicit constants independent of $j$, $\eta$, and $h\ge0$.
\end{lemma}
\begin{proof}
This is an easy consequence of the properties of $u$ proved in the previous lemma.
\end{proof}

\noindent
Next, we introduce the scale-dependent interval
\begin{align}\label{eq:interval}
I(\eta):=I_j(\eta):=[a_j(\eta),b_j(\eta)],
\end{align}
where $a_j(\eta)=E_j(-1)\cos\eta - \sin\eta$ and $b_j(\eta)=E_j(1)\cos\eta + \sin\eta$.
The restrictions of $u$ and $a$ to $I(\eta)$ correspond precisely to that part of the edge curve $\Gamma$ lying inside the square $[-1,1]^2$.
In particular, we have a bijection $u:I(\eta)\rightarrow[-1,1]$.
\ms{In the sequel it is more convenient to work with an extension of $I(\eta)$, given by
\[
\widetilde{I}(\eta):=\widetilde{I}_j(\eta):=[a_j(\eta)-C\delta_j, b_j(\eta)+C\delta_j],
\]
for some suitable fixed constant $C>0$.}

\begin{lemma}\label{lem:estinterval}
For $|\sin\eta|>\delta_j$ we have
\[
|I(\eta)|\lesssim |\sin\eta| \quad\text{and}\quad |\widetilde{I}(\eta)|\lesssim |\sin\eta|.
\]
\end{lemma}
\begin{proof}
In view of $|\sin\eta|>\delta_j$ and \eqref{eq:vertstrip} we can estimate
\begin{align*}
|I(\eta)|\le |E_j(1)-E_j(-1)| |\cos\eta| + 2 |\sin\eta| \le 2\delta_j + 2 |\sin\eta| \lesssim |\sin\eta|.
\end{align*}
The estimate for $\widetilde{I}(\eta)$ then follows directly from $|\sin\eta|>\delta_j$.
\end{proof}


We want to analyze $\widehat{F}$ along lines through the origin with orientation $\eta\in[-\pi/2,\pi/2]$.
The central tool in this investigation is the Fourier slice theorem.
In view of this theorem it makes sense to first study the Radon transform $\mathcal{R}F$, in particular its regularity. By Paley-Wiener type arguments
we can then later extract information about the decay of $\widehat{F}$. This is the same approach taken in \cite{CD04}. Due to the lack
of regularity in our case, however, we have to use a more refined technique in this investigation. The main idea is to
use finite differences instead of derivatives.

The value $\mathcal{R}F(t,\eta)$ of the Radon transform is obtained by integrating $F$ along the line $\mathfrak{L}_{t,\eta}$ defined in \eqref{eq:intline}.
Rotating $F$ by the angle $\eta$ yields the function $F^\eta$ and we can write
\[
(\mathcal{R}F)(t,\eta)
= \int_{\R} F^\eta(t,u)\,du. 
\]
The rescaled edge fragment $F$ can be rewritten as the product $F=G\chi_{\{x_1\ge E_j(x_2)\}}$ with the function
\begin{align*}
G:=\omega g(2^{-\alpha j}\cdot)=\omega g_j.
\end{align*}
Then we have
\begin{align}\label{eq:RepRadon}
(\mathcal{R}F)(t,\eta)=\int_{-\infty}^{a(t,\eta)} G^\eta(t,u) \,du,
\end{align}
where $G^\eta$ is the function obtained by rotating $G$ by the angle $\eta$. Using the notation
$g_j^\eta$ and $\omega^\eta$ for the rotated versions of $g_j$ and $\omega$, the integrand of \eqref{eq:RepRadon} takes the form
\[
G^\eta=g_j^\eta\omega^\eta.
\]
We see, that the component $g^\eta_j=g^\eta(2^{-\alpha j}\cdot)\in C_0^{\beta}(\R^2)$ of $G^\eta$ is scaled and the window $\omega^\eta\in C_0^\infty(\R^2)$ remains fixed.
During the following investigation the angle $\eta$ remains constant. Therefore, we can simplify the notation by omitting the index $\eta$.

The central lemma of this subsection is given below. Its proof relies on estimates
of the functions $g^\eta_j$ and $\omega^\eta$, outsourced to the appendix.

\begin{lemma}\label{lem:central1}
Assume that $|\sin\eta|\ge 2\delta_j$. For $h=C2^{-(1-\alpha)j}$, where $C>0$ is some fixed constant, we then have
\[
\Delta_h\partial_1\mathcal{R}F(t,\eta)=S_1(t,\eta) + S_2(t,\eta)
\]
with
\begin{align*}
\| S_1(\cdot,\eta) \|^2_{2} &\lesssim \delta_j^2 h^{2(\beta-1)} |\sin\eta|^{-1-2\beta}, \\
\| \Delta_{(h,0)} S_2(\cdot,\eta) \|^2_{2} &\lesssim h^{2\beta} |\sin\eta|^{-1-2\beta},
\end{align*}
where the implicit constants are independent of the scale $j$ and the angle $\eta$.
\end{lemma}
\begin{proof}
    %
    First we differentiate $\mathcal{R}F(t,\eta)$ with respect to $t$ and obtain from \eqref{eq:RepRadon}
    \[
    \partial_1(\mathcal{R}F)(t,\eta)=a^\prime(t)  G(t,a(t)) + \int_{-\infty}^{a(t)} \partial_1G(t,u) \,du=:T(t)
    \]
    where on the right-hand side the dependence on $\eta$ is omitted in the notation.
    Applying $\Delta_h$ then yields for $t\in\R$
    \begin{align*}
    \Delta_hT(t)&=\Delta_h a^\prime(t)  G(t+h,a(t+h))+  a^\prime(t) \Delta_h G(t,a(t))  + \int_{a(t)}^{a(t+h)} \partial_1G(t+h,u) \,du
    + \int_{-\infty}^{a(t)} \Delta_{(h,0)}  \partial_1G(t,u) \,du \\
    &=: T_1(t)+T_2(t)+T_3(t)+T_4(t).
    \end{align*}
    Next, we estimate the $L^\infty$-norms of the functions $T_i$ for $i\in\{1,2,3,4\}$.
    Let us begin with $T_1$. Applying Lemma~\ref{lem:propa} we obtain
    \begin{align*}
    \|T_1\|_\infty \le \|\Delta_h a^\prime\|_\infty \| G \|_\infty \lesssim \|\Delta_h a^\prime\|_\infty \lesssim \delta_j h^{\beta-1} |\sin\eta|^{-1-\beta} \lesssim  h^{\beta} |\sin\eta|^{-1-\beta}.
    \end{align*}
    %
    %
    The estimate of $T_2$ takes some more effort. The product rule yields for $t\in\R$
    \begin{align*}
    T_2(t) =  a^\prime(t) \Delta_h G(t,a(t)) = a^\prime(t) \Delta_h g_j(t,a(t)) \omega(t+h,a(t+h)) + a^\prime(t) g_j(t,a(t)) \Delta_h \omega(t,a(t)) =: T_{21}(t) + T_{22}(t).
    \end{align*}
    Using the mean value theorem and Lemmas~\ref{lem:propa} and \ref{lemapp:edgeg} yields
    \begin{align*}
    \|T_{21} \|_\infty \le \|a^\prime\|_\infty  \sup_{t\in\R}|\Delta_h g_j(t,a(t))| \|\omega\|_\infty
    \lesssim h |\sin\eta|^{-2} 2^{-\alpha j} \lesssim  h^{\beta} |\sin\eta|^{-1-\beta}.
    \end{align*}
    We take another forward difference of the component $T_{22}$ and obtain
    \beqn
    \Delta_h T_{22}(t) &=& \Delta_h a^\prime(t) g_j(t+h,a(t+h)) \Delta_h \omega(t+h,a(t+h)) \\
    && + a^\prime(t) \Delta_h g_j(t,a(t)) \Delta_h \omega(t+h,a(t+h)) + a^\prime(t) g (t,a(t)) \Delta^2_h \omega(t,a(t)) \\
    &=:& T^1_{22}(t) +  T^2_{22}(t) +  T^3_{22}(t).
    \eeqn
    These terms allow the following estimates, where we use Lemmas~\ref{lem:propa}, \ref{lemapp:edgew} and \ref{lemapp:edgeg}.
    Also note $h\lesssim |\sin\eta|$.
    \begin{align*}
    \|T^1_{22}\|_\infty &\le h^{\beta+1} |\sin\eta|^{-2-\beta} \lesssim  h^{\beta} |\sin\eta|^{-1-\beta} , \\
    \|T^2_{22}\|_\infty &\le h^2 |\sin\eta|^{-3}  2^{-\alpha j} \lesssim h^{\beta} |\sin\eta|^{-1-\beta} ,  \\
    \|T^3_{22}\|_\infty &\le  h^2 |\sin\eta|^{-3} +  h^{\beta+1}  |\sin\eta|^{-2-\beta}   \lesssim h^{\beta} |\sin\eta|^{-1-\beta} .
    \end{align*}
    %
    %
    By substitution the term $T_3$ transforms to
    \beqn
    T_3(t)&=&\int_{t}^{t+h} \partial_1G(t+h,a(u)) a^\prime(u) \,du \\
       &=&\int_{t}^{t+h} \partial_1g_j(t+h,a(u))\omega(t+h,a(u)) a^\prime(u) \,du + \int_{t}^{t+h} g_j(t+h,a(u)) \partial_1\omega(t+h,a(u))a^\prime(u) \,du \\
       &=:& T_{31}(t) + T_{32}(t)
    \eeqn
    We apply $\Delta_h$ to $T_{31}$. Here $\Delta_h$ acts exclusively on $t$ and $\tau$. We obtain
    \beqn
    \Delta_h T_{31}(t) &=&  \int_{t}^{t+h} \Delta_h\big[ \partial_1g_j(t+h,a(\tau))\omega(t+h,a(\tau)) a^\prime(\tau) \big] \,d\tau \\
    &=&\int_{t}^{t+h} \Delta_h\partial_1g_j(t+h,a(\tau))\omega(t+2h,a(\tau+h)) a^\prime(\tau+h)\,d\tau \\
    &&+ \int_{t}^{t+h} \partial_1g_j(t+h,a(\tau)) \Delta_h\omega(t+h,a(\tau)) a^\prime(\tau+h) \,d\tau \\
    &&+ \int_{t}^{t+h} \partial_1g_j(t+h,a(\tau))\omega(t+h,a(\tau))  \Delta_ha^\prime(\tau) \,d\tau  \\
    &=:&  T^1_{31}(t)+ T^2_{31}(t) + T^3_{31}(t).
    \eeqn
    Analogously, we decompose
    \begin{align*}
    \Delta_h T_{32}(t) =  \int_{t}^{t+h} \Delta_h\big[ g_j(t+h,a(\tau)) \partial_1\omega(t+h,a(\tau))a^\prime(\tau) \big] \,d\tau
    =:  T^1_{32}(t)+ T^2_{32}(t) + T^3_{32}(t).
    \end{align*}
    Then we estimate with the results from the appendix
    \begin{align*}
    \|T^1_{31}\|_\infty&\lesssim h |\sin\eta|^{-1} 2^{-j} \big(   h^{\beta-1} + h^{\beta-1} |\sin\eta|^{1-\beta} \big) \lesssim  h^{\beta} |\sin\eta|^{-1-\beta},  \\
    \|T^2_{31}\|_\infty&\lesssim h |\sin\eta|^{-1} 2^{-\alpha j} \big(   h + h |\sin\eta|^{-1} \big) \lesssim  h^{\beta} |\sin\eta|^{-1-\beta},  \\
    \|T^3_{31}\|_\infty&\lesssim h  2^{-\alpha j} \|\Delta_h a^\prime\|_\infty \lesssim  h^{\beta} |\sin\eta|^{-1-\beta} ,
    \end{align*}
    and
    \begin{align*}
    \|T^1_{32}\|_\infty&\lesssim h |\sin\eta|^{-1} 2^{-\alpha j} \big(   h + h|\sin\eta|^{-1} \big) \lesssim  h^{\beta} |\sin\eta|^{-1-\beta}, \\
    \|T^2_{32}\|_\infty&\lesssim h |\sin\eta|^{-1} \big(   h + h |\sin\eta|^{-1} \big) \lesssim  h^{\beta} |\sin\eta|^{-1-\beta},  \\
    \|T^3_{32}\|_\infty&\lesssim h \|\Delta_h a^\prime\|_\infty \lesssim  h^{\beta} |\sin\eta|^{-1-\beta}.
    \end{align*}


    Finally, we treat the term $T_4$,
    \beqn
    T_4(t)&=&\int_{-\infty}^{a(t)} \Delta_h \partial_1G(t,u) \,du =\int_{-\infty}^{a(t)} \Delta_h \big(\partial_1g_j(t,u)\omega(t,u) + g_j(t,u)\partial_1\omega(t,u)\big) \,du \\
    &=& \int_{-\infty}^{a(t)} \Delta_h \partial_1g_j(t,u) \omega(t+h,u) \,du
    + \int_{-\infty}^{a(t)} \big( \partial_1g_j(t,u) \Delta_h \omega(t,u) + \Delta_h g_j(t,u) \partial_1\omega(t+h,u) \big)\,du \\
    && + \int_{-\infty}^{a(t)} g_j(t,u)  \Delta_h\partial_1\omega(t,u) \,du
    =: T_{41}(t) +  T_{42}(t) +  T_{43}(t).
    \eeqn
    The terms $T_{41}$ and $T_{42}$ can be estimated directly,
    \begin{align*}
    \|T_{41} \|_\infty &\lesssim h^{\beta-1}\cdot 2^{-j} \le  h^{\beta} , \\
    \|T_{42} \|_\infty &\lesssim h\cdot 2^{-\alpha j}\asymp 2^{-j}\le 2^{-j(\beta-1)}\asymp h^{\beta}.
    \end{align*}
    The term $T_{43}$ again needs some further preparation,
    \begin{align*}
    \Delta_h T_{43}(t)= \int_{a(t)}^{a(t+h)} g_j(t+h,u)\Delta_h\partial_1\omega(t+h,u) \,du
    +  \int_{-\infty}^{a(t)} \Delta_h\big(g_j(t,u)\Delta_h\partial_1\omega(t,u)\big) \,du =: T^1_{43}(t) + T^2_{43}(t).
    \end{align*}
    In the end we arrive at
    \begin{align*}
    \|T^1_{43}\|_\infty &\lesssim h^2 |\sin\eta|^{-1}\lesssim h^{\beta} |\sin\eta|^{-1},  \\
    \|T^2_{43}\|_\infty &\lesssim h^2 \lesssim h^{\beta} .
    \end{align*}
Now we collect the appropriate terms and add them up to obtain $S_1$ and $S_2$. In a last step, we use our $L^\infty$-estimates
to obtain the desired $L^2$-estimates. Here we use that $|\supp T_i|\lesssim |I(\eta)| \lesssim |\sin\eta|$ according to Lemma~\ref{lem:estinterval}
for $i\in\{1,2,3\}$ and $|\supp T_4|\lesssim 1$.
This finishes the proof.
\end{proof}

The previous lemma is the key to the following theorem.
The notation $|\lambda|\sim 2^{(1-\alpha) j}$ indicates that $|\lambda|\in [C_12^{(1-\alpha) j},C_22^{(1-\alpha) j}]$ for fixed constants $0<C_1\le C_2<\infty$.
A typical choice would be $C_1=1$ and $C_2=2^{1-\alpha}$.

\begin{theorem}\label{thm:mainest1}
It holds
\begin{align*}
\int_{|\lambda|\sim 2^{(1-\alpha)j}} |\widehat{F}(\lambda\cos\eta, \lambda\sin\eta)|^2 \,d\lambda \lesssim 2^{-(1-\alpha)j}\Big( 1+ 2^{(1-\alpha)j}|\sin\eta| \Big)^{-1-2\beta}
\end{align*}
with an implicit constant independent of $j$ and $\eta$.
\end{theorem}
%

\begin{proof}
First we assume $|\sin\eta|\ge 2\delta_j$.
The integration domain is given by $[C_12^{(1-\alpha)j},C_22^{(1-\alpha)j}]$ for fixed constants $0<C_1\le C_2<\infty$.
Let us fix $h:=C2^{-j(1-\alpha)}$, where $C>0$ is chosen such that $C_2C<2\pi$.
For this choice of $h$ there is $c>0$ such that  $ |e^{i\lambda h}-1|^2\ge c $ for all $|\lambda|\in [C_12^{(1-\alpha)j},C_22^{(1-\alpha)j}]$ at all scales.
We conclude, where $S_i$, $i=1,2$,
denote the entities from Lemma~\ref{lem:central1}:
\begin{align*}
\int\limits_{|\lambda|\sim 2^{j(1-\alpha)}} |\lambda|^2 |\widehat{F}(\lambda\cos\eta,\lambda\sin\eta)|^2 \,d\lambda
&\lesssim \int\limits_{|\lambda|\sim 2^{j(1-\alpha)}} |e^{i\lambda h}-1|^2 |\lambda|^2 |\widehat{F}(\lambda\cos\eta,\lambda\sin\eta)|^2 \,d\lambda \\
&= \int\limits_{|\lambda|\sim 2^{j(1-\alpha)}} |\mathcal{F}\big[\Delta_h \partial_1 \mathcal{R}F(\cdot,\eta)\big](\lambda)|^2  \,d\lambda  \\
&\lesssim \sum_{i=1}^2 \int\limits_{|\lambda|\sim 2^{j(1-\alpha)}} |\widehat{S_i(\cdot,\eta)}(\lambda)|^2  \,d\lambda \\
&\lesssim \int\limits_{|\lambda|\sim 2^{j(1-\alpha)}} |e^{i\lambda h}-1|^2 |\widehat{S_2(\cdot,\eta)}(\lambda)|^2  \,d\lambda +
\int\limits_{|\lambda|\sim 2^{j(1-\alpha)}} |\widehat{S_1(\cdot,\eta)}(\lambda)|^2  \,d\lambda \\
&\le \int_{\R}  |e^{i\lambda h}-1|^2 |\widehat{S_2(\cdot,\eta)}(\lambda)|^2 \,d\lambda + \int_{\R} |\widehat{S_1(\cdot,\eta)}(\lambda)|^2  \,d\lambda  \\
&= \| S_1(\cdot,\eta) \|^2_{2} + \| \Delta_hS_2(\cdot,\eta) \|^2_{2}
\lesssim  h^{2\beta} |\sin\eta|^{-1-2\beta}.
\end{align*}
It follows
\begin{align*}
\int_{|\lambda|\sim 2^{j(1-\alpha)}} |\widehat{F}(\lambda\cos\eta,\lambda\sin\eta)|^2 \,d\lambda
\lesssim  2^{-j(1-\alpha)} \big(2^{j(1-\alpha)} |\sin\eta|\big)^{-1-2\beta} .
\end{align*}
Next, we handle the case $|\sin\eta|< 2\delta_j$. We want to show
\begin{align}\label{eq:estpart2}
\int_{|\lambda|\sim 2^{j(1-\alpha)}} |\widehat{F}(\lambda\cos\eta,\lambda\sin\eta)|^2 \,d\lambda \lesssim  2^{-j(1-\alpha)}.
\end{align}
Altogether we then obtain the desired estimate
\begin{align*}
\int_{|\lambda|\sim 2^{(1-\alpha)j}} |\widehat{F}(\lambda\cos\eta, \lambda\sin\eta)|^2 \,d\lambda \lesssim 2^{-(1-\alpha)j}\Big( 1+ 2^{(1-\alpha)j}|\sin\eta| \Big)^{-1-2\beta},
\end{align*}
since $2^{(1-\alpha)j}|\sin\eta|\gtrsim 1$ if $|\sin\eta|\ge 2\delta_j$ and $2^{(1-\alpha)j}|\sin\eta|\lesssim 1$ if $|\sin\eta|< 2\delta_j$,

It remains to show \eqref{eq:estpart2}. For this we write the edge fragment as a sum $F=F_0+D$, where
\[
F_0(x)= g(2^{-j\alpha}x)\omega(x)\chi_{\{x_1\ge\delta_j\}},\quad x\in\R^2,
\]
is a fragment with a straight edge and $D(x)=F(x)-F_0(x)$ is the deviation.

The function $D$ is supported in a vertical strip around the $x_2$-axis of width $2\delta_j$.
For $\eta$ satisfying $|\sin\eta|<2\delta_j$ the Radon transform $\mathcal{R}D(\cdot,\eta)$ is $L^\infty$-bounded
and supported in an interval of length
\begin{align*}
2(\delta_j\cos\eta + \sin\eta)\lesssim \delta_j.
\end{align*}
It follows
$\|\mathcal{R}D(\cdot,\eta)\|^2_{2}\lesssim \delta_j \lesssim 2^{-j(1-\alpha)}$,
and therefore
\begin{align*}
\int_{|\lambda|\sim 2^{j(1-\alpha)}} |\widehat{D}(\lambda\cos\eta,\lambda\sin\eta)|^2 \,d\lambda
\le \int_{\R} |\widehat{\mathcal{R}D(\cdot,\eta)}(\lambda)|^2 \,d\lambda \lesssim  2^{-j(1-\alpha)}.
\end{align*}
Finally, the estimate
\begin{align*}
\int_{|\lambda|\sim 2^{j(1-\alpha)}} |\widehat{F_0}(\lambda\cos\eta,\lambda\sin\eta)|^2 \,d\lambda \lesssim  2^{-j(1-\alpha)}
\end{align*}
follows from the fact, that we have decay $|\widehat{F_0}(\lambda,0)|\sim |\lambda|^{-1/2}$ normal to the straight singularity curve, and that
the second argument ($\lambda\sin\eta$) remains bounded due to the condition $|\sin\eta|< 2\delta_j$.
This finishes the proof.
\end{proof}

\noindent
A direct consequence is the following corollary.

\begin{cor}\label{cor:localf}
We have
\begin{align*}
\int_{|\lambda|\sim 2^j} |\widehat{f}(\lambda\cos\eta, \lambda\sin\eta)|^2 \,d\lambda \lesssim 2^{-(1+2\alpha)j}\Big( 1+ 2^{(1-\alpha)j}|\sin\eta| \Big)^{-1-2\beta}.
\end{align*}
\end{cor}

\begin{proof}
The statement follows directly from the relation $\widehat{f}(\xi)=2^{-2\alpha j}\widehat{F}(2^{-\alpha j}\xi)$.
\end{proof}


\subsubsection{Refinement of the Previous Discussion}

In this subsection we will prove a refinement of Corollary~\ref{cor:localf}. For that we need
to analyze the modified edge fragment $\widetilde{F}$, given for fixed $m\in\N_0$ by
\begin{align}\label{eq:modF}
\widetilde{F}(x)=r(x)^{m}F(x), \quad x\in\R^2,
\end{align}
where the function $r:\R^2\rightarrow\R$ maps $x\in\R^2$ to its first component $x_1\in\R$. It can be written as the product $\widetilde{F}(x)= \widetilde{G}(x)\chi_{\{x_1\ge E_j(x_2)\}}$ with the function
\begin{align}\label{eq:functGF}
\widetilde{G}(x):=r(x)^m G(x) =r(x)^m \omega(x) g_j(x), \quad x\in\R^2.
\end{align}
Rotating by the angle $\eta$ yields $\widetilde{G}^\eta(x)= (r^\eta(x))^m G^\eta(x)$,
where $G^\eta$ and $r^\eta$ are the functions obtained by rotating $G$ and $r$, respectively. The function $r^\eta:\R^2\rightarrow\R$ then has the form
\[
r^\eta(t,a):=t\cos\eta - a \sin\eta, \quad (t,a)\in\R^2.
\]
Note, that since
$
G^\eta(x)=g_j^\eta(x)\omega^\eta(x)
$
we have $\widetilde{G}^\eta(x)=(r^\eta(x))^m g_j^\eta(x) \omega^\eta(x)$.
Some important properties of $r^\eta$ and $\widetilde{G}^\eta$ are collected in the appendix.

Similar to the investigation of $F$ in the previous subsection, the angle $\eta$ remains fixed in the remainder. We can therefore again simplify
the notation by omitting this index.
The following result generalizes Lemma~\ref{lem:central1}.
\begin{lemma}\label{lem:central2}
For $m\in\N_0$ let $\widetilde{F}$ be the modified edge fragment \eqref{eq:modF}.
Further, assume that $|\sin\eta|\ge 2\delta_j$ and $h\asymp 2^{-(1-\alpha)j}$. Then the function $S:=\Delta_h\partial_1\mathcal{R}\widetilde{F}(\cdot,\eta)$ admits a decomposition
 \begin{flalign*}
&& S&=S^0_1+S^0_2, &&\\
\text{such that} && \Delta_h S^0_2&= S^1_1+ S^1_2, &&\\
&&    \Delta_h S^1_2&= S^2_1+ S^2_2, &&\\
&&    &\vdots &&\\
&&    \Delta_h S^{m-1}_2&= S^{m}_1+ S^{m}_2, &&\\
&&    \Delta_h S^{m}_2&= S^{m+1}_1, &&
    \end{flalign*}
with the estimates
\begin{align*}
 \|S^{k}_1\|_{2}^2&\lesssim 2^{-2jm(1-\alpha)}  h^{2\beta} |\sin\eta|^{-1-2\beta} + 2^{-j(1-\alpha)(2\beta+1)},\qquad k=0,1,\ldots,m+1.
 \end{align*}
For convenience we set $S_2^{m+1}=0$.
\end{lemma}
We introduce the following language and say, that the function $S$ \emph{admits a decomposition $(S^k_1,S^k_2)_k$ of the form $(*)$
of length $m+1$ with the estimates}
\begin{align*}
 \|S^{k}_1\|_{2}^2&\lesssim 2^{-2jm(1-\alpha)}  h^{2\beta} |\sin\eta|^{-1-2\beta} + 2^{-j(1-\alpha)(2\beta+1)},\qquad k=0,1,\ldots,m+1.
\end{align*}
\begin{proof}
  We have $\widetilde{G}^{\eta}(x)=r^\eta(x)^{m}G^{\eta}(x)$ for $x\in\R^2$ and
  \[
  \mathcal{R}\widetilde{F}(t,\eta)= \int_{-\infty}^{a(t,\eta)} \widetilde{G}^{\eta}(t,u) \,du .
  \]
  For simplicity we omit the superindex $\eta$ subsequently.
  Similar to the proof of Lemma~\ref{lem:central1} we obtain
  \begin{align*}
    S(t)&=\Delta_h a^\prime(t)  \widetilde{G}(t+h,a(t+h))+  a^\prime(t) \Delta_h \widetilde{G}(t,a(t))  + \int_{a(t)}^{a(t+h)} \partial_1\widetilde{G}(t+h,u) \,du
    + \int_{-\infty}^{a(t)} \Delta_{(h,0)}  \partial_1\widetilde{G}(t,u) \,du \\
    &=: \widetilde{T}_1(t)+\widetilde{T}_2(t)+\widetilde{T}_3(t)+\widetilde{T}_4(t).
    \end{align*}
  We will show the assertion for each of these terms separately. Moreover, it suffices to prove $L^\infty$-estimates, which
  can be transformed to the desired $L^2$-estimates via the corresponding support properties.
  Note that $|\supp \widetilde{T}_i|\lesssim |\widetilde{I}(\eta)| \lesssim |\sin\eta|$ according to Lemma~\ref{lem:estinterval}
  for $i\in\{1,2,3\}$ and that $|\supp \widetilde{T}_4|\lesssim 1$.

  For $\widetilde{T}_1$ the estimate
    \beqn
    \|\widetilde{T}_1 \|_\infty &\le& \|\Delta_h a^\prime\|_\infty \sup_{t\in\R}| \widetilde{G}(t,a(t))| \lesssim \|\Delta_h a^\prime \|_\infty \sup_{t\in I(\eta)}|r(t,a(t))|^{m} \\
     &\lesssim& \delta^{m+1}_j h^{\beta-1} |\sin\eta|^{-1-\beta} \lesssim  h^m h^{\beta} |\sin\eta|^{-1-\beta}
    \eeqn
    is sufficient.
    %
    Next, we show that $\widetilde{T}_2$ and $\widetilde{T}_3$ admit decompositions $(\widetilde{T}^k_1,\widetilde{T}^k_2)_k$ of the form $(\ast)$
    of length $m+1$ with $\supp \widetilde{T}^k_i\subset\widetilde{I}(\eta)$, $i\in\{1,2\}$, and the estimates
    \begin{align*}
    \|\widetilde{T}^{k}_1\|_\infty&\lesssim h^{m}h^{\beta} |\sin\eta|^{-1-\beta}, \quad k=0,\ldots,m+1, \\
    \|\widetilde{T}^{k}_2\|_\infty&\lesssim h^{m}|\sin\eta|^{-1}, \quad k=0,\ldots,m.
    \end{align*}
    The decomposition of the component $\widetilde{T}_2$ is provided by Lemma~\ref{lemapp:main1}.
    %
    %
    Let us turn to $\widetilde{T}_3$. By substitution this term transforms to
    \begin{align*}
    \widetilde{T}_3(t)&=\int_{t}^{t+h} \partial_1\widetilde{G}(t+h,a(u)) a^\prime(u) \,du.
    \end{align*}
    %
    %
    We put $\widetilde{T}^0_1=0$ and $\widetilde{T}^0_2=\widetilde{T}_3$. These terms clearly satisfy the assertions.
    Next we take the forward difference of $\widetilde{T}^0_2$. Here $\Delta_h$ acts on both $t$ and $\tau$. We obtain
    \beqn
    \Delta_h \widetilde{T}^0_2(t) = \Delta_h \widetilde{T}_3(t) &=& \int_{t}^{t+h} \Delta_h\big(\partial_1\widetilde{G}(t+h,a(\tau)) a^\prime(\tau)\big) \,d\tau \\
      &=& \int_{t}^{t+h} \Delta_h\partial_1\widetilde{G}(t+h,a(\tau)) a^\prime(\tau) \,d\tau \\
      && + \int_{t}^{t+h} \partial_1\widetilde{G}(t+2h,a(\tau+h))\Delta_ha^\prime(\tau) \,d\tau  \\
      &=:& \widetilde{T}_{31}(t) + \widetilde{T}_{32}(t).
    \eeqn
    Lemma~\ref{lemapp:main3} then yields a decomposition $(\widetilde{S}^k_1,\widetilde{S}^k_2)_k$, such that we can write
    \[
    \Delta_h\partial_1\widetilde{G}(t+h,a(\tau)) a^\prime(\tau)=\widetilde{S}^0_1(t,\tau)+\widetilde{S}^0_2(t,\tau).
    \]
    This leads to
    \[
    \widetilde{T}_{31}(t)=\int_{t}^{t+h} \Delta_h\partial_1\widetilde{G}(t+h,a(\tau)) a^\prime(\tau) \,d\tau
    =\int_{t}^{t+h} \widetilde{S}^0_1(t,\tau) \,d\tau + \int_{t}^{t+h} \widetilde{S}^0_2(t,\tau) \,d\tau.
    \]
    We put
    $
    \widetilde{T}^1_1(t):= \widetilde{T}_{32}(t) + \int_{t}^{t+h} \widetilde{S}^0_1(t,\tau) \,d\tau$ and
    $\widetilde{T}^1_2(t):= \int_{t}^{t+h} \widetilde{S}^0_2(t,\tau) \,d\tau$.
    These terms $\widetilde{T}^1_1$ and $\widetilde{T}^1_2$  then satisfy the requirements, i.e.,
    \begin{align*}
    \|\widetilde{T}^1_1\|_\infty &\lesssim \|\widetilde{T}_{32}\|_\infty + h\cdot \sup_{t\in\R}\sup_{\tau\in[t,t+h]} |\widetilde{S}^0_1(t,\tau)| \lesssim h \cdot \delta_j^{m-1}
    \cdot h^{\beta-1}\delta_j|\sin\eta|^{-1-\beta}, \\
    \|\widetilde{T}^1_2\|_\infty&\lesssim h \sup_{t\in\R}\sup_{\tau\in[t,t+h]}|\widetilde{S}^0_2(t,\tau)| \lesssim h^{m} |\sin\eta|^{-1}.
    \end{align*}
    Taking another forward difference of $\widetilde{T}^1_2$ yields
    \begin{align*}
    \Delta_h \widetilde{T}^1_2(t) = \int_t^{t+h} \Delta_h S^0_2(t,\tau) \,d\tau.
    \end{align*}
    Proceeding inductively from here with Lemma~\ref{lemapp:main3} settles the claim for the component $\widetilde{T}_3$.



Finally, we turn to the function $\widetilde{T}_4(t)=\int_{-\infty}^{a(t)} \Delta_{(h,0)} \partial_1\widetilde{G}(t,u)\,du$.
First, we calculate
\begin{flalign*}
&&&\Delta_h \widetilde{T}_4(t) = \int_{a(t)}^{a(t+h)} \Delta_{(h,0)} \partial_1\widetilde{G}(t+h,u)\,du + \int_{-\infty}^{a(t)} \Delta^2_{(h,0)} \partial_1\widetilde{G}(t,u)\,du =: \widetilde{T}_{41}(t) + \widetilde{T}_{42}(t) & \\
\text{and} &&&
\Delta_h \widetilde{T}_{42}(t) = \int_{a(t)}^{a(t+h)} \Delta^2_{(h,0)} \partial_1\widetilde{G}(t+h,u)\,du + \int_{-\infty}^{a(t)} \Delta^3_{(h,0)} \partial_1\widetilde{G}(t,u)\,du =: \widetilde{T}_{43}(t) + \widetilde{T}_{44}(t). &
\end{flalign*}
Next, we show $\| \widetilde{T}_{44}\|_\infty \lesssim h^{\beta+\frac{1}{2}}$ because then, in view of $|\supp \widetilde{T}_{44}|\asymp 1$,
\[
\|\widetilde{T}_{44}\|^2_{2}\lesssim h^{2\beta+1}\asymp 2^{-j(1-\alpha)(2\beta+1)}.
\]
The $L^\infty$-estimate of the term $\widetilde{T}_{44}$ relies on the fact, that for $h\asymp 2^{-j(1-\alpha)}$
\[
\|\Delta^3_{(h,0)}\partial_1\widetilde{G}\|_\infty \lesssim h^{\beta}2^{-\alpha j} \lesssim h^{\beta+1}.
\]
This estimate is a consequence of Lemmas~\ref{lemapp:basicg}, \ref{lemapp:basicw}, and \ref{lemapp:basicr} and is analogous to \eqref{eq:aux_lemg0}. Essential is the observation that since $\alpha\ge\frac{1}{2}$
Lemma~\ref{lemapp:basicg} yields
\[
\|\Delta_{(h,0)}\partial_1g_j\|_\infty \lesssim 2^{-\alpha j} h^{\beta} \lesssim h^{\beta+1}.
\]

Finally, we take care of the remaining terms $\widetilde{T}_{41}$ and $\widetilde{T}_{43}$. First we note
that $|\supp \widetilde{T}_{41}|\lesssim  |\widetilde{I}(\eta)| \lesssim |\sin\eta|$ and also $|\supp \widetilde{T}_{43}|\lesssim  |\widetilde{I}(\eta)| \lesssim |\sin\eta|$.
Hence, it suffices to prove $\|\widetilde{T}_{41}\|_\infty \lesssim h^mh^\beta |\sin\eta|^{-1-\beta} $ and $\|\widetilde{T}_{41}\|_\infty \lesssim h^mh^\beta |\sin\eta|^{-1-\beta}$.
It holds
\beqn
\|\widetilde{T}_{41}\|_\infty &\le& \sup_{t\in\R}\Big|\int_{a(t)}^{a(t+h)} \partial_1\widetilde{G}(t+2h,u)\,du\Big| + \sup_{t\in\R}\Big|\int_{a(t)}^{a(t+h)} \partial_1\widetilde{G}(t+h,u)\,du\Big|.
\eeqn
Analogously, we have
\beqn
\|\widetilde{T}_{43}\|_\infty &\le&  \sup_{t\in\R}\Big|\int_{a(t)}^{a(t+h)} \partial_1\widetilde{G}(t+3h,u)\,du\Big| + 2\sup_{t\in\R}\Big|\int_{a(t)}^{a(t+h)} \partial_1\widetilde{G}(t+2h,u)\,du\Big| \\
&&+ \sup_{t\in\R}\Big|\int_{a(t)}^{a(t+h)} \partial_1\widetilde{G}(t+h,u)\,du\Big|.
\eeqn
All these terms on the right hand side can be estimated in the same way as
\[
\widetilde{T}_{3}(t)=\int_{t}^{t+h} \partial_1\widetilde{G}(t+h,a(u)) a^\prime(u) \,du.
\]
This finishes the proof.
\end{proof}


The following theorem generalizes Theorem~\ref{thm:mainest1}.
\begin{theorem}\label{thm:mainest2}
We have for $m=(m_1,m_2)\in\N_0^2$ the estimate
\begin{align*}
 \int_{|\lambda|\sim 2^{j(1-\alpha)}} |\partial^{m}\widehat{F}(\lambda\cos\eta, \lambda\sin\eta)|^2 \,d\lambda
\lesssim 2^{-2(1-\alpha)jm_1} 2^{-(1-\alpha)j} \big( 1+ 2^{(1-\alpha)j}\sin\eta \big)^{-1-2\beta} + 2^{3\alpha j}2^{(-1-2\beta)j}
\end{align*}
\end{theorem}
\begin{proof}
Note that $\partial^{m}\widehat{F}=\big(x^mF\big)^{\wedge}$.
We define the edge fragment $F:=x_2^{m_2}F$, with a little abuse of notation. Further, we put $\widetilde{F}:=x_1^{m_1}F$.
Analogous to Theorem~\ref{thm:mainest1} we distinguish between the cases $|\sin\eta|< 2\delta_j$ and $|\sin\eta|\ge 2\delta_j$.

In case $|\sin\eta|< 2\delta_j$ we write $\widetilde{F}=\widetilde{F}_0+\widetilde{D}$ with
\[
\widetilde{F}_0(x)=x_1^{m_1} F_0(x) = x_1^{m_1} g(2^{-j\alpha}x)\omega(x) \chi_{\{x_1\ge\delta_j\}}=  x_1^{m_1} F(x) \chi_{\{x_1\ge\delta_j\}}
\]
and $\widetilde{D}(x)=\widetilde{F}(x)-\widetilde{F}_0(x)$ the deviation. We also let $D(x)=F(x)-F_0(x)$.
Note that $\widetilde{F}_0$ is a fragment with a straight edge of height about $\delta_j^{m_1}$.
We want to show
\begin{align*}
\int_{|\lambda|\sim 2^{j(1-\alpha)}} |\widehat{\widetilde{F}}(\lambda\cos\eta,\lambda\sin\eta)|^2 \,d\lambda \lesssim  2^{-2jm_1(1-\alpha)} 2^{-j(1-\alpha)}.
\end{align*}

The function $\widetilde{D}$ is supported in a vertical strip of width $2\delta_j$.
For $\eta$ satisfying $|\sin\eta|<2\delta_j$ the Radon transform $\mathcal{R}\widetilde{D}(\cdot,\eta)$ is $L^\infty$-bounded
with $\|\mathcal{R}\widetilde{D}(\cdot,\eta)\|_\infty \lesssim \delta_j^{m_1} \|\mathcal{R}D(\cdot,\eta)\|_\infty$ and
it is supported in an interval of length
\begin{align*}
2(\delta_j\cos\eta + \sin\eta)\lesssim \delta_j.
\end{align*}
It follows
$\|\mathcal{R}D(\cdot,\eta)\|^2_{2}\lesssim \delta_j^{2 m_1} \delta_j \lesssim \delta_j^{2 m_1} 2^{-j(1-\alpha)}$.
Therefore
\begin{align*}
\int_{|\lambda|\sim 2^{j(1-\alpha)}} |\widehat{D}(\lambda\cos\eta,\lambda\sin\eta)|^2 \,d\lambda
\le \int_{\R} |\widehat{\mathcal{R}D(\cdot,\eta)}(\lambda)|^2 \,d\lambda \lesssim   \delta_j^{2 m_1} 2^{-j(1-\alpha)}.
\end{align*}

It remains to show
\begin{align*}
\int_{|\lambda|\sim 2^{j(1-\alpha)}} |\widehat{\widetilde{F}_0}(\lambda\cos\eta,\lambda\sin\eta)|^2 \,d\lambda \lesssim  \delta_j^{2 m_1} 2^{-j(1-\alpha)}.
\end{align*}
This follows from the fact, that we have decay $|\widehat{\widetilde{F}_0}(\lambda,0)|\lesssim \delta_j^{m_1} |\lambda|^{-1/2}$ normal to the straight singularity curve, since the height
of the jump is $\delta_j^{m_1} $.
Further, the second argument $(\lambda\sin\eta)$ remains bounded due to the condition $|\sin\eta|< 2\delta_j$.

In case $|\sin\eta|\ge 2\delta_j$ we conclude as follows. Let $C_1,\,C_2>0$ be the constants specifying the integration domain
$[C_1 2^{j(1-\alpha)},C_2 2^{j(1-\alpha)}]$. We choose $C>0$ such that $C_2C<2\pi$ and fix $h:=C 2^{-j(1-\alpha)}$.
Then there is $c>0$ such that $ |e^{i\lambda h}-1|^{m_1}\ge c$ for $|\lambda|\in [C_1 2^{j(1-\alpha)},C_2 2^{j(1-\alpha)}]$ at all scales.
We obtain
 \begin{align*}
 \int_{|\lambda|\sim 2^{j(1-\alpha)}} |\lambda|^2 |\partial^{m}\widehat{F}(\lambda\cos\eta,\lambda\sin\eta)|^2 \,d\lambda
 &\lesssim \int_{|\lambda|\sim 2^{j(1-\alpha)}} |e^{i\lambda h}-1|^2 |\lambda|^2|\widehat{x^{m}F}(\lambda\cos\eta,\lambda\sin\eta)|^2 \,d\lambda\\
 &\lesssim\int_{|\lambda|\sim 2^{j(1-\alpha)}}  |e^{i\lambda h}-1|^2 |\lambda|^2  |\left(\mathcal{R}\left(x^{m}F\right)(\cdot,\eta)\right)^{\wedge}(\lambda)|^2 \,d\lambda\\
 &\lesssim\int_{|\lambda|\sim 2^{j(1-\alpha)}}  |\big[\Delta_h\partial_1\mathcal{R}\widetilde{F}(\cdot,\eta)\big]^{\wedge}(\lambda)|^2 \,d\lambda
 \end{align*}
From Lemma~\ref{lem:central2} we know, that $S=\Delta_h\partial_1\mathcal{R}\widetilde{F}(\cdot,\eta)$ admits a decomposition $(S^k_1,S^k_2)_k$ of length $m_1$
with estimates
\begin{align*}
\|S^{k}_1\|_{2}^2\lesssim 2^{-2jm_1(1-\alpha)}  h^{2\beta} |\sin\eta|^{-1-2\beta} + 2^{-j(1-\alpha)(2\beta+1)},\qquad k=0,1,\ldots,m_1+1.
\end{align*}

Using the same trick as in Theorem~\ref{thm:mainest1} we can then conclude
\begin{align*}
 \int_{|\lambda|\sim 2^{j(1-\alpha)}} |\lambda|^2 |\partial^m\widehat{F}(\lambda\cos\eta,\lambda\sin\eta)|^2 \,d\lambda
&\lesssim\int_{|\lambda|\sim 2^{j(1-\alpha)}}  |\big[\Delta_h\partial_1\mathcal{R}\widetilde{F}(\cdot,\eta)\big]^{\wedge}(\lambda)|^2 \,d\lambda \\
&\lesssim \sum_{k=0}^{m_1+1} \int_{|\lambda|\sim 2^{j(1-\alpha)}}   |\widehat{S^k_1}|^2 \,d\lambda  \\
&\lesssim \sum_{k=0}^{m_1+1} \|S^k_1\|_2^2  \\
&\lesssim  2^{-2jm_1(1-\alpha)}  h^{2\beta} |\sin\eta|^{-1-2\beta} + 2^{-j(1-\alpha)(2\beta+1)}.
\end{align*}

It follows
\begin{align*}
\int_{|\lambda|\sim 2^{j(1-\alpha)}} |\partial^{m}\widehat{F}(\lambda\cos\eta,\lambda\sin\eta)|^2 \,d\lambda
&\lesssim  2^{-2j(1-\alpha)(m_1+1)}  h^{2\beta} |\sin\eta|^{-1-2\beta} + 2^{-j(1-\alpha)(2\beta+3)}\\
&\lesssim  2^{-2(1-\alpha)jm_1} 2^{-(1-\alpha)j} \big( 2^{(1-\alpha)j}\sin\eta \big)^{-1-2\beta} + 2^{3\alpha j}2^{(-1-2\beta)j}.
\end{align*}
This finishes the proof.
\end{proof}

Rescaling $F$ to the original edge fragment $f$ yields the following corollary.

\begin{cor}\label{cor:localf2}
We have for $m=(m_1,m_2)\in\N_0^2$ the estimate
\begin{align*}
\int\limits_{|\lambda|\sim 2^{j}} |\partial^{m}\widehat{f}(\lambda\cos\eta, \lambda\sin\eta)|^2 \,d\lambda
\lesssim 2^{-j2\alpha|m|} \Big( 2^{-j2(1-\alpha)m_1} 2^{-(1+2\alpha)j} \Big( 1+ 2^{(1-\alpha)j}|\sin\eta| \Big)^{-1-2\beta} + 2^{(-1-2\beta)j}\Big).
\end{align*}
\end{cor}
\begin{proof}
The statement follows directly from the relation $\widehat{f}(\xi)=2^{-2\alpha j}\widehat{F}(2^{-\alpha j}\xi)$.
\end{proof}

\subsubsection{Curvelet Analysis of an Edge Fragment}

In the following $J=(j,\ell)$ shall denote a scale-angle pair with $j\in\N_0$, $\ell\in\{0,\ldots,L_j-1\}$, and
$\chi_J$ and $\chi_0$ shall be the functions from \eqref{eq:suppfunctions}.
Recall also the characteristic angle $\omega_j=\pi 2^{-\lfloor j(1-\alpha) \rfloor}$ at scale $j$ and the corresponding orientations
$\omega_J=\ell\omega_j$, ranging between $0$ and $\pi$.
From Corollary~\ref{cor:localf} we can directly conclude the following result.

\begin{theorem}\label{thm:central1}
We have
\begin{align}\label{eq:central1}
\int_{\R^2} |\widehat{f}\chi_J(\xi)|^2 \,d\xi \lesssim 2^{-(1+\alpha)j}\Big( 1+ 2^{(1-\alpha)j}|\sin\omega_J| \Big)^{-1-2\beta}.
\end{align}
\end{theorem}
\begin{proof}
It holds $\|\chi_J\|_\infty\le1$ and $\supp \chi_J\subset \mathcal{W}_J$, where $\mathcal{W}_J$ is the wedge defined in \eqref{eq:wedgePJ}.
Let us define the intervals $\mathcal{K}_j=\frac{1}{8\pi}[2^{j-1},2^{j+1}]$ and $\mathcal{A}_J=[\omega_J-\omega_j/2,\omega_J+\omega_j/2]$.
Using Corollary~\ref{cor:localf} we calculate
\begin{align*}
\int_{\R^2} |\widehat{f}\chi_J(\xi)|^2 \,d\xi
&\le \int_{\mathcal{W}_J}  |\widehat{f}(\xi)|^2 \,d\xi \\
&=   \int_{\mathcal{K}_j} \Big(\int_{\mathcal{A}_J} + \int_{\pi+\mathcal{A}_J} \Big)  |\widehat{f}(\lambda,\eta)|^2 \lambda \,d\eta\,d\lambda \\
&\lesssim \Big(\int_{\mathcal{A}_J} + \int_{\pi+\mathcal{A}_J} \Big)   2^{-(1+2\alpha)j}\Big( 1+ 2^{(1-\alpha)j}|\sin\eta| \Big)^{-1-2\beta}2^j  \,d\eta \\
&\lesssim 2^{-(1+\alpha)j}\Big( 1+ 2^{(1-\alpha)j}|\sin\omega_J| \Big)^{-1-2\beta}.
\end{align*}
\end{proof}

\noindent
For a scale-angle pair $J=(j,\ell)$ let us define
\[
\ell_J=1+ 2^{(1-\alpha)j}|\sin\omega_J|
\]
and the differential operator
\[
\mathcal{L}=(\mathcal{I}-(2^j/\ell_J)^2\mathcal{D}_1^2)(\mathcal{I}-2^{2\alpha j}\mathcal{D}_2^2)=\mathcal{I}-2^{2j}\ell_J^{-2}\mathcal{D}_1^2-2^{2\alpha j}\mathcal{D}_2^2+2^{2(1+\alpha)j}\ell_J^{-2}\mathcal{D}_1^2\mathcal{D}_2^2,
\]
with identity $\mathcal{I}$ and partial derivatives
\begin{align}\label{eq:diffop}
\mathcal{D}_1=\cos \omega_J\cdot \partial_1 + \sin \omega_J\cdot \partial_2 \quad\text{and}\quad \mathcal{D}_2=-\sin \omega_J\cdot \partial_1 + \cos \omega_J\cdot \partial_2.
\end{align}
We will show that $\mathcal{L}(\widehat{f}\chi_J)$ obeys the same estimate \eqref{eq:central1} as $\widehat{f}\chi_J$.


\begin{theorem}\label{thm:central2}
We have
\begin{align*}
\int_{\R^2} |\mathcal{L}(\widehat{f}\chi_J)(\xi)|^2 \,d\xi \lesssim 2^{-(1+\alpha)j}\Big( 1+ 2^{(1-\alpha)j}|\sin\omega_J| \Big)^{-1-2\beta}.
\end{align*}
\end{theorem}
\begin{proof}
First, observe that for each pair $m=(m_1,m_2)\in N_0^2$ the mixed derivative of $\chi_J$ obeys
\begin{align}\label{basicfact}
\|\mathcal{D}_1^{m_1}\mathcal{D}_2^{m_2}\chi_J\|_{\infty}= \mathcal{O}\big(2^{-jm_1}\cdot 2^{-j\alpha m_2}\big).
\end{align}
This follows from the fact, that the functions $\chi_J$ from \eqref{eq:suppfunctions} scale with their support wedges $\mathcal{W}_J$, which are
of length $\sim 2^{j}$ and width $\sim2^{\alpha j}$.

Next, from the definition \eqref{eq:diffop} of the operators $\mathcal{D}_1$ and $\mathcal{D}_2$ we deduce for $m_1\in\N_0$
\[
\mathcal{D}_1^{m_1} \widehat{f} = \sum_{a+b=m_1}  c_{a,b} (\cos\omega_J)^a(\sin\omega_J)^b \partial^{(a,b)}\widehat{f}
\]
with binomial coefficients $c_{a,b}\in\N$. A similar formula holds for $\mathcal{D}_2^{m_2} \widehat{f}$ and $m_2\in\N_0$. Using $|\sin\omega_J|\le 2^{-(1-\alpha)j}\ell_J$ we obtain the estimate
\[
\|\mathcal{D}_1^{m_1} \widehat{f}\|^2_{L^2(\mathcal{W}_J)} \lesssim \sum_{a+b=m_1} |\sin\omega_J|^{2b} \cdot \| \partial^{(a,b)}\widehat{f} \|^2_{L^2(\mathcal{W}_J)}
\le \sum_{a+b=m_1} (2^{-(1-\alpha)j}\ell_J)^{2b} \cdot \| \partial^{(a,b)}\widehat{f} \|^2_{L^2(\mathcal{W}_J)}.
\]
Analogously, we obtain
\[
\|\mathcal{D}_2^{m_2} \widehat{f}\|^2_{L^2(\mathcal{W}_J)} \lesssim \sum_{a+b=m_2}  \| \partial^{(a,b)}\widehat{f} \|^2_{L^2(\mathcal{W}_J)}.
\]

Taking into account the width $\sim2^{\alpha j}$ of the wedges $\mathcal{W}_J$, Corollary~\ref{cor:localf2} gives for $(a,b)\in\N_0^2$ the bound
\[
\|\partial^{(a,b)}\widehat{f} \|^2_{L^2(\mathcal{W}_J)} \le C_{a,b}  2^{\alpha j} \cdot 2^{-j2\alpha(a+b)} \big(  2^{-2j(1-\alpha)a}2^{-(1+2\alpha)j} \ell_J^{-1-2\beta}+ 2^{(-1-2\beta)j}\big),
\]
with some constant $C_{a,b}>0$ independent of scale. Therefore we can estimate for $m_1\in\N_0$
\begin{align*}
\|\mathcal{D}_1^{m_1} \widehat{f}\|^2_{L^2(\mathcal{W}_J)} \lesssim  2^{-j2\alpha m_1} \big(  2^{-2j(1-\alpha)m_1} 2^{-(1+\alpha)j} \ell_J^{2m_1-1-2\beta} + 2^{\alpha j} 2^{(-1-2\beta)j} \big).
\end{align*}
If $m_1\le 2$ this further simplifies to
\begin{align}\label{eq:finest1}
\|\mathcal{D}^{m_1}_1 \widehat{f} \|^2_{L^2(\mathcal{W}_J)} \lesssim 2^{-2jm_1} 2^{-(1+\alpha)j} \ell_J^{2m_1-1-2\beta},
\end{align}
since for every $m_1\le (1+\alpha)/\alpha$ we have
\[
2^{\alpha j} 2^{(-1-2\beta)j} \lesssim 2^{-2j(1-\alpha)m_1} 2^{-(1+\alpha)j} \ell_J^{2m_1-1-2\beta}.
\]

Similar calculations lead to
\begin{align}\label{eq:finest2}
\|\mathcal{D}^{m_2}_2 \widehat{f} \|^2_{L^2(\mathcal{W}_J)} \lesssim 2^{-2\alpha jm_2} 2^{-(1+\alpha)j} \ell_J^{-1-2\beta}.
\end{align}
Indeed, if $a+b=m_2$ we have
\[
\|\partial_1^a\partial_2^b \widehat{f} \|^2_{L^2(\mathcal{W}_J)} \le C_{a,b}\cdot 2^{\alpha j} \cdot 2^{-2\alpha j m_2} \big(  2^{-(1+2\alpha)j} \ell_J^{-1-2\beta}+ 2^{(-1-2\beta)j}\big).
\]
Since $1\le\ell_J\le 2\cdot 2^{(1-\alpha)j}$ it holds
\[
2^{j(1+\alpha)}2^{-2\beta j}=2^{-(1-\alpha)(1+2\beta)j} \lesssim \ell_J^{-1-2\beta}.
\]
Therefore, taking into account $1\le 2^{(1-\alpha)j}$, we can conclude
\[
2^{(-1-2\beta)j} \le 2^{j(1+\alpha)}2^{-2\beta j}2^{-(1+2\alpha)j} \lesssim 2^{-(1+2\alpha)j} \ell_J^{-1-2\beta}.
\]
Altogether we obtain the desired estimate \eqref{eq:finest2}.

After this preliminary work we can finally prove the statement of Theorem~\ref{thm:central2}. We have
\[
\mathcal{L}(\widehat{f}\chi_J)=\widehat{f}\chi_J-2^{2j}\ell_J^{-2}\mathcal{D}_1^2(\widehat{f}\chi_J)-2^{2\alpha j}\mathcal{D}_2^2(\widehat{f}\chi_J)+2^{2(1+\alpha)j}\ell_J^{-2}\mathcal{D}_1^2\mathcal{D}_2^2(\widehat{f}\chi_J),
\]
which allows us to show the desired estimate for each term separately. For $\widehat{f}\chi_J$ the estimate
holds true by Theorem~\ref{thm:central1}.

Let us turn to the second term. The product rule yields
\[
\mathcal{D}^2_1(\widehat{f}\chi_J) = (\mathcal{D}_1^2\widehat{f})\chi_J + 2(\mathcal{D}_1\widehat{f})(\mathcal{D}_1\chi_J) + \widehat{f}(\mathcal{D}_1^2\chi_J).
\]
The previous estimates together with the Hölder inequality then lead to
\begin{align*}
\| \mathcal{D}^2_1(\widehat{f}\chi_J) \|^2_{2}\lesssim 2^{-4j} \ell_J^4 2^{-(1+\alpha)j} \ell_J^{-1-2\beta}.
\end{align*}
Here \eqref{eq:finest1} was used, and that $\|\mathcal{D}^{m_1}_1 \chi_J \|^2_{\infty} \le 2^{-2jm_1} \le 2^{-2jm_1}\ell_J^{2m_1}$ by \eqref{basicfact}.
This settles the claim for the second term.

Analogously, we can deduce
\begin{align*}
\| \mathcal{D}^2_2(\widehat{f}\chi_J) \|^2_{2} \lesssim 2^{-4\alpha j} 2^{-(1+\alpha)j} \ell_J^{-1-2\beta},
\end{align*}
using \eqref{eq:finest2} and that $\|\mathcal{D}^{m_2}_2 \chi_J \|^2_{\infty} \le 2^{-2\alpha jm_2}$ by \eqref{basicfact}.
This gives the estimate for the third term.

Finally, it also holds
\begin{align*}
\| \mathcal{D}^2_1\mathcal{D}^2_2(\widehat{f}\chi_J) \|^2_{2} \lesssim 2^{-4j(1+\alpha)} \ell_J^4 2^{-(1+\alpha)j} \ell_J^{-1-2\beta},
\end{align*}
which establishes the result for the fourth term.

\end{proof}

At last we are ready to give the proof of Proposition~\ref{prop:sequence1}.
The essential tool is Theorem~\ref{thm:central2}.

\begin{proof}[Proof of Proposition~\ref{prop:sequence1}]
Recall the curvelet frame $\mathcal{C}_\alpha(W^{(0)},W,V)=(\psi_\mu)_{\mu\in M}$. On the Fourier side we have
\[
\widehat{\psi}_{j,\ell,k}=\chi_J u_{j,k}(R_J\cdot),
\]
with rotation matrix $R_J$ given as in \eqref{eq:matrixrot} and functions
\[
u_{j,k}(\xi)=2^{-j(1+\alpha)/2}e^{2\pi i (2^{-j}k_1,2^{-\alpha j}k_2) \cdot \xi}, \quad\xi\in\R^2.
\]
The curvelet coefficients $(\theta_\mu)_{\mu\in M}$ are therefore given by the formula
\begin{align*}
\theta_\mu = \langle f, \psi_{j,\ell,k} \rangle = \int_{\R^2} \widehat{f}\chi_J(\xi) \overline{u_{j,k}}(R_J\xi) \,d\xi.
\end{align*}
Since
\begin{align*}
\mathcal{L} (u_{j,k}) = (1+ \ell_J^{-2}k_1^2)(1+ k^2_2 ) u_{j,k},
\end{align*}
integration by parts yields
\begin{align*}
\theta_\mu = (1+ \ell_J^{-2}k_1^2)^{-1}(1+ k^2_2 )^{-1} \int_{\R^2} \mathcal{L}(\widehat{f}\chi_J)(\xi) \overline{u_{j,k}}(R_J\xi) \,d\xi.
\end{align*}
Let $K=(K_1,K_2)\in\Z^2$ and define
\[
\mathfrak{Z}_{K}:=\Big\{ (k_1,k_2) \in \Z^2 ~:~ \ell_J^{-1}k_1\in[K_1,K_1+1),\, k_2 = K_2 \Big\}.
\]
For fixed $J=(j,\ell)$ the Fourier system $\big(u_{j,k}(R_J\cdot)\big)_{k\in\Z^2}$ is an orthonormal basis for $L^2(\Xi_J)$, where $\Xi_J$ is the
rectangle defined in \eqref{eq:supprect} containing the support of $\chi_J$. Therefore,
\begin{align*}
\sum_{k\in \mathfrak{Z}_K}|\theta_\mu|^2 \lesssim (1+ K_1^2)^{-2}(1+ K^2_2 )^{-2}  \int_{\R^2} |\mathcal{L}(\widehat{f}\chi_J)(\xi)|^2 \,d\xi.
\end{align*}
The integral on the right-hand side is bounded by Theorem~\ref{thm:central2}, and we thus arrive at
\begin{align}\label{eq:keyest2}
\sum_{k\in \mathfrak{Z}_K}|\theta_\mu|^2 \lesssim (L_K )^{-2}  2^{-(1+\alpha)j} \ell_J^{-1-2\beta}
\end{align}
with $L_K:=(1+ K_1^2)(1+ K^2_2 )$.

Let $M_J$ denote the subset of curvelet coefficients associated with a fixed scale-angle pair $J=(j,\ell)$.
Further, let 
$N_{J,K}(\varepsilon)$
be the number of indices $\mu\in M_J$ such that $k\in \mathfrak{Z}_K$ and $|\theta_\mu|>\varepsilon$.

Since $\#\mathfrak{Z}_K\le \ell_J$ and because of \eqref{eq:keyest2} we can conclude
\begin{align}\label{eq:estNJK}
N_{J,K}(\varepsilon) \lesssim \min\Big\{\ell_J,  (\varepsilon L_K)^{-2} 2^{-(1+\alpha)j} \ell_J^{-1-2\beta} \Big\}.
\end{align}

For $\omega_J=\pi \ell 2^{-\lfloor j(1-\alpha)\rfloor}\in [0,\pi)$
let $\langle  \omega_J \rangle$ denote the equivalent angle modulo $\pi$ in the interval $(-\pi/2,\pi/2]$. The corresponding indices in the range $\{\lfloor-L_j/2+1\rfloor,\ldots,\lfloor L_j/2\rfloor\}$
shall be denoted by $\langle\ell\rangle$.
Since it holds $|\sin\eta|\asymp |\eta|$ for $\eta\in[-\pi/2,\pi/2]$, it follows
\begin{align}\label{eq:eqiv}
\ell_J= 1+2^{(1-\alpha)j}|\sin\omega_J| =  1+2^{(1-\alpha)j}|\sin \langle\omega_J\rangle| \asymp 1+|\langle\ell\rangle|.
\end{align}

Let $\ell_\ast$ be the solution of $\ell_\ast=(\varepsilon L_K)^{-2} 2^{-(1+\alpha)j} \ell_\ast^{-1-2\beta}$ and put $L^\ast=\lfloor \ell_\ast \rfloor$.
Utilizing \eqref{eq:estNJK} and \eqref{eq:eqiv} yields
\begin{align*}
\sum_{|J|=j} N_{J,K}(\varepsilon) &\lesssim
\sum_{\substack{\ell\in\{0,\ldots,L_j-1\} \\ |\langle\ell\rangle|\le L^\ast-1 }} (1+|\langle\ell\rangle|) + \sum_{\substack{\ell\in\{0,\ldots,L_j-1\} \\ |\langle\ell\rangle|\ge L^\ast }}  (\varepsilon L_K)^{-2} 2^{-(1+\alpha)j} (1+|\langle\ell\rangle|)^{-1-2\beta} \\
&\lesssim \sum_{\ell=0}^{L^*-1} (1+|\ell|) + \sum_{\ell=L^*}^{\infty} (\varepsilon L_K)^{-2} 2^{-(1+\alpha)j} (1+|\ell|)^{-1-2\beta} \\
&\lesssim  (L^*)^2 +   (\varepsilon L_K)^{-2} 2^{-(1+\alpha)j} (L^*)^{-2\beta}.
\end{align*}
This translates to
\begin{align*}
\sum_{|J|=j} N_{J,K}(\varepsilon) \lesssim  \varepsilon^{-2/(1+\beta)} \cdot L_K^{-2/(1+\beta)} \cdot 2^{-(1+\alpha)j/(1+\beta)}.
\end{align*}
Since $\beta<3$ we have $\sum_{K\in\Z^2} L_K^{-2/(1+\beta)}< \infty$. Hence
\[
\# \Big\{ \mu\in M_j, |\theta_\mu|>\varepsilon \Big\} = \sum_{K\in\Z^2} \sum_{|J|=j} N_{J,K}(\varepsilon)   \lesssim 2^{-(1+\alpha)j/(1+\beta)} \varepsilon^{-2/(1+\beta)}  .
\]
This finishes the proof.
\end{proof}

\clearpage
\begin{appendix}

\section{Additional Technical Estimates}

This appendix contains some technical estimates, which are needed for the proofs of the main results.
They were outsourced from the main exposition to enhance the readability.
We first state a lemma, which shows how scaling affects the H\"older constant in an abstract setting.

\begin{lemma}\label{lemapp:hoel}
Let $f\in C^\alpha(\R)$ and $0<\alpha<1$. Then for $t,s>0$
\[
\text{\sl Höl}(s f(t\cdot),\alpha)=st^\alpha\cdot \text{\sl Höl}(f,\alpha).
\]
\end{lemma}

Subsequenlty we establish some estimates for the `elementary functions' $g^\eta$, $\omega^\eta$, and $r^\eta$,
which occur as components of the edge fragments $F$ and $\widetilde{F}$. These estimates provide the basis
for the more complex estimates needed throughout this paper. To enhance readability
we omit the superindex $\eta$ during this discussion.

\subsection{Estimates for $g^\eta$}

The functions $g^\eta_j$ are defined for $j\in\N_0$ by $g^\eta_j:=g^\eta(2^{-\alpha j}\cdot)$, where $g^\eta\in C_0^{\beta}(\R^2)$ is
a fixed function and $\beta\in(1,2]$. For simplicity we write henceforth $g_j$ and $g$.

Clearly, still $g_j\in C_0^{\beta}(\R^2)$. However, the parameters of the regularity change.
It is obvious that $\|g_j\|_\infty\lesssim 1$. Further, the chain rule yields
\begin{align*}
\|\partial_1 g_j\|_\infty \lesssim 2^{-\alpha j} \quad\text{and}\quad  \|\partial_2 g_j\|_\infty\lesssim 2^{-\alpha j}.
\end{align*}
Applying Lemma~\ref{lemapp:hoel} yields the following result.
\begin{lemma}\label{lemapp:hoelg}
Let $\beta\in(1,2]$ and $g\in C_0^{\beta}(\R^2)$. Then for $g_j=g(2^{-\alpha j}\cdot)$
\[
\Hol(\partial_1g_{j},\beta-1)=2^{-j}\Hol(\partial_1g,\beta-1).
\]
\end{lemma}
\begin{proof}
In view of Lemma~\ref{lemapp:hoel} we have
\begin{align*}
\Hol(\partial_1g_{j},\beta-1)=\Hol(2^{-\alpha j}\partial_1g(2^{-\alpha j}\cdot),\beta-1)=2^{-j}\Hol(\partial_1g,\beta-1).
\end{align*}
\end{proof}

Some more estimates for $g_j$ are collected in the following two lemmas.

\begin{lemma}\label{lemapp:basicg}
The following estimates hold true for $g_j$:
\begin{align*}
\|\Delta_{(h,0)} g_j\|_\infty &\lesssim 2^{-\alpha j}h , \\
\|\Delta_{(h,0)} \partial_1g_j \|_\infty,\, \|\Delta_{(h,0)} \partial_2g_j \|_\infty &\lesssim 2^{-j} h^{\beta-1}=  2^{-\alpha j}h^{\beta}  , \\
\|\Delta^2_{(h,0)} g_j \|_\infty &\lesssim 2^{-j}h^{\beta}  ,
\end{align*}
with implicit constants, that do not depend on $j\in\N_0$ and $h\ge0$.
\end{lemma}
\begin{proof}
Applying the mean value theorem yields
\begin{align*}
\|\Delta_{(h,0)} g_j \|_\infty \le h \|\partial_1g_j \|_\infty
\lesssim 2^{-\alpha j}h.
\end{align*}
Considering Lemma~\ref{lemapp:hoelg} we obtain
\begin{align*}
\|\Delta_{(h,0)} \partial_1g_j \|_\infty \lesssim 2^{-j} h^{\beta-1} = 2^{-\alpha j} h^{\beta}.
\end{align*}
Noting the commutativity $\partial_1\Delta_{(h,0)}=\Delta_{(h,0)}\partial_1$, we obtain
\begin{align*}
\|\Delta^2_{(h,0)} g_j\|_\infty \lesssim h \|\Delta_{(h,0)} \partial_1g_j \|_\infty
\lesssim 2^{-j}h^{\beta}.
\end{align*}
\end{proof}

The next lemma gives estimates for $g_j$ along the edge curve.
Here the function $a\in C^\beta(\R)$ comes into play, which was defined in \eqref{eq:a_def}.
The following estimates also depend on the properties of $a$, which are summarized
in Lemma~\ref{lem:propa}.

\begin{lemma}\label{lemapp:edgeg}
Assume $|\sin\eta|\ge 2\delta_j$.
The following estimates hold true for $g_j$:
\begin{align*}
\sup_{t\in\R}|\Delta_h g_j(t,a(t))| &\lesssim h |\sin\eta|^{-1} 2^{-\alpha j}, \\
\sup_{t\in\R}|\Delta_h \partial_1g_j(t,a(t))|,\, \sup_{t\in\R}|\Delta_h \partial_2g_j(t,a(t))| &\lesssim h^{\beta-1} |\sin\eta|^{1-\beta} 2^{-j},\\
\sup_{t\in\R}|\Delta^2_h g_j(t,a(t))| &\lesssim h^{\beta} |\sin\eta|^{-1-\beta}  2^{-j},
\end{align*}
where the implicit constants are independent of $j\in\N_0$ and $h\ge0$.
\end{lemma}
\begin{proof}
In view of Lemma~\ref{lem:propa} it holds
\begin{align*}
\sup_{t\in\R}|\Delta_h g_j(t,a(t))| &\lesssim h\cdot \sup_{t\in\R}|\frac{d}{dt} g_j(t,a(t))| \\
&\lesssim h\cdot \Big(  \sup_{t\in\R}| \partial_1g_j(t,a(t))|  + \sup_{t\in\R} |\partial_2g_j(t,a(t))a^\prime(t)| \Big) \\
&\lesssim h\cdot |\sin\eta|^{-1} 2^{-\alpha j}.
\end{align*}
Considering the transformation behavior of the H\"older constant we obtain with Lemma~\ref{lem:propa}
\begin{align*}
\sup_{t\in\R}|\Delta_h \partial_1g_j(t,a(t))| &\lesssim 2^{-j} \sup_{t\in\R}|(h,a(t+h)-a(t))|_2^{\beta-1} \\
&\lesssim 2^{-j} \big( h^{\beta-1} + \sup_{t\in\R}|a(t+h)-a(t)|^{\beta-1} \big)  \\
&\lesssim 2^{-j} h^{\beta-1} |\sin\eta|^{1-\beta}.
\end{align*}
Applying Lemma~\ref{lem:propa}, the mean value theorem and $\frac{d}{dt}\Delta_h=\Delta_h\frac{d}{dt}$ yields
\begin{align*}
\sup_{t\in\R}|\Delta^2_h g_j(t,a(t))| &\lesssim h \cdot \sup_{t\in\R}|\Delta_h \frac{d}{dt}g_j(t,a(t)) |  \\
&=  h\cdot \sup_{t\in\R}|\Delta_h \big(\partial_1g_j(t,a(t)) + \partial_2g_j(t,a(t))a^\prime(t) \big) |  \\
&=  h\cdot \sup_{t\in\R}| \Delta_h \partial_1g_j(t,a(t)) + \Delta_h \partial_2g_j(t,a(t))a^\prime(t+h) + \partial_2g_j(t,a(t))\Delta_ha^\prime(t)  |  \\
&\lesssim  h^{\beta} |\sin\eta|^{1-\beta}2^{-j} +   h^{\beta} |\sin\eta|^{-\beta}2^{-j}
+ \delta_j h^{\beta}  |\sin\eta|^{-1-\beta}2^{-\alpha j}.
\end{align*}
\end{proof}

\subsection{Estimates for $\omega^\eta$}

Similarly, we obtain estimates for the window function $\omega^\eta\in C_0^\infty(\R^2)$, which
in contrast to the functions $g_j$ remains fixed at all scales.
This fact and the smoothness of $\omega$ result in different estimates.

First we state the trivial estimates $\|\omega\|_\infty\lesssim 1$, $\|\partial_1\omega\|_\infty\lesssim 1$, and $\|\partial_2\omega\|_\infty\lesssim 1$.
Next, we apply the forward difference operator $\Delta_{(h,0)}$ to $\omega$.

\begin{lemma}\label{lemapp:basicw}
Let $k\in\N_0$. It holds with implicit constants independent of $h\ge0$
\begin{align*}
\|\Delta^k_{(h,0)} \omega\|_\infty \lesssim h^k  \quad\text{and}\quad
\|\Delta^k_{(h,0)} \partial_1\omega\|_\infty \lesssim h^k.
\end{align*}
\end{lemma}

Analogous to Lemma~\ref{lemapp:edgeg} we establish estimates along the edge curve.

\begin{lemma}\label{lemapp:edgew}
Assume $|\sin\eta|\ge 2\delta_j$. It holds
\begin{align*}
\sup_{t\in\R}|\Delta_h \omega(t,a(t))| &\lesssim h |\sin\eta|^{-1},  \\
\sup_{t\in\R}|\Delta_h \partial_1\omega(t,a(t))|,\, \sup_{t\in\R}|\Delta_h \partial_2\omega(t,a(t))| &\lesssim h |\sin\eta|^{-1}, \\
\sup_{t\in\R}|\Delta^2_h \omega(t,a(t))| &\lesssim h^2 |\sin\eta|^{-2} + \delta_j h^{\beta} |\sin\eta|^{-1-\beta}.
\end{align*}
\end{lemma}
\begin{proof}
This proof is analogous to the proof of Lemma~\ref{lemapp:edgeg}.
%
%
\end{proof}

\subsection{Estimates for $r^\eta$}

Next we analyze the function $r^\eta:\R^2\rightarrow\R$ given by
\[
r^\eta(t,a):=t\cos\eta - a \sin\eta,\quad (t,a)\in\R^2,
\]
which is a component of the modified edge fragment $\widetilde{F}$.
Clearly $r\in C^\infty(\R^2)$. Note however, that $r$ is not compactly supported.
Since $r$ only occurs as a factor in products with the window $\omega$ this does not cause any problems.

Thanks to the smoothness of $r$ we have the following result.

\begin{lemma}\label{lemapp:basicr}
Let $k,m\in\N_0$ and $K\subset\R^2$ a compact set. Then we have
\begin{align*}
\| \Delta^{k}_{(h,0)} r^m \|_{L^\infty(K)} \lesssim h^{k}.
\end{align*}
\end{lemma}

Along the edge curve the following estimates hold.
Here $\widetilde{I}(\eta)$ denotes the interval defined in \eqref{eq:interval}.

\begin{lemma}\label{lemapp:r}
Let $|\sin\eta|\ge2\delta_j$. Then we have $\sup_{t\in\widetilde{I}(\eta)}|r(t,a(t))|\lesssim \delta_j$. Moreover, for $h\ge 0$ it holds
\[
\sup_{t\in\R}|\Delta_hr(t,a(t))|\lesssim h \quad\text{and}\quad  \sup_{t\in\R}|\Delta^2_hr(t,a(t))|\lesssim h^{\beta}\delta_j |\sin\eta|^{-\beta}.
\]
\end{lemma}
\begin{proof}
For every $t\in\R$ the point $(t,a(t))\in\R^2$ in rotated coordinates lies on the (extended) edge curve $\Gamma$.
We know that the function $E_j$ deviates little from zero and obeys $\sup_{|x_2|\le 1}|E_j(x_2)|\le \delta_j \lesssim 2^{-j(1-\alpha)}$
according to \eqref{eq:vertstrip}. Furthermore, the slope of $E_j$ outside of $[-1,1]$ is constant and bounded by $\delta_j$.
This yields the estimate $\sup_{t\in \widetilde{I}(\eta)} |r(t,a(t))|\lesssim \delta_j$.

The other estimates follow from Lemma~\ref{lem:propa}.
In view of this lemma we conclude
\[
\sup_{t\in\R}|\Delta_hr(t,a(t))|\le h\cdot \sup_{t\in\R}|\cos\eta - a^\prime(t)\sin\eta| \lesssim h |\sin\eta|^{-1} |\sin\eta| =h,
\]
and
\[
\sup_{t\in\R}|\Delta^2_hr(t,a(t))|\le h |\sin\eta| \|\Delta_ha^\prime\|_{\infty} \lesssim h^{\beta}\delta_j |\sin\eta|^{-\beta}.
\]
\end{proof}

\subsection{Estimates for $\widetilde{G}^\eta$}

The function $\widetilde{G}^\eta$ is the rotated version of the function
\[
\widetilde{G}(t,a)=r(t,a)^m G(t,a)=r(t,a)^m g_j(t,a)\omega(t,a),\quad (t,a)\in\R^2,
\]
which is a composition of the `elementary functions' discussed before.
Hence we can apply the previous estimates to obtain estimates for $\widetilde{G}^\eta$.

\begin{lemma}\label{lemapp:edgeG}
Let $|\sin\eta|\ge2\delta_j$.
Let $\widetilde{G}^\eta(t,a)=r^\eta(t,a)^m G^\eta(t,a)$ for $(t,a)\in\R^2$, $m\in\N$, $m\neq0$.
Then there are the estimates
\begin{align*}
\sup_{t\in\R}|\widetilde{G}(t,a(t))|&\lesssim \delta_j^m, &
\sup_{t\in\R}|\Delta_h\widetilde{G}(t,a(t))| &\lesssim \delta_j^{m-1}h, \\
\sup_{t\in\R}|\partial_1\widetilde{G}(t,a(t))|&\lesssim \delta_j^{m-1}, &
\sup_{t\in\R}|\partial_2\widetilde{G}(t,a(t))|&\lesssim \delta_j^{m-1}|\sin\eta|.
\end{align*}
\end{lemma}
\begin{proof}
We calculate for $(t,a)\in\R^2$
\begin{flalign*}
&&&\partial_1\widetilde{G}(t,a)=\partial_1 \big( r(t,a)^mG(t,a) \big) = (\cos\eta) m r(t,a)^{m-1}  G(t,a) + r(t,a)^m \partial_1G(t,a), & \\
\text{and} &&&\partial_2\widetilde{G}(t,a)=\partial_2 \big( r(t,a)^mG(t,a) \big) = -(\sin\eta) m r(t,a)^{m-1}  G(t,a) + r(t,a)^m \partial_2G(t,a). &
\end{flalign*}
The assertion is then a consequence of the following facts. It holds $\|G\|_\infty\lesssim 1$ and $|r(t,a(t))|\le\delta_j$ for all $t\in I(\eta)$.
Further, for $t\notin I(\eta)$ the expressions $G(t,a(t))$, $\partial_1G(t,a(t))$, and $\partial_2G(t,a(t))$ vanish.
\end{proof}

\subsection{Auxiliary Lemmas}


In this last subsection we prove three important technical lemmas, which are the foundation of the proof of Lemma~\ref{lem:central2}.
We refer to Lemma~\ref{lem:central2} for an explanation of the terminology.

\begin{lemma}\label{lemapp:main1}
Let $\widetilde{G}(x)=r(x)^m G(x)$ for $x\in\R^2$ and $m\in\N_0$. Further, let $h\asymp 2^{-j(1-\alpha)}$.
The function $T:\R\rightarrow\R$ defined by $T(t)=a^\prime(t)\Delta_h \widetilde{G}(t,a(t))$ then admits a decomposition $(T^k_1,T^k_2)_k$
of the form $(\ast)$ of length $(m+1)$ with the
estimates
\begin{align*}
\|T^{k}_1\|_\infty &\lesssim h^{m}h^{\beta} |\sin\eta|^{-1-\beta}, \quad k=0,\ldots,m+1, \\
\|T^{k}_2\|_\infty &\lesssim h^{m}|\sin\eta|^{-1}, \quad k=0,\ldots,m,
\end{align*}
and subject to the condition $\supp T^{k}_i \subset \widetilde{I}(\eta)$.
\end{lemma}
An explanation of the terminology is given in Lemma~\ref{lem:central2}.
\begin{proof}
    We prove this by induction on $m$.
    If $m=0$ we put
    $T^0_1=T_{21}$, $T^0_2=T_{22}$, $T^1_1=\Delta_h T_{22}$, and $T^1_2=0$, with entities $T_{21}$ and $T_{22}$ as defined in the proof of Lemma~\ref{lem:central1}.
    The estimates for $T_{21}$ and $\Delta_h T_{22}$ have been carried out there.
    In view of $h\lesssim \sin\eta$ we can further estimate
    \[
    \|T^0_2\|_\infty=\|T_{22}\|_\infty\lesssim h|\sin\eta|^{-2} \lesssim h^0|\sin\eta|^{-1}.
    \]
    This proves the case $m=0$.

    We proceed with the induction and assume that the lemma is true for $T$, where $m\in\N_0$ is fixed but arbitrary.
    The associated decomposition of length $m+1$ shall be denoted by $(T^k_1,T^k_2)_k$.
    We will show that under this hypothesis also the function $\widetilde{T}(t):=a^\prime(t)\Delta_h \widetilde{G}_+(t,a(t))$,
    where $\widetilde{G}_+(x)=r(x)^{m+1}G(x)$ for $x\in\R^2$, admits a decomposition $(\tilde{T}^k_1,\tilde{T}^k_2)_k$
    of the form $(\ast)$ of length $(m+2)$ with the desired properties.

    First we can decompose as follows,
    \begin{align*}
    \widetilde{T}(t)&=a^\prime(t) \Delta_h\big( r(t,a(t)) \widetilde{G}(t,a(t))\big)  \\
    &= a^\prime(t) \Delta_hr(t,a(t)) \widetilde{G}(t+h, a(t+h)) + r(t,a(t)) a^\prime(t)\Delta_h\widetilde{G}(t, a(t)) \\
    &=  \big[ r(t,a(t)) T^0_1(t) \big] + \big[ a^\prime(t) \Delta_hr(t,a(t)) \widetilde{G}(t+h, a(t+h)) + r(t,a(t)) T^0_2(t)\big] \\
    &=: \tilde{T}^0_1(t)+ \tilde{T}^0_2(t).
    \end{align*}
    In view of the properties of $T^0_1$ and Lemma~\ref{lemapp:r} we see that
    the function $ \tilde{T}^0_1$ satisfies the assertion. The estimate
     \[
     \| \tilde{T}^0_2 \|_\infty \lesssim \|a^\prime\|_\infty  \sup_{t\in\R}|\Delta_h r(t,a(t))| \sup_{t\in\R}|\widetilde{G}(t,a(t))| \lesssim  |\sin\eta|^{-1} \cdot h \cdot \delta_j^{m},
     \]
     where Lemmas~\ref{lem:propa}, \ref{lemapp:r} and \ref{lemapp:edgeG} were used, shows the claim also for $\tilde{T}^0_2 $.

    We take another forward difference of the component $\tilde{T}^0_2$ and obtain
    \beqn
    \Delta_h \tilde{T}^0_2(t) &=&  \Delta_h a^\prime(t) \Delta_hr(t+h,a(t+h)) \widetilde{G}(t+2h, a(t+2h))   +  a^\prime(t) \Delta^2_hr(t,a(t)) \widetilde{G}(t+2h, a(t+2h)) \\
    && + \Delta_hr(t,a(t)) a^\prime(t) \Delta_h\widetilde{G}(t+h, a(t+h))
    + \Delta_hr(t,a(t)) T^0_2(t+h) + r(t,a(t)) \Delta_hT^0_2(t) \\
    &=&  \Delta_h a^\prime(t) \Delta_hr(t+h,a(t+h)) \widetilde{G}(t+2h, a(t+2h))   +  a^\prime(t) \Delta^2_hr(t,a(t)) \widetilde{G}(t+2h, a(t+2h)) \\
    && + \Delta_hr(t,a(t))( T^0_1(t+h) + T^0_2(t+h) ) - \Delta_hr(t,a(t))\Delta_ha^\prime(t)\Delta_h\widetilde{G}(t+h, a(t+h)) \\
    && + \Delta_hr(t,a(t)) T^0_2(t+h) + r(t,a(t)) T^1_1(t) + r(t,a(t)) T^1_2(t) \\
    &=& \big[ \Delta_h a^\prime(t) \Delta_hr(t+h,a(t+h)) \widetilde{G}(t+2h, a(t+2h))   +  a^\prime(t) \Delta^2_hr(t,a(t)) \widetilde{G}(t+2h, a(t+2h))\\
    && - \Delta_hr(t,a(t))\Delta_ha^\prime(t)\Delta_h\widetilde{G}(t+h, a(t+h)) + r(t,a(t)) T^1_1(t)  +  \Delta_hr(t,a(t)) T^0_1(t+h) \big] \\
    && + \big[ 2\Delta_hr(t,a(t)) T^0_2(t+h)  + r(t,a(t)) T^1_2(t) \big] \\
    &=:&\tilde{T}^1_1(t)+ \tilde{T}^1_2(t).
    \eeqn
    For $\tilde{T}^1_1$ we check directly
    \begin{align*}
    &\sup_{t\in\R}|\Delta_h a^\prime(t) \Delta_hr(t,a(t)) \widetilde{G}(t+h, a(t+h))| \lesssim h^{\beta-1}\delta_j |\sin\eta|^{-1-\beta} h h^{m}
    = h^{m+1} h^{\beta} |\sin\eta|^{-1-\beta}, \\
    &\sup_{t\in\R}|a^\prime(t) \Delta^2_hr(t,a(t)) \widetilde{G}(t+2h, a(t+2h))| \lesssim |\sin\eta|^{-1}   h^{\beta}\delta_j |\sin\eta|^{-\beta}   h^{m}
    = h^{m+1} h^{\beta} |\sin\eta|^{-1-\beta}, \\
    &\sup_{t\in\R}| \Delta_hr(t,a(t))\Delta_ha^\prime(t)\Delta_h\widetilde{G}(t+h, a(t+h))| \lesssim \delta_j^m h^\beta h |\sin\eta|^{-1-\beta} \lesssim h^{m+1} h^{\beta} |\sin\eta|^{-1-\beta}.
    \end{align*}
    The estimates for the remaining two terms are obvious. Hence $\tilde{T}^1_1$ fulfills the desired properties.

    For $\tilde{T}^1_2$ we use the induction hypothesis and Lemma~\ref{lemapp:r} to obtain
    \[
    \|\tilde{T}^1_2\|_\infty \lesssim \sup_{t\in\R} |r(t,a(t))T^{1}_2(t)| + \sup_{t\in\R}|\Delta_hr(t,a(t))T^{0}_2(t+h)| \lesssim h^{m+1}|\sin\eta|^{-1}.
    \]
    Moving forward, this procedure yields terms for $k=1,\ldots,m+1$,
    \beqn
    \tilde{T}^{k+1}_1(t)&=& r(t,a(t))T^{k+1}_1(t) + (k+1)\Delta_hr(t+h,a(t+h))T^{k}_1(t+h)  \\
    && +  (k+1)\Delta^2_hr(t,a(t))T^{k-1}_2(t+h) +  (k+1)\Delta^2_hr(t,a(t))T^{k}_2(t+h),   \\
    \tilde{T}^k_2(t)&=& r(t,a(t))T^{k}_2(t) + (k+1)\Delta_hr(t,a(t))T^{k-1}_2(t+h),
    \eeqn
    which satisfy the desired estimates. Here we put $T_1^{m+2}=T_2^{m+2}=0$ for convenience. Indeed, using the induction assumptions, we obtain
    \beqn
    \|\tilde{T}^{k+1}_1\|_\infty &\lesssim&  \sup_{t\in\R}|r(t,a(t))T^{k+1}_1(t)| +  \sup_{t\in\R}|\Delta_hr(t+h,a(t+h))T^{k}_1(t+h)|  \\
    &&+  \sup_{t\in\R}|\Delta^2_hr(t,a(t))T^{k-1}_2(t+h)| +  \sup_{t\in\R}|\Delta^2_hr(t,a(t))T^{k}_2(t+h)| \lesssim h^{m+1}h^{\beta} |\sin\eta|^{-1-\beta}     , \\
    \|\tilde{T}^k_2\|_\infty &\lesssim&   \sup_{t\in\R}|r(t,a(t))T^{k}_2(t)| +  \sup_{t\in\R}|\Delta_hr(t,a(t))T^{k-1}_2(t+h)| \lesssim h^{m+1}|\sin\eta|^{-1}.
    \eeqn
    Note, that $T_2^{m+1}=T_1^{m+2}=T_2^{m+2}=0$. Hence, for $k=m+1$ these expressions read
    \begin{align*}
    \tilde{T}^{m+2}_1(t)&= (m+2)\Delta_hr(t+h,a(t+h))T^{m+1}_1(t+h) +  (m+2)\Delta^2_hr(t,a(t))T^{m}_2(t+h),   \\
    \tilde{T}^{m+1}_2(t)&= (m+2)\Delta_hr(t,a(t))T^{m}_2(t+h).
    \end{align*}
    Since $\Delta_h \tilde{T}^{m+1}_2=\tilde{T}^{m+2}_1$ we have $\tilde{T}^{m+2}_2=0$ and the proof is finished.
\end{proof}

The following Lemma~\ref{lemapp:main2} is in the same spirit as Lemma~\ref{lemapp:main1}.

\begin{lemma}\label{lemapp:main2}
Let $\widetilde{G}(x)=r(x)^m G(x)$ for $x\in\R^2$, $m\in\N_0$, and $h\asymp2^{-j(1-\alpha)}$.
Then the function $S:\R\rightarrow\R$ defined by $S(t)=a^\prime(t) \Delta_h\partial_1\widetilde{G}(t,a(t)) $ 
admits a decomposition $(S^k_1,S^k_2)_k$ 
of the form $(\ast)$ of length $m+1$ with estimates
\begin{align*}
    \|S^{k}_1\|_\infty&\lesssim h^{m-1}h^{\beta} |\sin\eta|^{-1-\beta}, \quad k=0,\ldots,m+1, \\
    \|S^{k}_2\|_\infty&\lesssim h^{m-1}|\sin\eta|^{-1}, \quad k=0,\ldots,m.
\end{align*}
Moreover, these functions can be chosen such that $\supp S^{k}_i \subset\widetilde{I}(\eta)$.
\end{lemma}
\begin{proof}
The proof is by induction on $m$. The assumptions are clearly true for $m=0$. 

For the induction we let $m\in\N_0$ be fixed and let $S$ be the function defined in the setting. Further, let us assume that we have a decomposition
$(S^k_1,S^k_2)_k$ of length $m+1$ with the desired properties for $S$.
We put $S_2^{m+1}=0$ and for convenience we also define $S_1^{m+2}=S_2^{m+2}=0$. 
We will show that under these assumptions the function $\widetilde{S}:\R\rightarrow\R$ given by $\widetilde{S}(t):=a^\prime(t)\Delta_h\partial_1\widetilde{G}_+(t,a(t))$,
where $\widetilde{G}_+(x)=r(x)^{m+1} G(x)$ for $x\in\R^2$,
admits a decomposition $(\tilde{S}^k_1,\tilde{S}^k_2)_k$ of length $m+2$ of the same form.
First we calculate
\begin{align*}
\widetilde{S}(t)&=a^\prime(t) \Delta_h\partial_1\widetilde{G}_+(t,a(t))  =a^\prime(t)\Delta_h\big( \cos\eta \widetilde{G}(t,a(t)) + r(t,a(t)) \partial_1\widetilde{G}(t,a(t)) \big) \\
&= a^\prime(t)\cos\eta\Delta_h \widetilde{G}(t,a(t)) + a^\prime(t) \Delta_hr(t,a(t)) \partial_1\widetilde{G}(t,a(t)) + r(t+h,a(t+h))a^\prime(t)\Delta_h\partial_1\widetilde{G}(t,a(t)).
\end{align*}
%
Using the induction hypothesis we can proceed,
\beqn
\widetilde{S}(t)&=& \big[ r(t+h,a(t+h)) S^0_1(t) \big] \\
&& + \big[ r(t+h,a(t+h)) S^0_2(t) + a^\prime(t) \cos\eta \Delta_h \widetilde{G}(t,a(t)) + a^\prime(t) \Delta_hr(t,a(t)) \partial_1\widetilde{G}(t,a(t))  \big] \\
&=:&\tilde{S}^0_1(t) + \tilde{S}^0_2(t).
\eeqn
The terms $\tilde{S}^0_1$ and $\tilde{S}^0_2$ have the desired properties, which follows from the estimates
\begin{align*}
\sup_{t\in\R}|r(t+h,a(t+h))S^0_1(t)| &\lesssim h h^{m-1} h^{\beta} |\sin \eta|^{-1-\beta}, \\
\sup_{t\in\R}|r(t+h,a(t+h))S^0_2(t)| &\lesssim h h^{m-1} |\sin \eta|^{-1}, \\
\sup_{t\in\R}|a^\prime(t) \cos\eta \Delta_h\widetilde{G}(t,a(t))| &\lesssim |\sin \eta|^{-1} \delta_j^{m-1} h, \\
\sup_{t\in\R}|a^\prime(t) \Delta_hr(t,a(t)) \partial_1\widetilde{G}(t,a(t))| &\lesssim |\sin \eta|^{-1} h \delta_j^{m-1}.
\end{align*}
%
%
Taking another forward difference of $\tilde{S}^0_2$ yields
\beqn
\Delta_h \tilde{S}^0_2(t)&=&\Delta_hr(t+h,a(t+h))S^0_2(t) + r(t+2h,a(t+2h))\Delta_hS^0_2(t) +\Delta_ha^\prime(t)\cos\eta \Delta_h \widetilde{G}(t,a(t)) \\
&& + a^\prime(t+h) \cos\eta \Delta_h^2 \widetilde{G}(t,a(t)) + \Delta_ha^\prime(t) \Delta_h r(t,a(t)) \partial_1\widetilde{G}(t,a(t)) \\
&& + a^\prime(t+h)\Delta_h^2r(t,a(t)) \partial_1\widetilde{G}(t,a(t)) + a^\prime(t+h)\Delta_hr(t+h,a(t+h))\Delta_h\partial_1\widetilde{G}(t,a(t)).
\eeqn
Let $T$ denote the function from Lemma~\ref{lemapp:main1}. We observe,
\begin{align*}
a^\prime(t+h) \Delta_h^2 \widetilde{G}(t,a(t)) &= a^\prime(t+h) \big( \Delta_h \widetilde{G}(t+h,a(t+h)) - \Delta_h \widetilde{G}(t,a(t))  \big) \\
&= a^\prime(t+h) \Delta_h\widetilde{G}(t+h,a(t+h)) -a^\prime(t) \Delta_h\widetilde{G}(t,a(t)) + (a^\prime(t) - a^\prime(t+h)) \Delta_h\widetilde{G}(t,a(t))  \\
&= a^\prime(t+h) \Delta_h\widetilde{G}(t+h,a(t+h)) -a^\prime(t) \Delta_h\widetilde{G}(t,a(t)) - \Delta_h a^\prime(t) \Delta_h\widetilde{G}(t,a(t)) \\
&= T(t+h) - T(t) -  \Delta_h a^\prime(t) \Delta_h\widetilde{G}(t,a(t)) = \Delta_hT(t)- \Delta_h a^\prime(t) \Delta_h\widetilde{G}(t,a(t)).
\end{align*}
Now we know by Lemma~\ref{lemapp:main1} that there is a decomposition $(T^k_1,T^k_2)_k$ of $T$ of length $m+1$ with
the specific properties given there. This allows to decompose $\Delta_hT=\Delta_hT^0_1+T^1_1+T^1_2$ and we obtain
\[
a^\prime(t+h) \Delta_h^2 \widetilde{G}(t,a(t)) = \Delta_hT^0_1(t)+T^1_1(t)+T^1_2(t)- \Delta_h a^\prime(t) \Delta_h\widetilde{G}(t,a(t)).
\]
Using this observation we obtain
\beqn
\Delta_h \tilde{S}^0_2(t)
&=& \Delta_hr(t+h,a(t+h))S^0_2(t)+ r(t+2h,a(t+2h))S^1_1(t)+ r(t+2h,a(t+2h))S^1_2(t) \\
&& + \Delta_ha^\prime(t)\cos\eta \Delta_h \widetilde{G}(t,a(t))
 +\cos\eta \Delta_hT^0_1(t) + \cos\eta(T^1_1(t)+T^1_2(t))  -  \cos\eta\Delta_h a^\prime(t) \Delta_h\widetilde{G}(t,a(t)) \\
&& + \Delta_ha^\prime(t) \Delta_h r(t,a(t)) \partial_1\widetilde{G}(t,a(t))  + a^\prime(t+h)\Delta_h^2r(t,a(t)) \partial_1\widetilde{G}(t,a(t))  \\
&& + a^\prime(t)\Delta_hr(t+h,a(t+h))\Delta_h\partial_1\widetilde{G}(t,a(t))
 + (a^\prime(t+h)-a^\prime(t))\Delta_hr(t+h,a(t+h))\Delta_h\partial_1\widetilde{G}(t,a(t))
\eeqn
and further
\beqn
\Delta_h \tilde{S}^0_2(t)
&=& \big[ r(t+2h,a(t+2h))S^1_1(t) + \Delta_ha^\prime(t)\cos\eta \Delta_h \widetilde{G}(t,a(t)) \\
&& + \cos\eta \Delta_hT^0_1(t) - \cos\eta\Delta_h a^\prime(t) \Delta_h\widetilde{G}(t,a(t)) + \cos\eta T^1_1(t) \\
&& + \Delta_ha^\prime(t) \Delta_h r(t,a(t)) \partial_1\widetilde{G}(t,a(t)) + a^\prime(t+h)\Delta_h^2r(t,a(t)) \partial_1\widetilde{G}(t,a(t)) \\
&& + \Delta_ha^\prime(t)\Delta_hr(t+h,a(t+h))\Delta_h\partial_1\widetilde{G}(t,a(t)) + \Delta_hr(t+h,a(t+h))S^0_1(t) \big] \\
&& + \big[  r(t+2h,a(t+2h))S^1_2(t) +  \cos\eta T^1_2(t) + 2\Delta_hr(t+h,a(t+h))S^0_2(t) \big] \\
&=:& \tilde{S}^1_1(t)+\tilde{S}^1_2(t).
\eeqn
Now we can split $\Delta_h \tilde{S}^0_2=\tilde{S}^1_1+\tilde{S}^1_2$ with
\beqn
\tilde{S}^1_1(t)&=& r(t+2h,a(t+2h))S^1_1(t) +\cos\eta\Delta_ha^\prime(t)\Delta_h\widetilde{G}(t,a(t))  + \Delta_ha^\prime(t) \Delta_hr(t,a(t)) \partial_1\widetilde{G}(t,a(t)) \\
&& + a^\prime(t+h)\Delta_h^2r(t,a(t)) \partial_1\widetilde{G}(t,a(t))
   + \Delta_hr(t+h,a(t+h)) S^0_1(t) + \cos\eta\Delta_hT^0_1(t) \\
&&+ \Delta_ha^\prime(t)\Delta_hr(t+h,a(t+h))\Delta_h\partial_1\widetilde{G}(t,a(t))
  - \cos\eta\Delta_h a^\prime(t) \Delta_h\widetilde{G}(t,a(t)) + \cos\eta T^1_1(t) , \\
\tilde{S}^1_2(t)&=&2\Delta_hr(t+h,a(t+h))S^0_2(t) + r(t+2h,a(t+2h))S^1_2(t) +   \cos\eta T^1_2(t).
\eeqn
These terms have the desired properties. To see this, we calculate
\begin{align*}
\sup_{t\in\R}|r(t+2h,a(t+2h))S^1_1(t)| &\lesssim h  h^{m-1} h^{\beta} |\sin \eta|^{-1-\beta},   \\
\sup_{t\in\R}|\Delta_ha^\prime(t)\Delta_h\widetilde{G}(t,a(t))|&\lesssim  \delta_j h^{\beta-1} |\sin\eta|^{-1-\beta} \cdot \delta_j^{m-1}h, \\
\sup_{t\in\R}|\Delta_ha^\prime(t) \Delta_h r(t,a(t)) \partial_1\widetilde{G}(t,a(t))| &\lesssim \delta_j h^{\beta-1}|\sin\eta|^{-1-\beta} \cdot h \cdot \delta_j^{m-1},  \\
\sup_{t\in\R}|a^\prime(t+h)\Delta_h^2r(t,a(t)) \partial_1\widetilde{G}(t,a(t))| &\lesssim |\sin\eta|^{-1} \cdot  h^{\beta} \delta_j |\sin\eta|^{-\beta}  \cdot   \delta_j^{m-1}, \\
\sup_{t\in\R}|\Delta_hr(t+h,a(t+h)) S^0_1(t)| &\lesssim h  h^{m-1} h^{\beta} |\sin \eta|^{-1-\beta}, \\
\sup_{t\in\R}|\Delta_hT^0_1(t)| &\lesssim  h^{m} h^{\beta} |\sin \eta|^{-1-\beta},  \\
\sup_{t\in\R}|\Delta_h a^\prime(t) \Delta_h\widetilde{G}(t,a(t))| &\lesssim \delta_j h^{\beta-1} |\sin\eta|^{-1-\beta} \cdot \delta_j^{m-1}h, \\
\sup_{t\in\R}|\Delta_ha^\prime(t)\Delta_hr(t+h,a(t+h))\Delta_h\partial_1\widetilde{G}(t,a(t)) |&\lesssim \delta_j h^{\beta-1} |\sin\eta|^{-1-\beta} \cdot h \cdot \delta_j^{m-1}, \\
\sup_{t\in\R}|T^1_1(t)|&\lesssim  \delta_j^m h^{\beta} |\sin\eta|^{-1-\beta},
\end{align*}
and
\begin{align*}
\sup_{t\in\R}|2\Delta_hr(t+h,a(t+h))S^0_2(t)| &\lesssim h h^{m-1} |\sin\eta|^{-1}, \\
\sup_{t\in\R}|r(t+2h,a(t+2h))S^1_2(t)| &\lesssim  h h^{m-1} |\sin\eta|^{-1},   \\
\sup_{t\in\R}|T^1_2(t)| &\lesssim  h h^{m-1} |\sin\eta|^{-1}.
\end{align*}
We proceed with
\beqn
\Delta_h\tilde{S}^1_2(t)&=&\big[ \cos\eta T^2_1(t)+  r(t+3h,a(t+3h))S^2_1(t) + 2\Delta_hr(t+2h,a(t+2h))S^1_1(t) \\
&&+ 2\Delta_h^2r(t+h,a(t+h)) S^0_2(t) \big] + \big[ \cos\eta T^2_2(t) + r(t+3h,a(t+3h))S^2_2(t) \\
&&+ 3\Delta_hr(t+2h,a(t+2h))S^1_2(t)  \big] =: \tilde{S}^2_1(t) + \tilde{S}^2_2(t).
\eeqn
Inductively, we put for $k=1,\ldots,m+1$, where for convenience $T^{m+2}_1=0$,
\beqn
\tilde{S}^{k+1}_1(t)&:=& \cos\eta T^{k+1}_1(t) +  r(t+(k+2)h,a(t+(k+2)h))S^{k+1}_1(t) \\
&& + (k+1)\Delta_hr(t+(k+1)h,a(t+(k+1)h))S^k_1(t) + (k+1)\Delta_h^2r(t+kh,a(t+kh)) S^{k-1}_2(t),  \\
\tilde{S}^k_2(t)&:=& \cos\eta T^{k}_2(t) + r(t+(k+1)h,a(t+(k+1)h))S^{k}_2(t) + (k+1)\Delta_hr(t+kh,a(t+kh))S^{k-1}_2(t).
\eeqn
These terms clearly satisfy $\Delta_h \tilde{S}^k_2 = \tilde{S}^{k+1}_1 + \tilde{S}^{k+1}_2$. They also have the desired properties since
\begin{align*}
\sup_{t\in\R}|T_1^{k+1}(t)| &\lesssim  h^{m} h^{\beta}   |\sin\eta|^{-1-\beta}, \\
\sup_{t\in\R}| r(t+(k+2)h,a(t+(k+2)h))S^{k+1}_1(t)| &\lesssim  h\cdot h^{m-1} h^{\beta}   |\sin\eta|^{-1-\beta}, \\
\sup_{t\in\R}|\Delta_hr(t+(k+1)h,a(t+(k+1)h))S^k_1(t)| &\lesssim  h\cdot h^{m-1} h^{\beta}   |\sin\eta|^{-1-\beta}, \\
\sup_{t\in\R}|\Delta_h^2r(t+kh,a(t+kh)) S^{k-1}_2(t)| &\lesssim h^{\beta} \delta_j  |\sin\eta|^{-\beta}  \cdot h^{m-1} |\sin\eta|^{-1},
\end{align*}
and
\begin{align*}
\sup_{t\in\R}|T^{k}_2(t)| &\lesssim h^{m} |\sin\eta|^{-1},\\
\sup_{t\in\R}|r(t+(k+1)h,a(t+(k+1)h))S^{k}_2(t)| &\lesssim h \cdot h^{m-1} |\sin\eta|^{-1},\\
\sup_{t\in\R}|\Delta_hr(t+kh,a(t+kh))S^{k-1}_2(t)| &\lesssim h \cdot h^{m-1} |\sin\eta|^{-1}.
\end{align*}

Since $S_2^{m+1}=S_1^{m+2}=S_2^{m+2}=T_2^{m+1}=T_1^{m+2}=0$, for $k=m+1$ these expressions read
\beqn
    \tilde{S}^{m+2}_1(t)&=& (m+2)\Delta_hr(t+(m+2)h,a(t+(m+2)h)) S^{m+1}_1(t)  \\
                        &&+(m+2)\Delta^2_hr(t+(m+1)h,a(t+(m+1)h)) S^{m}_2(t), \\
    \tilde{S}^{m+1}_2(t)&=& (m+2)\Delta_hr(t+(m+1)h,a(t+(m+1)h)) S^{m}_2(t).
\eeqn
We see that $\Delta_h\tilde{S}^{m+1}_2=\tilde{S}^{m+2}_1$. Therefore $\tilde{S}^{m+2}_2=0$ and the proof is finished.

\end{proof}

A slight modification of the previous proof leads to the following lemma.

\begin{lemma}\label{lemapp:main3}
Let $\widetilde{G}(x)=r(x)^m G(x)$ for $x\in\R^2$ and $m\in\N_0$ and $h\asymp 2^{-j(1-\alpha)}$.
The function $\widetilde{S}:\R^2\rightarrow\R$ given by
$\widetilde{S}(t,\tau)=a^\prime(\tau) \Delta_h\partial_1\widetilde{G}(t,a(\tau)) $ for $(t,\tau)\in\R^2$ 
admits a decomposition $(\tilde{S}^k_1,\tilde{S}^k_2)_k$ 
of the form $(\ast)$ of length $m+1$ with estimates
\begin{align*}
\sup_{t\in\R}\sup_{\tau\in[t-h,t+h]} |\tilde{S}^{k}_1(t,\tau)|&\lesssim h^{m-1}h^{\beta} |\sin\eta|^{-1-\beta}, \quad k=0,\ldots,m+1, \\
\sup_{t\in\R}\sup_{\tau\in[t-h,t+h]} |\tilde{S}^{k}_2(t,\tau)|&\lesssim h^{m-1}|\sin\eta|^{-1}, \quad k=0,\ldots,m.
\end{align*}
\end{lemma}
\begin{proof}
A small adaption of the previous proof is required to account for the little deviation of $\tau$ from $t$.
We just make the following remark. For $t,\,\tau\in\R$ we have
$
r(t,a(\tau))=r(t,a(t))+(a(t)-a(\tau))\sin\eta.
$
It follows for $h\in\R$
\[
\sup_{\tau\in[t-h,t+h]}|r(t,a(\tau))| \le |r(t,a(t))| + |h\sin\eta|\|a^\prime\|_\infty  \lesssim  |r(t,a(t))| + |h|.
\]
Since $h\asymp 2^{-j(1-\alpha)}$ this additional term poses no problem in the estimations.
\end{proof}

\end{appendix}

%
%
%
%
%
%

\end{document}